\documentclass[10pt,reqno]{amsart}
\usepackage[margin=1in]{geometry}
\usepackage{amssymb,mathrsfs,graphicx,amsmath}
\usepackage{ifthen}
\usepackage{graphicx,extpfeil}
\usepackage{colortbl}
\usepackage{colortbl}
\definecolor{black}{rgb}{0.0, 0.0, 0.0}
\definecolor{red}{rgb}{1.0, 0.5, 0.5}
\provideboolean{shownotes} 
\setboolean{shownotes}{true} 
\newcommand{\margnote}[1]{
\ifthenelse{\boolean{shownotes}}%
{\marginpar{\raggedright\tiny\texttt{#1}}}%
{}%
}
\newcommand{\hole}[1]{
\ifthenelse{\boolean{shownotes}}%
{\begin{center} \fbox{ \rule {.25cm}{0cm} \rule[-.1cm]{0cm}{.4cm}
\parbox{.85\textwidth}{\begin{center} \texttt{#1}\end{center}} \rule
{.25cm}{0cm}}\end{center}} {} }

\title[On regular solutions and singularity formation for Vlasov/Navier--Stokes equations]{On regular solutions and singularity formation for Vlasov/Navier--Stokes equations with degenerate viscosities and vacuum}

\author[Choi]{Young-Pil Choi}
\address[Young-Pil Choi]{\newline Department of Mathematics \newline Yonsei University, Seoul 03722, Republic of Korea}
\email{ypchoi@yonsei.ac.kr}

\author[Jung]{Jinwook Jung}
\address[Jinwook Jung]{\newline Research Institute of Basic Sciences \newline Seoul National University, Seoul  08826, Republic of Korea}
\email{warp100@snu.ac.kr}

\numberwithin{equation}{section}

\newtheorem{theorem}{Theorem}[section]
\newtheorem{lemma}{Lemma}[section]

\newtheorem{proposition}{Proposition}[section]
\newtheorem{remark}{Remark}[section]
\newtheorem{definition}{Definition}[section]

\newcommand{\nc}{\newcommand}

\newcommand{\R}{\mathbb R}
\newcommand{\om}{\Omega}
\newcommand{\ls}{\lesssim}

\newcommand{\T}{\mathbb T}
\newcommand{\N}{\mathbb N}
\newcommand{\mbs}{\mathbb S}

\newcommand{\bu}{{\bf u}}
\newcommand{\bq}{\begin{equation}}
\newcommand{\eq}{\end{equation}}

\newcommand{\lt}{\left}
\newcommand{\rt}{\right}
\newcommand{\mc}{\mathcal{C}}

\newcommand{\pa}{\partial}

\newcommand{\ml}{\mathcal{L}}
\newcommand{\intr}{\int_{\R^3}}
\newcommand{\intrr}{\iint_{\R^3 \times \R^3}}
\newcommand{\bn}{{\bf n}}
\nc{\mh}{\mathcal H}
\nc{\md}{\mathcal D}

\newcommand{\sfI}{\mathsf{I}}
\newcommand{\sfJ}{\mathsf{J}}
\newcommand{\sfK}{\mathsf{K}}
\newcommand{\sfL}{\mathsf{L}}
\newcommand{\sfM}{\mathsf{M}}

\begin{document}
\allowdisplaybreaks

\date{\today}


\begin{abstract} We analyze the Vlasov equation coupled with the compressible Navier--Stokes equations with degenerate viscosities and vacuum. These two equations are coupled through the drag force which depends on the fluid density and the relative velocity between particle and fluid. We first establish the existence and uniqueness of local-in-time regular solutions with arbitrarily large initial data and a vacuum. We then present sufficient conditions on the initial data leading to the finite-time blowup of regular solutions. In particular, our study makes the result on the finite-time singularity formation for Vlasov/Navier--Stokes equations discussed by Choi [J. Math. Pures Appl., 108, (2017), 991--1021] completely rigorous. 
\end{abstract}

\maketitle \centerline{\date}

\tableofcontents

%
%
%
%
\section{Introduction}
In the current work, we are interested in the existence and uniqueness of local-in-time regular solutions and the finite-time singularity formation for the Vlasov/Navier--Stokes equations with degenerate viscosities and vacuum. Let $f=f(x,\xi,t)$ be the one-particle distribution function at $(x,\xi) \in \R^3 \times \R^3$ and at time $t$, and let $\rho = \rho(x,t)$ and $u = u(x,t)$ be the fluid density and velocity, respectively. Then, our main system, Vlasov/Navier--Stokes system, is given by
\begin{align}\label{main_sys}
\begin{aligned}
&\pa_t f + \xi \cdot \nabla_x f + \nabla_\xi \cdot (D(\rho,u-\xi)f) = 0, \quad (x,\xi,t) \in \R^3 \times \R^3 \times \R_+,\cr
&\pa_t \rho + \nabla_x \cdot (\rho u) = 0, \quad (x,t) \in \R^3 \times \R_+,\cr
&\pa_t (\rho u) + \nabla_x \cdot (\rho u \otimes u) + \nabla_x p(\rho) - \nabla_x \cdot \mbs (\nabla_x u) = - \int_{\R^3} D(\rho,u-\xi) f\,d\xi,
\end{aligned}
\end{align}
with the drag force operator $D:= D(\rho, u-\xi )$ and
\[
\mbs(\nabla_x u) := 2\mu(\rho) \T(u) + \lambda(\rho) (\nabla_x \cdot u) \mathbb{I}_3,
\]
together with the initial data
\[
\lt(f(x,\xi,0), \rho(x,0), u(x,0)\rt) =: (f_0(x,\xi), \rho_0(x), u_0(x)), \quad (x,\xi) \in \R^3 \times \R^3,
\]
and far-field behavior:
\[
f(x,\xi,t) \to 0, \quad (\rho, u) \to (\rho_\infty, 0), \quad \rho_\infty\ge0,
\]
sufficiently fast as $|x|$, $|\xi| \to \infty$. Here $\mu = \mu(\rho)$ and $\lambda = \lambda(\rho)$ are the two Lam\'e viscosity coefficients depending on the fluid density $\rho$ which satisfy
\[
\mu > 0 \quad \mbox{and} \quad 2\mu + 3\lambda \geq 0,
\]
and $\mathbb{I}_3$ denotes the $3 \times 3$ real identity matrix. Throughout this paper, we choose 
\[
\mu(\rho) = \alpha\rho^\delta \quad \mbox{and} \quad \lambda(\rho) = \beta\rho^\delta,
\] 
with $\delta>1$, $\alpha>0$ and $2\alpha+ 3\beta \geq 0$. Note that in \cite{BD07} a new entropy estimate, named BD entropy after Bresch and Desjardins, is proposed and the assumption $\lambda(\rho) = 2\rho \mu'(\rho) - 2\mu(\rho)$ on the relation between viscosity coefficients is used to have more regularity on $\rho$. In our case, this relation gives $\beta = 2\alpha (\delta - 1) > 0$, thus $2\alpha + 3\beta = 2\alpha + 6\alpha(\delta - 1)>0$. The drag force operator $D$, the pressure law $p$ and the strain tensor $\T$ are given by
\[
D(\rho,u-\xi) = \rho^m(u-\xi), \quad p(\rho) = \rho^\gamma \quad \mbox{with} \quad m, \gamma > 1 \quad \mbox{and} \quad \T(u) := \frac12 \lt( \nabla_x u + (\nabla_x u)^T\rt).
\]
Here we would like to address that the drag force depends on the fluid density which is more physically relevant, but more difficulty in analysis. 

These types of kinetic-fluid models are widely used in describing the dynamics of small particles, dealt with the statistical point of view (mesoscopic level), immersed in a fluid, which is usually considered from the hydrodynamic point of view (macroscopic level). We refer to \cite{Des10, Rou81, Wil58} and references therein for general discussions on the multiphase flows. Here, as stated above, we describe the particle dynamics by employing the particle distribution function $f$ which is governed by the Vlasov equation, and the compressible Navier--Stokes equations with degenerate viscosities are used for the fluid. These two equations are coupled through the drag force $D$, which is also often called the friction force. We are interested in the {\it thin sprays regime}, thus the other effects on inter-particle interactions, that would be considered in the {\it thick sprays regime}, for instance, collisions, break-up, coalescence phenomena, are neglected. 

Over the past two decades, mathematical analysis on the Vlasov/Navier--Stokes system has been extensively investigated, and some previous works can be briefly summarized as follows. The global existence of weak solutions is studied in different spatial domains; periodic domain \cite{BDGM09}, bounded domain with reflection boundary conditions \cite{Yu13}, time-dependent domain with absorption boundary condition \cite{BGM17}. Other inter-particle interactions are also considered; BGK collision operator \cite{CY20}, particle breakup operator \cite{YY18}, and global weak solutions are obtained. In the case with diffusion, i.e., Vlasov--Fokker--Planck/Navier--Stokes system, the global existence of weak solutions is established in \cite{CKL11, CJacc, MV07}. The existence theory of strong/classical solutions is also found in \cite{CDM11, CKL13, CK15, CLY21, LMW17}. The large-time behavior of solutions is also discussed in \cite{C16, CK15, Hpre, HMM20}. It shows the velocity-alignment behavior of solutions, in particular, the particle distribution function converge towards the monokinetic ansatz, i.e., roughly speaking, $f\,dxd\xi \simeq \rho_f dx \otimes \delta_{u}(v)\,dv$, where $\rho_f = \int f\,dv$ for sufficiently large $t$. The asymptotic analysis of Vlasov or Vlasov--Fokker--Planck equation coupled with Navier--Stokes equations is also studied in \cite{CCK16, CG06, CJpre,CJacc, CJ20, GJV04, GJV04_2, HMpre, MV08}, based on the relative entropy arguments, to derive two-phase fluid models. In \cite{C17}, the finite-time singularity formation for the Vlasov/Navier--Stokes system is first discussed. At the formal level, sufficient conditions on the initial data leading to a finite-time breakdown of smoothness of solutions are analyzed. Our main purpose is to make the work \cite{C17} completely rigorous. It requires the well-posedness for the system \eqref{main_sys} and investigation of the finite-time blowup of solutions. 

Before presenting our main results, we introduce several notations used throughout the paper. For a function $f(x,\xi)$, $\|f\|_{L^p}$ denotes the usual $L^p(\R^3 \times \R^3)$-norm, and if $g$ is a function of $x$, the usual $L^p(\R^3)$-norm is denoted by $\|g\|_{L^p}$, unless otherwise specified. For $p \in [1,\infty)$, we denote by $L^{2,p}_\nu := L^2_\nu(\R^3 \times \R^3)$ the space of measurable functions which are weighted by $\nu_p(x,\xi) :=(1 + |x|^2 + |\xi|^2)^{p/2}e^{|\xi|^2}$ and equipped with the norm
\[
\|f\|_{L^{2,p}_\nu} := \lt(\intrr \nu_p(x,\xi) f^2\,dxd\xi\rt)^{1/2}.
\]
For any $k \in \N$, $H^{k,p}_\nu := H^{k,p}_\nu(\R^3 \times \R^3)$ stands for $L^{2,p}_\nu$ Sobolev space of $k$-th order, i.e.
\[
\|f\|_{H^{k,p}_\nu} := \lt(\sum_{|\alpha| \leq k}\intrr \nu_p(x,\xi) |\nabla_{(x,\xi)}^\alpha f|^2\,dxd\xi\rt)^{1/2}.
\]
$f \ls g$ represents that there exists a positive constant $C>0$ such that $f \leq C g$. We also denote by $C$ a generic positive constant. For simplicity, we often drop $x$-dependence of differential operators $\pa_{x_i}, i=1,2,3$, $\nabla_x$, and $\Delta_x$, that is, $\pa_i f:= \pa_{x_i} f$, $\nabla f := \nabla_x f$, and $\Delta f := \Delta_x f$. For any nonnegative integer $s$, $H^s$ denotes the $s$-th order $L^2$ Sobolev space and $\dot{H}^s$ stands for the homogenous Sobolev space. $\mc^s([0,T];E)$ is the set of $s$-times continuously differentiable functions from an interval $[0,T]\subset \R$ into a Banach space $E$, and $L^p(0,T;E)$ is the set of the $L^p$ functions from an interval $(0,T)$ to a Banach space $E$. $\nabla^s$ denotes any partial derivative $\pa^\alpha$ with multi-index $\alpha, |\alpha| = s$. Moreover, $\mc_w([0,T];X)$ denotes the set of the continuous functions from $[0,T]$ to $X$ with respect to the weak topology on $X$, that is, if $f \in \mc_w([0,T];X)$, then $f(t) \rightharpoonup f(t_0)$ weakly in $X$ as $t \to t_0$ for any $t_0 \in [0,T]$.

Employing the notations introduced above, we define a notion of our regular solution to the system \eqref{main_sys}.
\begin{definition}\label{def_sol}
For $T>0$ and $p \geq 2$, we say a triplet $(f,\rho,u)$ is a regular solution to \eqref{main_sys} on the time interval $[0,T]$ if it satisfies the followings:
\begin{enumerate}
\item[(i)]
$f \in \mc([0,T];H^{2,p}_\nu(\R^3 \times \R^3))$, $\pa_t f \in \mc([0,T];H^{1,p-2}_\nu(\R^3 \times \R^3))$,
\item[(ii)]
$\rho\ge 0$, $\rho^{\frac{\delta-1}{2}} - {\rho_\infty}^{\frac{\delta-1}{2}} \in \mc([0,T];H^3(\R^3))$, $\pa_t\lt(\rho^{\frac{\delta-1}{2}}\rt) \in \mc([0,T]; H^2(\R^3))$,
\item[(iii)]
$u \in \mc([0,T];H^{s'}(\R^3))\cap L^\infty(0,T;H^3(\R^3))$, $\rho^{\frac{\delta-1}{2}} \nabla^4 u \in L^2(0,T;L^2(\R^3))$, $\pa_t u \in \mc([0,T];H^1(\R^3)) \cap L^2(0,T;\dot{H}^2(\R^3))$,
\item[(iv)]
$(f,\rho,u)$ satisfies system \eqref{main_sys} in the sense of distribution,
\item[(v)]
when $\rho(t,x) = 0$, the following relation holds:
\[
\pa_t f + \xi \cdot \nabla f =0 \quad \mbox{and} \quad \pa_t u + u \cdot \nabla u = 0. 
\]
\end{enumerate}
\end{definition}
We now state our first main theorem below:
\begin{theorem}\label{T1.1}
Suppose that the initial data $(f_0, \rho_0, u_0)$ satisfy the following conditions:
\begin{align}\label{init_cond}
\begin{aligned}
&\textup{(i)}~~f_0 \in H^{2,p}_\nu(\R^3 \times \R^3) \quad \mbox{with} \quad p \ge 2, \\
&\textup{(ii)}~~ \rho_0 \ge 0, \  \mbox{ and  } \ (\rho_0^{\frac{\delta-1}{2}} -{\rho_\infty}^{\frac{\delta-1}{2}}, u_0) \in H^3(\R^3) \times H^3(\R^3),\\
&\textup{(iii)}~~ 1<\delta \le \min\lt\{\frac{\gamma+1}{2}, 3, \frac{2m+1}{3} \rt\}.
\end{aligned}
\end{align}
Then we can find a constant $T^*>0$ such that system \eqref{main_sys} admits a regular solution $(f,\rho,u)$ on the time interval $[0,T^*]$ in the sense of Definition \ref{def_sol}.
\end{theorem}

\begin{remark} When $\delta=1$, the condition $\textup{(ii)}$ in Theorem \ref{T1.1} can be replaced by
\[
\textup{(ii)}^*~~ \rho_0 \ge 0, \quad \ (\rho_0^{\frac{\gamma-1}{2}} -{\rho_\infty}^{\frac{\gamma-1}{2}}, u_0) \in H^3(\R^3) \times H^3(\R^3), \quad \mbox{and}\quad \nabla \rho_0/\rho_0 \in \mc([0,T];H^2(\R^3)),
\]
and the condition $\textup{(iii)}$ can be removed. Then we can obtain a solution $(f,\rho,u)$ to \eqref{main_sys} on the time interval $[0,T^*]$ in the following sense, for some constant $T^*>0$:
\begin{enumerate}
\item[(i)]
$f \in \mc([0,T^*];H^{2,p}_\nu(\R^3 \times \R^3))$, $\pa_t f \in \mc([0,T^*];H^{1,p-2}_\nu(\R^3 \times \R^3))$,
\item[(ii)]
$\rho\ge 0$, $\rho^{\frac{\gamma-1}{2}} - {\rho_\infty}^{\frac{\gamma-1}{2}} \in \mc([0,T^*];H^3(\R^3))$, $\pa_t\lt(\rho^{\frac{\gamma-1}{2}}\rt) \in \mc([0,T^*]; H^2(\R^3))$, 
\item[(iii)]
$\nabla\rho/\rho \in \mc([0,T^*];H^2(\R^3))$, $\pa_t(\nabla\rho/\rho) \in \mc([0,T^*];H^1(\R^3))$,
\item[(iv)]
$u \in \mc([0,T^*];H^{s'}(\R^3))\cap L^\infty(0,T^*;H^3(\R^3))$, $\pa_t u \in \mc([0,T^*];H^1(\R^3)) \cap L^2(0,T^*;\dot{H}^2(\R^3))$,
\item[(v)]
$(f,\rho,u)$ satisfies system \eqref{main_sys} in the sense of distribution,
\item[(vi)]
when $\rho(t,x) = 0$, the following relation holds:
\[
\pa_t f + \xi \cdot \nabla f =0 \quad \mbox{and} \quad \pa_t u + u \cdot \nabla u = 0. 
\]
\end{enumerate}
This is possible by the combination of arguments in \cite{LPZ17} and ours. Here we do not have to assume the condition $1<\gamma\le 3$ for the regularity
\[
\rho \in \mc([0,T^*];H^3(\R^3)) \quad \mbox{and} \quad \pa_t \rho \in \mc([0,T^*];H^2(\R^3)),
\] 
as in  \cite{LPZ17}, due to the regularity of $\phi:=\nabla\rho/\rho$. Indeed, for example, we have
\[\begin{aligned}
\|\pa_i\pa_j\pa_t\rho\|_{L^2} &= \| \pa_i\pa_j \lt( \nabla \rho \cdot u + \rho\nabla \cdot u\rt)\|_{L^2}\\
&= \lt\| \pa_i\pa_j \lt( n^{\frac{2}{\gamma-1}} \lt(\frac{2}{\gamma-1}\phi \cdot u + \nabla \cdot u\rt) \rt)\rt\|_{L^2}\\
&\le  \lt\|\pa_i \lt[n^{\frac{2}{\gamma-1}}\lt(\lt(\frac{2}{\gamma-1}\phi_j \lt(\frac{2}{\gamma-1}\phi \cdot u + \nabla \cdot u\rt)\rt) + \pa_j\lt(\frac{2}{\gamma-1} \phi\cdot u + (\nabla\cdot u)\rt)\rt)  \rt]\rt\|_{L^2}\\
&\le C\|n\|_{L^\infty}^{\frac{2}{\gamma-1}} \Big(\| \phi_i \phi_j (\phi\cdot u + \nabla\cdot u)  \|_{L^2} +\| \pa_i \phi_j (\phi\cdot u + \nabla\cdot u)  \|_{L^2} + \| \phi_j \pa_i(\phi\cdot u + \nabla\cdot u)  \|_{L^2} \\
&\hspace{2.5cm} +\| \phi_i \pa_j(\phi\cdot u + \nabla\cdot u)  \|_{L^2} + \| \pa_i \pa_j(\phi\cdot u + \nabla\cdot u)  \|_{L^2}\Big)\\
&\le C\|n\|_{L^\infty}^{\frac{2}{\gamma-1}}\Big(\|\phi\|_{L^\infty}^2(\|\phi\|_{L^\infty}+1)\|u\|_{H^1} + \|\phi\|_{L^\infty}\|u\|_{H^3}\|\nabla\phi\|_{L^2} \\
&\hspace{2.5cm}+ \|\phi\|_{L^\infty}(\|\phi\|_{H^2} + 1)\|u\|_{H^2} + (\|\phi\|_{H^2}+1)\|u\|_{H^3}\Big)\\
&\le  C\|n\|_{L^\infty}^{\frac{2}{\gamma-1}} (1+\|\phi\|_{H^2})^3\|u\|_{H^3} <\infty,
\end{aligned}\]
where $i,j=1,2,3$.
\end{remark}

\begin{remark}\label{rmk_spt}
Even if we replace the weight $\nu_p$ by,
\[
\nu_p(x,\xi) = (1+|x|^2 + |\xi|^2)^{p/2} e^{a|\xi|^2},
\]
for any positive constant $a>0$, we can still have the same result. For simplicity, in this current work, we took $a=1$.

By the way, if we assume that $f_0$ has a finite support in velocity, instead of a finite velocity moment, then our solution space for $f$, $H^{2,p}_\nu(\R^3 \times \R^3)$ can be replaced by $H^2(\R^3 \times \R^3)$. That is, we do not need to introduce the weighted Sobolev space. In Appendix \ref{app_spt}, we discuss more details on how we proceed our strategy for the proof of Theorem \ref{T1.1} in this supported initial data case. 
\end{remark}

\begin{remark} By using the same argument as in \cite[Corollary 1.1]{LPZ19}, if $\delta$ satisfies
\[
 1<\delta \le \min\lt\{\frac{\gamma+1}{2}, \frac53, \frac{2m+1}{3} \rt\},
\]
then 
\[
\rho - \rho_\infty \in \mc([0,T^*];H^3(\R^3)) \cap \mc^1([0,T^*];H^2(\R^3)).
\]
\end{remark}

Our main strategy for the existence theory, Theorem \ref{T1.1}, mainly relies on the classical regularization and linearization arguments. In order to handle difficulties arising from degenerate viscosities and possible vacuum states, we follow the arguments recently developed in \cite{LPZ19}. We refer to \cite{LPZ19} for the detailed discussion and review of studies on the related fluid equations. Our analysis mainly focuses on the construction of solutions to the kinetic equation in the desired solution space and the estimates of drag forcing effects in the momentum equation in \eqref{main_sys}. As stated above, we consider the case where the drag force also depends on the fluid density, thus an appropriate condition on the exponent $m$ is required to control the drag force by using the viscous term. We clarify this and the assumptions $\delta \leq \frac{2m+1}{3}$ in \eqref{init_cond} is added compared to the work \cite{LPZ19}. Regarding the kinetic equation, we would like to emphasize that the exponential weight in velocity of the form $e^{|\xi|^2}$ plays a crucial role in closing the estimates of $\|f\|_{H^{2,p}_\nu}$. Due to the presence of drag force, the $k$-th order velocity moment estimate on $f$ requires the bound on $(k+1)$-th order velocity moment on $f$. Thus the estimates of solutions with polynomial weights is not enough to overcome this difficulty. We observe that the exponential weight in velocity and the careful analysis of the drag forcing term can resolve this problem. By the way, due to some technicality, the polynomial weight in velocity is also taken into account together with the exponential weight in velocity. This can be easily verified by considering the free transport equation, $\pa_t f + \xi \cdot \nabla f = 0$. By the way, dealing with the polynomial weight in spatial variable $x$ is not necessary for the well-posedness theory. However, the estimates of finite-time singularity formation, which is our second main result stated below, require also some moment bounds in spatial variable $x$, thus we consider the weight in position. In this case, we can also easily check that the exponential weight in position is not necessary needed. On the other hand, as mentioned in Remark \ref{rmk_spt}, it is clear that this technical issue does not occur if we consider the compactly supported initial data $f_0$  in velocity. We first reformulate the system and introduce a linearized system. We would like to point out that appropriate existence results on the linearized system are not available, up to our best knowledge, so we further regularize the linearized system by considering an artificial viscosity $\eta > 0$. We then show the existence and uniqueness of regular solutions to that system and obtain the uniform-in-$\eta$ bound estimates of solutions. We finally pass to the limit $\eta \to 0$ and establish the existence of solutions to the linearized system by means of compactness arguments. 

Our second result is on the finite-time singularity formation for the system \eqref{main_sys}, within the framework of local existence theory established in Theorem \ref{T1.1}. As mentioned above, the finite-time blow-up phenomena are first observed in \cite{C17}. Under the assumptions on the existence of solutions satisfying some regularity and decay at infinity, the finite-time breakdown of regular solutions is discussed. Motivated from  \cite{LPZ19,Xin88,XY13}, where the finite-time blow-up of smooth solutions for the compressible Euler or Navier--Stokes system is obtained, and the {\it a priori} estimates in \cite{C17}, we introduce several physical quantities:\newline

\noindent $\bullet$ Mass.-
\[
m_\rho (t) := \intr \rho\,dx, \quad m_f(t) := \intrr f\,dxd\xi.
\]
$\bullet$ Momentum.-
\[
M(t):=  \intr \rho u\,dx + \intrr f \xi \,dxd\xi =: M_\rho(t) + M_f(t). 
\]
$\bullet$ Momentum weight.-
\[
W(t):= \intr \rho u \cdot x\,dx + \intrr (x\cdot \xi) f\,dxd\xi =: W_\rho(t) + W_f(t).
\]
$\bullet$ Momentum of inertia.-
\[
I(t) := \frac12 \intr \rho|x|^2\,dx + \frac12 \intrr f|x|^2\,dxd\xi =: I_\rho(t) + I_f(t).
\]
$\bullet$ Total energy.-
\[
E(t):= \frac12 \intr \rho|u|^2\,dx + \frac{1}{\gamma - 1}\intr \rho^\gamma\,dx + \frac12 \intrr f|\xi|^2\,dxd\xi=: E_k(t) + E_i(t) + E_f(t).
\]
Then we present our second result.
\begin{theorem}\label{main_thm2}Let $(f,\rho,u)$ be a solution to the Cauchy problem to \eqref{main_sys}  obtained in Theorem \ref{T1.1}. Suppose that the viscosity coefficient $\mu$ has a form of $\mu(\rho) = \rho^\delta$ with $\delta \in (1,\gamma)$ and the initial mass $m_\rho(0)$ is finite. Then the life-span $T$ of the solution $(f,\rho,u)$ is finite if 
\[
1 < \gamma < \frac53, \qquad \gamma - \frac13< \delta < \gamma, 
\]
and the initial data satisfy
\[
C_0 >\frac12\lt(\max\{2,3(\gamma - 1)\}E(0)\rt)\lt(J(0)^{\frac{\gamma-\delta}{\gamma-1}} +\frac{\lt(2\alpha + 9\beta \rt) \,(m_\rho(0))^{\frac{\gamma-\delta}{\gamma-1}}(\gamma-1)^{\frac{\delta-\gamma}{\gamma-1}}(\gamma-\delta)}{4(1 - 3(\gamma - \delta))}\rt)^{\frac{\gamma-1}{\gamma-\delta}}.
\]
Here $C_0$ and $J(0)$ are nonnegative constants given by
\[
C_0 := \lt(\frac{\pi^{3/2}}{\Gamma\lt(5/2 \rt)}\rt)^{1 - \gamma}\frac{(m_\rho(0))^{\frac{5\gamma - 3}{2}}}{2^{\frac{5\gamma - 3}{2}}(\gamma-1)},
\]
and $J(0) := I(0) - W(0) + E(0) \geq 0$, respectively. Here $\Gamma$ is the gamma function.
\end{theorem}

\begin{remark}If $2\alpha + 9\beta = 0$, then the initial condition in the above theorem reduces to
\[
C_0 >\frac12\lt(\max\{2,3(\gamma - 1)\}E(0)\rt)J(0).
\]
\end{remark}

The main idea of proving Theorem \ref{main_thm2} is based on the analysis of time evolution of physical quantities defined above. In brief, we estimate the lower and upper bounds on the internal energy $E_i$ in time, and investigate the initial data that leading to a contradiction as time goes to infinity. This yields that the life span of the regular solution satisfying that initial condition should be finite.

The rest of this paper is organized as follows. In Section \ref{sec:lin_ext}, we introduce a reformulated and linearized system associated to the system \eqref{main_sys} and show the local-in-time existence and uniqueness of regular solutions in the sense of Definition \ref{def_sol}. Based on the estimates for the linearized system, in Section \ref{sec:nlin_ext}, we construct approximate solutions to the nonlinear reformulated system and provide Cauchy estimates to prove the well-posedness. Finally, Section \ref{sec:blow} is devoted to have the finite-time singularity formation for the system \eqref{main_sys} under suitable assumptions on the initial physical quantities.

%
%
%
%
\section{Well-posedness of the linearized system}\label{sec:lin_ext}

%
%
\subsection{Reformulation, regularization, \& linearization}
Introducing a new variable $n := \rho^{\frac{\delta-1}{2}}$, inspired by the local sound speed, we rewrite the system \eqref{main_sys} as
\begin{align}\label{main_sys2}
\begin{aligned}
&\pa_t f + \xi \cdot \nabla f + \nabla_\xi \cdot \lt( (n)^{\frac{2m}{\delta-1}}(u - \xi)f \rt) = 0, \quad (x,\xi,t) \in \R^3 \times \R^3 \times \R_+,\cr
&\pa_t n + u\cdot \nabla n + \frac{\delta-1}{2}n \nabla \cdot u = 0, \quad (x,t) \in \R^3 \times \R_+,\cr
&\pa_t u + u \cdot \nabla u + \frac{\gamma}{\gamma-1}\nabla \lt(n^{\frac{2(\gamma-1)}{\delta-1}}\rt) +  n^2 Lu + \frac{\delta}{\delta-1}\nabla \lt(n^2\rt)\cdot \ml u = -n^{\frac{2(m-1)}{\delta-1}}\int_{\R^3} (u-\xi)f\,d\xi,
\end{aligned}
\end{align}
subject to initial data
\[
(f(x,\xi,0),n(x,0),u(x,0)) := (f_0(x,\xi), n_0(x), u_0(x)),
\]
and far-field behavior:
\[
f(x,\xi,t) \to 0, \quad (n, u) \to (n_\infty, 0), \quad n_\infty := \rho_\infty^{\frac{\delta-1}{2}},
\]
sufficiently fast as $|x|$, $|\xi| \to \infty$. Here, the Lam\'e operator $L$ is given by
\[
Lu = -\alpha \Delta u - (\alpha+\beta) \nabla (\nabla \cdot u),
\]
and the operator $\ml$ is given as 
\[
\ml u = -\lt(\alpha (\nabla u + (\nabla u)^T)  + \beta(\nabla \cdot u)\mathbb{I}_3\rt).
\] 

As we might expect, the well-posedness theory for the reformulated system \eqref{main_sys2} is closely related to our original system \eqref{main_sys}. Indeed, Definition \ref{def_sol} (ii) and (iii), together with the conditions on the parameters, yield the both characteristics associated to the kinetic and continuity equation in \eqref{main_sys} are well defined. This observation implies that Theorem \ref{T1.1} is equivalent to the following theorem. 

\begin{theorem}\label{thm_ref} Suppose that the initial data $(f_0, n_0, u_0)$ satisfy the following conditions:
\begin{align}\label{init_cond}
\begin{aligned}
&(i)~~f_0 \in H^{2,p}_\nu(\R^3 \times \R^3) \quad \mbox{with} \quad p \geq 2, \\
&(ii)~~ n_0 \ge 0, \  \mbox{ and  } \ (n_0 -n_\infty, u_0) \in H^3(\R^3) \times H^3(\R^3),\\
&(iii)~~\displaystyle 1<\delta \le \min\lt\{\frac{\gamma+1}{2}, 3, \frac{2m+1}{3} \rt\}.
\end{aligned}
\end{align}
Then we can find a constant $T^*>0$ such that system \eqref{main_sys2} admits a regular solution $(f,n,u)$ on the time interval $[0,T^*]$ in the sense of distributions satisfying
\begin{enumerate}
\item[(i)]
$f \in \mc([0,T^*];H^{2,p}_\nu(\R^3 \times \R^3))$, $\pa_t f \in \mc([0,T^*];H^{1,p-2}_\nu(\R^3 \times \R^3))$,
\item[(ii)]
$n\ge 0$, $n - n_\infty \in \mc([0,T^*];H^3(\R^3))$, $\pa_t n \in \mc([0,T^*]; H^2(\R^3))$,
\item[(iii)]
$u \in \mc([0,T^*];H^{s'}(\R^3))\cap L^\infty(0,T^*;H^3(\R^3))$, $n\nabla^4 u \in L^2(0,T^*;L^2(\R^3))$, $\pa_t u \in \mc([0,T^*];H^1(\R^3)) \cap L^2(0,T^*;\dot{H}^2(\R^3))$ for any $s' \in [2,3)$.
\end{enumerate}
\end{theorem}

Now our main interest is to prove the above theorem. For this, motivated from \cite{LPZ19}, we consider the following linearized and regularized system:
\begin{align}\label{lin_sys}
\begin{aligned}
&\pa_t f + \xi \cdot \nabla f + \nabla_\xi \cdot \lt( n^{m\theta} (v - \xi)f \rt) = 0, \quad (x,\xi,t) \in \R^3 \times \R^3 \times \R_+,\cr
&\pa_t n + v\cdot \nabla n + \theta^{-1}\psi \nabla \cdot v = 0, \quad (x,t) \in \R^3 \times \R_+,\cr
&\pa_t u + v \cdot \nabla u + \frac{\gamma}{\gamma-1}\nabla \lt(n^{\theta(\gamma-1)}\rt) +  (n^2 + \eta^2) Lu + \frac{\delta}{\delta-1}\nabla \lt(n^2\rt)\cdot \ml v = -n^{\theta(m-1)}\int_{\R^3} (v-\xi)f\,d\xi,
\end{aligned}
\end{align}
with the initial data
\bq\label{ini_lin_sys}
\lt(f(x,\xi,0), n(x,0), u(x,0)\rt) =: (f_0(x,\xi), n_0(x), u_0(x)), \quad (x,\xi) \in \R^3 \times \R^3,
\eq
where we set $\theta := 2/(\delta-1) > 0$ for notational simplicity. Here $\psi$ and $v$ are given functions satisfying 
\begin{align}\label{as_psi_v}
\begin{aligned}
&\psi-n_\infty \in \mc([0,T];H^3(\R^3)), \quad \pa_t \psi \in \mc([0,T];H^2(\R^3)),\\
&v \in \mc([0,T];H^{s'}(\R^3)) \cap L^\infty(0,T;H^3(\R^3)), \quad \pa_t v \in \mc([0,T];H^1(\R^3)) \cap L^2(0,T;\dot{H}^2(\R^3)), \quad \mbox{and}\\
&\psi \nabla^4 v \in L^2([0,T]\times\R^3),
\end{aligned}
\end{align}
for any constants $s'\in[2,3)$ and fixed time $T>0$ and
\[
\ml v = -\lt(\alpha (\nabla v + (\nabla v)^T)  + \beta(\nabla \cdot v)\mathbb{I}_3\rt).
\] 

In this section, our main objective is to establish the existence and uniqueness theory for the linearized system \eqref{lin_sys}, but without the artificial viscosity, i.e., \eqref{lin_sys} with $\eta= 0$. More precisely, we state the theorem below.
\begin{theorem}\label{thm_lin} Suppose that the initial data \eqref{ini_lin_sys} satisfy \eqref{init_cond} and $p \geq 2$.  Then there exists a positive constant $T^*>0$ such that system \eqref{lin_sys} with $\eta= 0$ admits a unique regular solution $(f,n,u)$ corresponding to the initial data $(f_0, n_0, u_0)$ on the time interval $[0,T^*]$, which satisfies
\[
\begin{aligned}
& f \geq 0, \quad f \in \mc([0,T^*]; H^{2,p}_\nu(\R^3 \times \R^3)) \quad \pa_t f \in \mc([0,T^*]; H^{1,p-2}_\nu(\R^3 \times \R^3)), \cr
&  n - n_\infty \in \mc([0,T^*];H^3(\R^3)), \quad \pa_t n \in \mc([0,T^*];H^2(\R^3)), \cr
& u \in \mc([0,T^*];H^{s'}(\R^3)) \cap L^\infty(0,T^*;H^3(\R^3)), \quad  n\nabla^4 u \in L^2(0,T^*;L^2(\R^3)),\\
&\mbox{and} \quad \pa_t u \in \mc([0,T^*];H^1(\R^3)) \cap L^2(0,T^*;\dot{H}^2(\R^3)),
\end{aligned}
\]
for $s' \in [2,3)$.
\end{theorem}

For the proof of Theorem \ref{thm_lin}, we first deal with the well-posedness of the linearized and regularized system \eqref{lin_sys}. Under the assumptions on the the given fluid velocity $v$, we notice that the characteristic associated to the continuity equation in \eqref{lin_sys} is well-defined. Thus, by a standard argument, we can readily deduce the $\mc([0,T];H^3(\R^3))$-regularity of $n-n_\infty$. This together with almost the same argument shows the characteristic associated to the kinetic equation \eqref{lin_sys} is well-defined, thus $\mc([0,T]; H^{2,p}_\nu(\R^3 \times \R^3))$-regularity of solution $f$ can be also obtained. We then consider the momentum equation in \eqref{lin_sys} with these regular solutions $n$ and $f$. This argument is now well established, thus we omit the details of that and only state the existence and uniqueness of regular solutions to the system \eqref{lin_sys}. We refer to \cite{CK15, LPZ19} for details. 

\begin{proposition}\label{prop_lsol}Suppose that the initial data \eqref{ini_lin_sys} satisfy \eqref{init_cond}. Then for any $T>0$ and $p \geq 2$, there exists a unique regular solution $(f,n,u)$ to the system \eqref{lin_sys} such that
\begin{align*}
\begin{aligned}
& f \geq 0, \quad f \in \mc([0,T]; H^{2,p}_\nu(\R^3 \times \R^3)) \quad \pa_t f \in \mc([0,T]; H^{1,p-2}_\nu(\R^3 \times \R^3)), \cr
&  n - n_\infty \in \mc([0,T];H^3(\R^3)), \quad \pa_t n \in \mc([0,T];H^2(\R^3)), \cr
& u \in \mc([0,T];H^{s'}(\R^3)) \cap L^\infty(0,T;H^3(\R^3)), \quad n\nabla^4 u \in L^2(0,T^*;L^2(\R^3)),\\
& \mbox{and} \quad \pa_t u \in \mc([0,T];H^1(\R^3)) \cap L^2(0,T;\dot{H}^2(\R^3)),
\end{aligned}
\end{align*}
for all constants $s'\in[2,3)$. 
\end{proposition}
 
%
%

\subsection{Uniform-in-$\eta$ bound estimates} In this subsection, we present uniform-in-$\eta$ bound estimates for the solution $(f,n,u)$ obtained in Proposition \ref{prop_lsol}. After establishing the required uniform bound estimates, we pass to the limit $\eta \to 0$ to have the existence of regular solutions to the system \eqref{lin_sys} without the artificial viscosity. 

For the bound estimates, we follow the strategy recently proposed in \cite{LPZ19}. We first choose a positive constant $\epsilon_0 > 1$ such that
\bq\label{mass_2}
2 + n_\infty + \|f_0\|_{H^{2,p}_\nu} + \|n_0\|_{L^\infty} + \|n_0 - n_\infty\|_{H^3} +  \|u_0\|_{H^3} \leq \epsilon_0.
\eq
We then assume that  for some $T^* \in (0,T]$,
\begin{align}\label{mass_3}
\begin{aligned}
&\sup_{0 \leq t \leq T^*}\lt(\|\psi(\cdot,t)-n_\infty\|_{H^3}^2 + \|v(\cdot,t)\|_{H^1}^2 \rt)\leq \epsilon_1^2,\cr
&\sup_{0 \leq t \leq T^*}\lt(\|\pa_t \psi(\cdot,t)\|_{L^2}^2 + \|\nabla^2 v(\cdot,t)\|_{L^2}^2  + \|\pa_t v(\cdot,t)\|_{L^2}^2\rt) \cr
&\hspace{3cm} + \int_0^{T^*} \lt(\|\psi(\cdot,t)\nabla^3 v(\cdot,t)\|_{L^2}^2 + \|\pa_t\nabla v(\cdot,t)\|_{L^2}^2 \rt)dt \leq \epsilon_2^2,\cr
&\sup_{0 \leq t \leq T^*}\lt(\|\pa_t \nabla \psi(\cdot,t)\|_{L^2}^2 + \|\nabla^3 v(\cdot,t)\|_{L^2}^2 + \|\pa_t \nabla v(\cdot,t)\|_{L^2}^2 \rt) \cr
&\hspace{3cm}+ \int_0^{T^*} \lt(\|\psi(\cdot,t) \nabla^4 v(\cdot,t)\|_{L^2}^2 + \|\pa_t \nabla^2 v(\cdot,t)\|_{L^2}^2 \rt) dt \leq \epsilon_3^2, \quad \mbox{and}\cr
&\sup_{0 \leq t \leq T^*} \|\pa_t \nabla^2 \psi(\cdot,t)\|_{L^2}^2 \leq \epsilon_4^2,\cr
\end{aligned}
\end{align}
where the constants $\epsilon_i, i=1,\dots, 4$ satisfy $\epsilon_i \leq \epsilon_{i+1}$ for $i=0,1,2,3$. Here the constants $\epsilon_i, i=1,\dots, 4$ and $T^*$ will be determined later. 

Then our goal of this part is to prove the following uniform-in-$\eta$ estimates. 
\begin{lemma}\label{lem_lsol}Let the assumptions of Proposition \ref{prop_lsol} be satisfied. If the assumptions \eqref{mass_3} hold, then the solution $(f,n,u)$ satisfies
\begin{align}\label{unif_bd}
\begin{aligned}
&\sup_{0 \leq t \leq T^*}\lt(\|n(\cdot,t)-n_\infty\|_{H^3}^2 + \|u(\cdot,t)\|_{H^1}^2 + \|f(\cdot,\cdot,t)\|_{H^{2,p}_\nu}^2 \rt)\leq \epsilon_1^2,\cr
&\sup_{0 \leq t \leq T^*}\lt(\|\pa_t n(\cdot,t)\|_{L^2}^2 + \|\nabla^2 u(\cdot,t)\|_{L^2}^2  + \|\pa_t u(\cdot,t)\|_{L^2}^2 + \|\pa_t f(\cdot,\cdot,t)\|_{L^{2,p-2}_\nu}^2\rt)\\
&\hspace{4cm}+ \int_0^{T^*} \lt(\|n(\cdot,t)\nabla^3 u(\cdot,t)\|_{L^2}^2 + \|\pa_t\nabla u(\cdot,t)\|_{L^2}^2 \rt)dt \leq \epsilon_2^2,\cr
&\sup_{0 \leq t \leq T^*}\lt(\|\pa_t \nabla n(\cdot,t)\|_{L^2}^2 + \|\nabla^3 u(\cdot,t)\|_{L^2}^2 + \|\pa_t \nabla u(\cdot,t)\|_{L^2}^2 +\|\pa_t \nabla_{(x,\xi)} f(\cdot,\cdot,t)\|_{L^{2,p-2}_\nu}^2\rt)\\
&\hspace{4cm}+ \int_0^{T^*} \lt(\|n(\cdot,t)\nabla^4 u(\cdot,t)\|_{L^2}^2 + \|\pa_t \nabla^2 u(\cdot,t)\|_{L^2}^2 \rt)dt \leq \epsilon_3^2, \quad \mbox{and}\cr
&\sup_{0 \leq t \leq T^*} \|\pa_t \nabla^2 n(\cdot,t)\|_{L^2}^2 \leq \epsilon_4^2.
\end{aligned}
\end{align}
\end{lemma}

In the following subsections, we provide the bound estimates on $n$, $f$, and $u$ step by step. 
%
%
%
%
\subsubsection{Estimates on the fluid density $n$}
\begin{lemma}\label{lem_n}Suppose that the initial data \eqref{ini_lin_sys} satisfy \eqref{init_cond}. Furthermore, we assume that the given vector field $v$ satisfies \eqref{mass_3}. Then there exists a unique solution $n$ to the continuity equation in \eqref{lin_sys} such that 
$$\begin{aligned}
& \|n(\cdot,t) - n_\infty\|_{H^3} \leq C\epsilon_0 \quad \mbox{and} \quad \|\nabla^k\pa_t n(\cdot,t)\|_{L^2} \leq C\epsilon_1\epsilon_{k+1}, \quad k=0,1,2
\end{aligned}$$
for $0 \leq t \leq T_1 := \min \{ T^*,  \epsilon_3^{-2}\}.$
\end{lemma}
\begin{proof} It follows from the continuity equation in \eqref{lin_sys} that for $0 \leq k \leq 3$
\[
\pa_t \pa^k (n - n_\infty) + \pa^k (v \cdot \nabla (n - n_\infty) ) + \theta^{-1}\pa^k (\psi \nabla \cdot v)= 0.
\]
This yields
$$\begin{aligned}
\frac12\frac{d}{dt}\|\pa^k (n - n_\infty) \|_{L^2}^2 &= \intr \pa^k (n - n_\infty) \cdot \pa_t \pa^k (n - n_\infty)\,dx \cr
&= -\intr \pa^k (n - n_\infty) \cdot \lt(\pa^k (v \cdot \nabla (n - n_\infty)) + \theta^{-1}\pa^k (\psi \nabla \cdot v) \rt) dx \cr
&=: \sfI_1 + \sfI_2.
\end{aligned}$$
Here, $\sfI_1$ can be estimated as follows.
$$\begin{aligned}
\sfI_1 &= - \intr \pa^k (n - n_\infty) \cdot \Big(v \cdot \nabla \pa^k (n - n_\infty) + \lt( \pa^k (v \cdot \nabla (n - n_\infty)) - v \cdot \nabla \pa^k (n - n_\infty) \rt) \Big) dx\cr
&\ls \|\nabla v\|_{L^\infty}\|\pa^k (n - n_\infty)\|_{L^2}^2 \cr
&\quad + \|\pa^k (n - n_\infty)\|_{L^2}\lt(\|\nabla (n - n_\infty)\|_{L^\infty}\|\pa^k v\|_{L^2} + \|\nabla v\|_{L^\infty}\|\pa^k (n - n_\infty)\|_{L^2} \rt)(1-\delta_{k,0})\cr
&\ls \|\nabla v\|_{L^\infty}\|\pa^k (n - n_\infty)\|_{L^2}^2 + \|\nabla (n - n_\infty)\|_{L^\infty}\|\pa^k (n - n_\infty)\|_{L^2}\|\pa^k v\|_{L^2}(1-\delta_{k,0}),
\end{aligned}$$
where $\delta_{k,0}$ denotes the Kronecker's delta, i.e., $\delta_{ij} = 1$ for $i=j$ and $\delta_{ij} = 0$ otherwise. 

For $\sfI_2$, we have
\begin{align*}
\sfI_2&= -\theta^{-1} \int_{\R^3} \lt(\psi \nabla \cdot (\pa^k v) + (\pa^k (\psi \nabla \cdot v) - \psi \nabla \cdot (\pa^k v)\rt) \cdot \pa^k(n-n_\infty)\,dx \\
& \ls \lt(\|\psi \nabla^{k+1} v\|_{L^2} \|\pa^k (n-n_\infty)\|_{L^2} + \lt(\|\nabla \psi\|_{L^\infty} \|\pa^k v\|_{L^2} + \|\pa^k \psi\|_{L^2} \|\nabla v\|_{L^\infty} \rt)\|\pa^k(n-n_\infty)\|_{L^2}\rt)(1-\delta_{k,0}).
\end{align*}
 Thus, by summing over $0 \leq k \leq 3$, we obtain
$$\begin{aligned}
&\frac{d}{dt}\|n - n_\infty\|_{H^3}^2 \cr
&\quad \ls \|\nabla v\|_{L^\infty}\|n - n_\infty\|_{H^3}^2 + \|\nabla (n - n_\infty)\|_{L^\infty}\|n - n_\infty\|_{H^3}\|\nabla v\|_{H^2} \cr
&\qquad + \|\psi-n_\infty\|_{H^3} \|\nabla v\|_{H^2} \|n-n_\infty\|_{H^3} + \|\psi \nabla^4 v\|_{L^2}\|n-n_\infty\|_{H^3} \cr
&\quad \ls \|\nabla v\|_{H^2} \|n - n_\infty\|_{H^3}^2 + \|\psi-n_\infty\|_{H^3} \|\nabla v\|_{H^2} \|n-n_\infty\|_{H^3} + \|\psi \nabla^4 v\|_{L^2}\|n-n_\infty\|_{H^3}.
\end{aligned}$$
This and together with applying the Gr\"onwall's lemma gives
$$\begin{aligned}
&\|n(\cdot,t) - n_\infty\|_{H^3} \cr
&\quad \leq C\|n_0 - n_\infty\|_{H^3}  \exp\lt(C\int_0^t\|\nabla v(\cdot,s)\|_{H^2}\,ds\rt)\cr
&\qquad + C  \lt(\int_0^t\|\psi(\cdot,s)-n_\infty\|_{H^3} \|\nabla v(\cdot,s)\|_{H^2} +\|\psi(\cdot,s)\nabla^4 v(\cdot,s)\|_{L^2}\,ds\rt) \exp\lt(C\int_0^t\|\nabla v(\cdot,s)\|_{H^2}\,ds\rt)\cr
&\quad \leq C\lt(\|n_0 - n_\infty\|_{H^3}+ \epsilon_1 \epsilon_3 t+ \sqrt{t}\lt(\int_0^t\|\psi(\cdot,s)\nabla^4 v(\cdot,s)\|_{L^2}^2\,ds\rt)^{1/2}\rt)\exp\lt(C\epsilon_3 t\rt).
\end{aligned}$$
We then use the inequality \eqref{mass_2} and the assumption \eqref{mass_3} to get
\[
\|n(\cdot,t) - n_\infty\|_{H^3} \leq C(\epsilon_0 + \epsilon_1\epsilon_3 t+ \epsilon_3\sqrt{t})e^{C\epsilon_3 t} \quad \mbox{for} \quad 0 \leq t \leq T^*.
\]
Choosing $T_1 = \min \{ T^*,  \epsilon_3^{-2}\}$ implies
\[
\|n(\cdot,t) - n_\infty\|_{H^3} \leq C\epsilon_0 \quad \mbox{for} \quad 0 \leq t \leq T_1.
\]
Using the above estimate, we easily estimate the $\dot{H}^k$-norm of $\pa_t n$ for $k=0,1,2$ as
$$\begin{aligned}
\|\pa_t n\|_{L^2} &\ls \|v\|_{L^2}\|\nabla n\|_{L^\infty} + (\|\psi - n_\infty\|_{L^\infty} + n_\infty)\|\nabla v\|_{L^2} \leq C\epsilon_1^2 , \cr
\|\nabla \pa_t n\|_{L^2} &\ls \|\nabla v \cdot \nabla n\|_{L^2} + \|v \cdot \nabla^2 n\|_{L^2} + \|\nabla \psi \cdot \nabla v\|_{L^2} + \|\psi \nabla^2 v\|_{L^2} \cr
&\ls (\|\nabla n\|_{L^\infty} + \|\nabla \psi\|_{L^\infty})\|\nabla v\|_{L^2} + \|v\|_{L^\infty}\|\nabla^2 n\|_{L^2} + (\|\psi - n_\infty\|_{L^\infty} + n_\infty)\|\nabla^2 v\|_{L^2} \cr
&\leq C\epsilon_1\epsilon_2,
\end{aligned}$$
and
$$\begin{aligned}
\|\nabla^2 \pa_t n\|_{L^2} &\ls \|\nabla^2v \cdot \nabla n\|_{L^2} + \|\nabla v \cdot \nabla^2 n\|_{L^2} + \|v \cdot \nabla^3 n\|_{L^2} + \|\nabla^2 \psi \cdot \nabla v\|_{L^2} + \|\nabla \psi \cdot \nabla^2 v\|_{L^2} + \|\psi \nabla^3 v\|_{L^2}\cr
&\ls (\|\nabla n\|_{L^\infty} + \|\nabla \psi\|_{L^\infty})\|\nabla^2 v\|_{L^2} + \|\nabla v\|_{L^\infty}(\|\nabla^2 n\|_{L^2} + \|\nabla^2 \psi\|_{L^2}) + \|v\|_{L^\infty}\|\nabla^3 n\|_{L^2} \cr
&\quad + (\|\psi - n_\infty\|_{L^\infty} + n_\infty)\|\nabla^3 v\|_{L^2}\cr
&\leq C\epsilon_1\epsilon_3.
\end{aligned}$$
Here we used $n_\infty \leq \epsilon_0$. This completes the proof.
\end{proof}
%
%
%
%

\subsubsection{Estimates on the particle distribution $f$}

\begin{lemma}\label{lem_f}
Suppose that the initial data \eqref{ini_lin_sys} satisfy \eqref{init_cond}. Furthermore, we assume that the given vector field $v$ satisfies \eqref{mass_3}. Then for $p \in [2,\infty)$, there exists a unique solution $f$ to the kinetic equation in \eqref{lin_sys} such that 
\[
\sup_{0 \leq t \leq T_2}\|f\|_{H^{2,p}_\nu} \leq C\epsilon_0 \quad \mbox{and} \quad  \sup_{0 \leq t \leq T_2}\|\pa_t \nabla^k_{(x,\xi)}f\|_{L^{2,p-2}_\nu} \leq C\epsilon_0^{m\theta+1} \epsilon_{k+1}, \quad k=0,1,
\]
where $T_2 = \min\lt\{ T_1, (\epsilon_0^{m\theta}\epsilon_3^2)^{-1}\rt\}$.
\end{lemma}
\begin{proof} 
For the smooth flow of reading, here we only provide the $H^{1,p}_\nu$-estimate of $f$ and postpone the rest to Appendix \ref{app_f}.

We begin with the $L^{2,p}_\nu$-estimate of $f$. It follows from the kinetic equation in \eqref{lin_sys} that
$$\begin{aligned}
\frac12\frac{d}{dt}\|f\|_{L^{2,p}_\nu}^2 &= -\intrr \nu_p(x,\xi) f \cdot \lt(\xi \cdot \nabla f + n^{m\theta}(v -\xi)\cdot\nabla_\xi f  - 3n^{m\theta} f \rt) dxd\xi\cr
&=\frac{p}{2}\intrr\nu_{p-2}(x,\xi)(x\cdot\xi) f^2\,dxd\xi +\frac32\intrr \nu_p(x,\xi)n^{m\theta} f^2\,dxd\xi \\
&\quad + \frac12\intrr(p\nu_{p-2}(x,\xi)+2\nu_p(x,\xi)) \xi\cdot(v-\xi) n^{m\theta} f^2\,dxd\xi\\
&\le C\|f\|_{L_\nu^{2,p}}^2 + C\|n\|_{L^\infty}^{m\theta}(1+\|v\|_{L^\infty})\|f\|_{L_\nu^{2,p}}^2\\
&\quad -\intrr \nu_p(x,\xi)|\xi|^2 n^{m\theta}f^2\,dxd\xi + \intrr \nu_p(x,\xi)v\cdot \xi n^{m\theta}f^2\,dxd\xi\\
&\le C\|f\|_{L_\nu^{2,p}}^2 + C\|n\|_{L^\infty}^{m\theta}(1+\|v\|_{L^\infty} + \|v\|_{L^\infty}^2)\|f\|_{L_\nu^{2,p}}^2,
\end{aligned}$$
where we used Young's inequality and the following identities:
\[
\nabla \nu_p(x,\xi) = p  \nu_{p-2}(x,\xi) x, \quad \nabla_\xi \nu_p(x,\xi) = (p  \nu_{p-2}(x,\xi) +2\nu_p(x,\xi))\xi
\]
and
\[
\nabla_\xi \cdot (\nu_p(x,\xi)  (v - \xi)) = (p \nu_{p-2}(x,\xi)+2\nu_p(x,\xi)) \xi \cdot (v - \xi) - 3 \nu_p(x,\xi).
\]
Thus we obtain
\bq\label{est_lf}
\frac{d}{dt}\|f\|_{L^{2,p}_\nu}^2 \leq C (1+\|n\|_{L^\infty}^{m\theta})(1+\|v\|_{L^\infty} )^2 \|f\|_{L^{2,p}_\nu}^2 
\eq
for $0 \leq t \leq T^*$.\\

\noindent For the $H^{1,p}_\nu$-estimate, we differentiate the kinetic equation in \eqref{lin_sys} with respect to $\pa_j, j=1,2,3$ to find
\[
\pa_t \pa_j f = - \xi \cdot \nabla \pa_j f  +3\pa_j ( n^{m\theta}f) - \pa_j(n^{m\theta}(v-\xi)\cdot\nabla_\xi f).
\]
Thus we obtain
$$\begin{aligned}
\frac12\frac{d}{dt}\|\pa_j f\|_{L^{2,p}_\nu}^2 & = - \intrr \nu_p(x,\xi)\pa_j f  (\xi \cdot \nabla \pa_j f) \, dxd\xi \cr
&\quad +3\intrr \nu_p(x,\xi)\pa_j f \pa_j (n^{m\theta}f)\,dxd\xi\\
&\quad - \intrr \nu_p(x,\xi)\pa_j f  \pa_j \lt( n^{m\theta}(v- \xi)\cdot \nabla_\xi f \rt) dxd\xi \cr
&=: \sfJ_1 + \sfJ_2+\sfJ_3,
\end{aligned}$$
where $\sfJ_1$ and $\sfJ_2$ can easily estimated as
\[\begin{aligned}
\sfJ_1&= \frac{p}{2}\intrr \nu_{p-2}(x,\xi)(x\cdot\xi)|\pa_j f|^2\,dx \le C\|\pa_j f\|_{L_\nu^{2,p}}^2,\\
\sfJ_2&\le C\|n\|_{L^\infty}^{m\theta-1}(\|\nabla n\|_{L^\infty} +\|n\|_{L^\infty})\|f\|_{H_\nu^{1,p}}^2.
\end{aligned}\]
On the other hand, we estimate $\sfJ_3$ as follows:
\[\begin{aligned}
\sfJ_3&= -\intrr \nu_p(x,\xi)\pa_j f \pa_j (n^{m\theta})(v-\xi)\cdot\nabla_\xi f\,dxd\xi\\
&\quad - \intrr \nu_p(x,\xi)\pa_j f n^{m\theta} \pa_j v \cdot \nabla_\xi f\,dxd\xi\\
&\quad -\intrr \nu_p(x,\xi)\pa_j f n^{m\theta}(v-\xi)\cdot\nabla_\xi(\pa_j f)\,dxd\xi\\
&=: \sfJ_3^1 + \sfJ_3^2 + \sfJ_3^3.
\end{aligned}\] 
For $\sfJ_3^1$, we use $m\theta>2$ and Young's inequality to have
\[\begin{aligned}
\sfJ_3^1&\le m\theta\intrr \nu_p(x,\xi) n^{m\theta-1}  |\pa_j f (\pa_j n)  v \cdot \nabla_\xi f|\,dxd\xi\\
&\quad +  \frac12\intrr \nu_p(x,\xi) |\xi|^2 n^{m\theta} |\pa_j f|^2\,dxd\xi + \frac{(m\theta)^2}{2}\intrr\nu_p(x,\xi)|\pa_j n|^2 n^{m\theta-2}|\nabla_\xi f|^2\,dxd\xi\\
&\le C\|n\|_{L^\infty}^{m\theta-2}(1+\|n\|_{L^\infty}+\|\nabla n\|_{H^2})^2(1+\|v\|_{L^\infty})\|f\|_{H_\nu^{1,p}}^2 \\
&\quad +   \frac12\intrr \nu_p(x,\xi) |\xi|^2n^{m\theta} |\pa_j f|^2\,dxd\xi.
\end{aligned}\]
For $\sfJ_3^2$, we easily get
\[
\sfJ_3^2\le \|n\|_{L^\infty}^{m\theta}\|\nabla v\|_{L^\infty} \|f\|_{H_\nu^{1,p}}^2.
\]
For $\sfJ_3^3$, we use Young's inequality to obtain
\[\begin{aligned}
\sfJ_3^3&= \frac12\intrr n^{m\theta}\nabla_\xi \cdot(\nu_p(x,\xi)(v-\xi))|\pa_jf|^2\,dxd\xi\\
&= \frac{p}{2}\intrr \nu_{p-2}(x,\xi)n^{m\theta}(v-\xi)\cdot\xi |\pa_j f|^2\,dx d\xi+ \intrr \nu_p(x,\xi) n^{m\theta}(v-\xi)\cdot\xi |\pa_j f|^2\,dxd\xi\\
&\quad -\frac32\intrr \nu_p(x,\xi)n^{m\theta}|\pa_j f|^2\,dxd\xi\\
&\le C\|n\|_{L^\infty}^{m\theta}(1+\|v\|_{L^\infty} + \|v\|_{L^\infty}^2)\|\pa_j f\|_{L_\nu^{2,p}}^2 - \frac12\intrr \nu_p(x,\xi)|\xi|^2 n^{m\theta}|\pa_j f|^2\,dxd\xi.
\end{aligned}\]
Combining the above observation yields
\bq\label{est_hf1}
\frac{d}{dt}\|\nabla f\|_{L^{2,p}_\nu}^2 \leq C \|n\|_{L^\infty}^{m\theta - 2}( 1+ \|n\|_{W^{1,\infty}} )^2(1+\|v\|_{H^3})^2\|f\|_{H^{1,p}_\nu}^2.
\eq
for $0 \leq t \leq T^*$.

We next take the differential operator $\pa_{\xi_j}, j=1,2,3$ to the kinetic equation in \eqref{lin_sys} to get
\[
\pa_t \pa_{\xi_j} f + \xi \cdot \nabla \pa_{\xi_j}f + n^{m\theta} (v- \xi) \cdot \nabla_\xi \pa_{\xi_j} f = -\pa_j f + 4 n^{m\theta}\pa_{\xi_j}f.
\]
Then we estimate the $L^2$-norm of $\pa_{\xi_j}f$ as
$$\begin{aligned}
\frac12\frac{d}{dt}\|\pa_{\xi_j}f\|_{L^{2,p}_\nu}^2 &= -\intrr \nu_p(x,\xi) \pa_{\xi_j} f \cdot \lt( \xi \cdot \nabla \pa_{\xi_j}f + n^{m\theta} (v- \xi) \cdot \nabla_\xi \pa_{\xi_j} f + \pa_j f - 4n^{m\theta}\pa_{\xi_j}f \rt) dxd\xi\cr
& = \frac p2\intrr \nu_{p-2}(x,\xi) (x \cdot \xi) |\pa_{\xi_j} f|^2\,dxd\xi +\frac12  \intrr \nabla_\xi \cdot (\nu_p(x,\xi)  (v - \xi)) n^{m\theta} |\pa_{\xi_j}f|^2\,dxd\xi \cr
&\qquad  - \intrr \nu_p(x,\xi) \pa_{\xi_j}f \pa_j f\,dxd\xi - 4\intrr n^{m\theta} \nu_p(x,\xi)  |\pa_{\xi_j}f|^2\,dxd\xi\cr
& \ls  (\|n\|_{L^\infty}^{m\theta} + 1)(\|v\|_{L^\infty} + 1)\| f\|_{H^{1,p}_\nu}^2 + \intrr \nu_p(x,\xi)n^{m\theta}(v-\xi)\cdot\xi |\pa_{\xi_j}f|^2\,dxd\xi\\
&\ls  (1+\|n\|_{L^\infty}^{m\theta})(1+\|v\|_{L^\infty})^2\| f\|_{H^{1,p}_\nu}^2 
\end{aligned}$$
This, together with \eqref{est_lf} and \eqref{est_hf1}, gives
\bq\label{est_hf2}
\frac{d}{dt}\|f\|_{H^{1,p}_\nu}^2 \ls  C (1+\|n\|_{L^\infty}^{m\theta})(1+ \|n\|_{W^{1,\infty}})^2(1+\|v\|_{H^3})^2\|f\|_{H^{1,p}_\nu}^2.
\eq
for $0 \leq t \leq T^*$. This completes the proof of the $H^{1,p}_\nu$-estimate of $f$.
\end{proof}

%
%
%
%

\subsubsection{Estimates on the fluid velocity $u$}
In this part, we provide the bound estimates for $H^3$ norm of the fluid velocity $u$.

\begin{lemma}\label{lem_u}
Suppose that the initial data \eqref{ini_lin_sys} satisfy \eqref{init_cond}. Furthermore, we assume that the given vector field $v$ satisfies \eqref{mass_3} and $p \ge 2$. Then there exists a unique solution $u$ to the momentum equations in \eqref{lin_sys} such that 
\begin{align*}
&\sup_{0 \le t \le T_3} \|u(t)\|_{H^3}^2 + \int_0^{T_3} 
\lt(\|\sqrt{(n^2(s)+\eta^2)} \nabla^3 u(s)\|_{L^2}^2+  \|\sqrt{(n^2(s)+\eta^2)} \nabla^4 u(s)\|_{L^2}^2\rt)\,ds \le C\epsilon_0^2,\\
& \sup_{0 \le t \le T_3} \|\pa_t \nabla^k u(t)\|_{L^2}^2 \le C\epsilon_0^{2\theta(\max\{m,\gamma\}-1)+2}\epsilon_{k+1}^2, \quad k = 0,1, \quad \mbox{and}\\
&\int_0^{T_3} \|\pa_t \nabla^2 u(s)\|_{L^2}^2\,ds \le C\epsilon_0^3,
\end{align*}
where $T_3 := \min\{T_2, (\epsilon_0^{\theta(\max\{m,\gamma\}-1)+1} \epsilon_3)^{-2}\}$.
\end{lemma}
\begin{proof} Here again, we only provide the $H^2$-estimate of $u$, and leave the rest of the proof to Appendix \ref{app_c}. Let us begin with the zeroth order estimate of $u$. A straightforward computation yields
\begin{align*}
\frac12\frac{d}{dt} \|u\|_{L^2}^2 &= - \int_{\R^3} \lt( v\cdot \nabla u + \frac{\gamma}{\gamma-1} \nabla \lt(n^{\theta(\gamma-1)}\rt) + (n^2 + \eta^2) Lu + \frac{\delta}{\delta-1} \nabla(n^2) \ml v\rt) \cdot u\,dx\\
&\quad -\intrr n^{\theta(m-1)}(v-\xi)f \cdot u\,dxd\xi\\
&=:\sum_{i=1}^5 \sfK_i.
\end{align*}
Here, we have
\begin{align*}
\sfK_1 &\ls \|\nabla v\|_{L^\infty} \|u\|_{L^2}^2 \ls \|v\|_{H^3} \|u\|_{L^2}^2 \ls \epsilon_3 \|u\|_{L^2}^2,\\
\sfK_2 &\ls \|n\|_{L^\infty}^{\theta(\gamma-1)-1} \|\nabla n\|_{L^2} \|u\|_{L^2} \ls \epsilon_0^{\theta(\gamma-1)}\|u\|_{L^2},\\
\sfK_4 &\ls \|n\|_{L^\infty} \|\nabla n\|_{L^2} \|\ml v\|_{L^\infty} \|u\|_{L^2}\ls \epsilon_0^2 \epsilon_3 \|u\|_{L^2},\\
\sfK_5 &\ls \|n\|_{L^\infty}^{\theta(m-1)}\|u\|_{L^2} \intrr |u|(|v| + |\xi|)f\,dxd\xi \cr
&\leq  \|n\|_{L^\infty}^{\theta(m-1)}\|u\|_{L^2} \lt(\|v\|_{L^\infty} \lt(\intr \lt(\intr f\,d\xi\rt)^2 dx  \rt)^{1/2}  + \lt(\intr \lt(\intr |\xi| f\,d\xi \rt)^2 dx \rt)^{1/2}\rt)\cr
&\leq C \|n\|_{L^\infty}^{\theta(m-1)}\|u\|_{L^2}(\|v\|_{L^\infty}+1)\|f\|_{L^{2,p}_\nu}  \ls \epsilon_0^{\theta(m-1)+1} \epsilon_2 \|u\|_{L^2}
\end{align*}
where we used, for the estimate of $\sfK_5$,
\[
\intr \lt(\intr |\xi|^k f\,d\xi\rt)^2 dx \le \lt(\intrr \nu_p(x,\xi)f^2\,dxd\xi\rt)\lt(\intr |\xi|^2 e^{-|\xi|^2}d\xi\rt)\le C\intrr \nu_p(x,\xi)f^2\,dxd\xi
\]
for $k=0,1$ due to $p \geq 2$.

For the term $\sfK_3$, we get
\begin{align*}
\sfK_3&= \int_{\R^3} (n^2 + \eta^2) \lt(\alpha \Delta u + (\alpha+\beta)\nabla(\nabla\cdot u)\rt) \cdot u\,dx\\
&= - \alpha\int_{\R^3} (n^2 + \eta^2)|\nabla u|^2\,dx -(\alpha+\beta)\int_{\R^3} (n^2 + \eta^2) (\nabla \cdot u)^2\,dx\\
&\quad - 2\alpha\int_{\R^3} n(\nabla n \cdot \nabla u) \cdot u\,dx -2(\alpha+\beta)\int_{\R^3} n (\nabla n \cdot u)(\nabla \cdot u)\,dx\\
&\le - \alpha\int_{\R^3} (n^2 + \eta^2)|\nabla u|^2\,dx -(\alpha+\beta)\int_{\R^3} (n^2 + \eta^2) (\nabla \cdot u)^2\,dx\\
&\quad +2\alpha \|n\nabla u\|_{L^2}\|\nabla n\|_{L^\infty}\|u\|_{L^2} +2(\alpha+\beta)\|n\nabla\cdot u\|_{L^2}\|\nabla n\|_{L^\infty}\|u\|_{L^2}\\
&\le -\frac\alpha2\int_{\R^3}(n^2 + \eta^2)|\nabla u|^2\,dx -\frac{\alpha+\beta}{2}\int_{\R^3} (n^2 + \eta^2) (\nabla \cdot u)^2\,dx\\
&\quad +(4\alpha+2\beta)\|\nabla n\|_{L^\infty}^2 \|u\|_{L^2}^2,
\end{align*}
where we used $\alpha>0$, $2\alpha + 3\beta \ge 0$ to get $\alpha+\beta \ge 0$.

Then we combine the estimates for $\sfK_i$'s and use Young's inequality to yield
\begin{align*}
\frac{d}{dt}&\|u\|_{L^2}^2 + \alpha\int_{\R^3} (n^2 + \eta^2) |\nabla u|^2\,dx + (\alpha+\beta)\int_{\R^3} (n^2 + \eta^2) (\nabla \cdot u)^2\,dx\\
&\le C(\epsilon_0^2 + \epsilon_3)\|u\|_{L^2}^2 + C(\epsilon_0^{2\theta(\gamma-1)} + \epsilon_0^4 \epsilon_3^2 + \epsilon_0^{2\theta(m-1)+2}\epsilon_3^2)\\
&\le C \epsilon_0^{2\theta(\zeta-1)+2}\epsilon_3^2(1+\|u\|_{L^2}^2),
\end{align*}
where $\zeta := \max\{m, \gamma\}$. Hence, one gets
\begin{align*}
\|u(t)&\|_{L^2}^2 +\int_0^t e^{C\epsilon_0^{2\theta(\zeta-1)+2}\epsilon_3^2 (t-s)} \|\sqrt{n^2(s) + \eta^2}\nabla u(s)\|_{L^2}^2  \,ds \\
&\le C\lt(\|u_0\|_{L^2}^2 e^{C\epsilon_0^{2\theta(\zeta-1)+2}\epsilon_3^2 t} + t \epsilon_0^{2\theta(\zeta-1)+2}\epsilon_3^2 e^{C\epsilon_0^{2\theta(\zeta-1)+2}\epsilon_3^2 t} \rt),
\end{align*}
for $0 \le t \le T_2$. We choose $T_3 := \min\lt\{T_2, (\epsilon_0^{\theta(\zeta-1)+1}\epsilon_3)^{-2}\rt\}$ to have
\[
\|u(t)\|_{L^2}^2 + \int_0^t \|\sqrt{n^2(s) + \eta^2}\nabla u(s)\|_{L^2}^2  \,ds \le C\epsilon_0^2.
\]

We next show the $H^k$-estimate ($k=1,2$) of $u$. We have from \eqref{lin_sys} that 
\begin{align}\label{u_H2}
\begin{aligned}
&\frac12\frac{d}{dt}\|\pa^k u\|_{L^2}^2 \cr
&\quad = -\int_{\R^3} (v \cdot \nabla \pa^k u)\cdot \pa^k u\,dx - \int_{\R^3} \lt( \pa^k (v \cdot \nabla u)  - v \cdot \nabla \pa^k u\rt)\cdot \pa^k u\,dx\\
&\qquad -\frac{\gamma}{\gamma-1} \int_{\R^3} \pa^k \nabla(n^{\theta(\gamma-1)}) \cdot \pa^k u\,dx - \int_{\R^3} (n^2 +\eta^2)\pa^k Lu \cdot \pa^k u\,dx \\
&\qquad - \int_{\R^3} \lt(\pa^k ((n^2 +\eta^2)Lu - (n^2 + \eta^2)\pa^k Lu\rt)\cdot \pa^k u\,dx -\frac{\delta}{\delta-1} \int_{\R^3} \nabla (n^2) \pa^k \ml v \cdot \pa^k u\,dx \\
&\qquad - \int_{\R^3} \lt(\pa^k (\nabla(n^2) \ml v) - \nabla (n^2) \pa^k \ml v \rt)\cdot \pa^k u\,dx +\int_{\R^3 \times \R^3} \pa^{k-1}\lt(n^{\theta(m-1)} (v-\xi)f \rt) \cdot \pa^{k+1} u\,dxd\xi\\
&\quad =: \sum_{i=1}^8 \sfL_i.
\end{aligned}
\end{align}
Here, for $k=1,2$, we easily get
\begin{align*}
\sfL_1 &\ls \|\nabla v\|_{L^\infty}\|\pa^k u\|_{L^2}^2 \ls \epsilon_3 \|\nabla^k u\|_{L^2}^2,\\
\sfL_2 &\ls \sum_{m=1}^k \int_{\R^3} |\pa^m v| | \nabla \pa^{k-m} u| |\pa^k u|\,dx \\
&\ls \|\nabla v\|_{L^\infty}\|\nabla\pa^{k-1} u\|_{L^2}^2 + \|\nabla^2 v\|_{L^3} \|\nabla u\|_{L^6}\|\pa^k u\|_{L^2}\delta_{k,2}\\
&\ls \epsilon_3 \|\nabla^k u\|_{L^2}^2.
\end{align*}
For the term $\sfL_3$, we obtain
\begin{align*}
\sfL_3 &= \frac{\gamma}{\gamma-1}\int_{\R^3} \pa^k (n^{\theta(\gamma-1)}) \pa^k (\nabla \cdot u)\,dx\\
&= \theta\gamma \int_{\R^3} \pa^{k-1}(n^{\theta(\gamma-1)-1} \pa n) \pa^k (\nabla \cdot u)\,dx\\
&\ls \left\{\begin{array}{lc} \|n\|_{L^\infty}^{\theta(\gamma-1)-2} \|\nabla n\|_{L^2} \|n \pa^k (\nabla u)\|_{L^2} & \mbox{if } \ k=1\\
\|n\|_{L^\infty}^{\theta(\gamma-1)-3} \|\nabla n\|_{L^\infty} \|\nabla n\|_{L^2} \|n \pa^k (\nabla u)\|_{L^2} + \|n\|_{L^\infty}^{\theta(\gamma-1)-2} \|\nabla^2 n\|_{L^2} \|n \pa^k (\nabla u)\|_{L^2} & \mbox{if } \ k=2\end{array} \right.\\
&\ls \epsilon_0^{\theta(\gamma-1)-1} \|n \pa^k (\nabla u)\|_{L^2}.
\end{align*}
We next estimate $\sfL_4$ as
\begin{align*}
\sfL_4 &= \alpha \int_{\R^3} (n^2 + \eta^2) \pa^k \Delta u \cdot \pa^k u\,dx  + (\alpha+\beta) \int_{\R^3} (n^2 + \eta^2) \pa^k \nabla (\nabla \cdot u)\cdot \pa^k u\,dx\\
&= -\alpha \int_{\R^3} (n^2 + \eta^2) |\pa^k \nabla u|^2\,dx - (\alpha+\beta) \int_{\R^3} (n^2 + \eta^2) (\pa^k (\nabla \cdot u))^2\,dx\\
&\quad - 2\alpha \int_{\R^3} n (\nabla n \cdot \pa^k \nabla u)\cdot \pa^k u\,dx - 2(\alpha+\beta) \int_{\R^3 } n \pa^k (\nabla \cdot u) \nabla n\cdot \pa^k u\,dx\\
&\le -\alpha \int_{\R^3} (n^2 + \eta^2) |\pa^k \nabla u|^2\,dx - (\alpha+\beta) \int_{\R^3} (n^2 + \eta^2) (\pa^k (\nabla \cdot u))^2\,dx\\
&\quad +2\alpha \|n \pa^k (\nabla u)\|_{L^2} \|\nabla n\|_{L^\infty}\|\pa^k u\|_{L^2} + 2(\alpha+\beta) \|n\pa^k (\nabla \cdot u)\|_{L^2} \|\nabla n\|_{L^\infty}\|\pa^k u\|_{L^2}\\
&\le -\frac{3\alpha}{4}\int_{\R^3} (n^2 + \eta^2) |\pa^k \nabla u|^2\,dx -\frac{3(\alpha+\beta)}{4}\int_{\R^3} (n^2 + \eta^2) (\pa^k (\nabla \cdot u))^2\,dx\\
&\quad + (8\alpha+4\beta) \|\nabla n\|_{L^\infty}^2 \|\pa^k u\|_{L^2}^2.
\end{align*}
Note that the last term on the right hand side can be bounded by $C\epsilon_0^2 \|\nabla^k u\|_{L^2}^2$. For $\sfL_5$, we get
\begin{align*}
\sfL_5 &\ls \sum_{m=1}^k \int_{\R^3} \lt|\pa^m (n^2) \pa^{k-m} Lu \cdot \pa^k u\rt|\,dx\\
&\ls \left\{\begin{array}{lc} \|n\nabla^{2} u\|_{L^2} \|\nabla n\|_{L^\infty}\|\pa u\|_{L^2} & \mbox{if } \ k=1\\
\|n\nabla^{3} u\|_{L^2} \|\nabla n\|_{L^\infty}\|\pa^2 u\|_{L^2} + \|\nabla n\|_{L^\infty}^2 \|\nabla^2 u\|_{L^2}^2 + \|\nabla^2 n\|_{L^3} \|n\nabla^2 u\|_{L^6}\|\nabla^2 u\|_{L^2} & \mbox{if } \ k=2  \end{array} \right.\\
&\ls \left\{\begin{array}{lc} \|n\nabla^{2} u\|_{L^2} \|\nabla n\|_{L^\infty}\|\nabla u\|_{L^2} & \mbox{if } \ k=1\\
\|n\nabla^{3} u\|_{L^2} \|\nabla n\|_{L^\infty}\|\nabla^2 u\|_{L^2} + \|\nabla n\|_{H^2}^2 \|\nabla^2 u\|_{L^2}^2 & \mbox{if } \ k=2  \end{array} \right.\\
&\leq C\epsilon_0 \|\nabla^k u\|_{L^2} \|n \nabla^{k+1}u\|_{L^2} + C\epsilon_0^2\|\nabla^k u\|_{L^2}^2,
\end{align*}
where we used
\[
\|n \nabla^k u\|_{L^6} \ls \|\nabla ( n \nabla^k u)\|_{L^2} \ls \|\nabla n\|_{L^\infty}\|\nabla^k u\|_{L^2} + \|n\nabla^{k+1} u\|_{L^2}.
\]
We treat $\sfL_6$ and $\sfL_7$ as
\[
\sfL_6 \ls \|n\|_{L^\infty}\|\nabla n\|_{L^\infty} \|\pa^k \ml v\|_{L^2} \|\pa^k u\|_{L^2} \ls \epsilon_0^2 \epsilon_3 \|\pa^k u\|_{L^2}
\]
and
\begin{align*}
\sfL_7 &\ls \sum_{m=1}^k \int_{\R^3} \lt| \pa^m (n\nabla n) \pa^{k-m} \ml v  \cdot \pa^k u \rt| dx \\
&\ls \lt\{\begin{array}{lc} \|n\|_{L^\infty}\|\nabla^2 n\|_{L^3} \|\ml v\|_{L^6} \|\pa u\|_{L^2} + \|\nabla n\|_{L^\infty}^2\|\ml v\|_{L^2}\|\pa u\|_{L^2}& \mbox{if } \ k=1\\
\lt((\|n\|_{L^\infty}+\|\nabla n\|_{H^2})^2 \|\ml v\|_{L^\infty}  + \|n\|_{L^\infty}\|\nabla^2 n\|_{L^3}\|\pa\ml v\|_{L^6}  + \|\nabla n\|_{L^\infty}^2 \|\pa \ml v\|_{L^2}\rt)\|\pa^2 u\|_{L^2} & \mbox{if } \ k=2 
\end{array} \rt.\\
&\ls \epsilon_0^2 \epsilon_3 \|\nabla^k u\|_{L^2}.
\end{align*}
For the estimate of $\sfL_8$, we use the similar argument as in the $L^2$ estimate to deduce that for $k=1$,
$$\begin{aligned}
\sfL_8 &= \lt| \intrr n^{\theta(m-1)-1} (v - \xi) f \cdot n \pa^2 u\,dxd\xi\rt|\cr
&\leq  \|n\|_{L^\infty}^{\theta(m-1)-1} \intrr |n \nabla^2 u|(|v| + |\xi|)f\,dxd\xi \cr
&\ls \|n\|_{L^\infty}^{\theta(m-1)-1} \|n\nabla^{2} u\|_{L^2}\lt( \|v\|_{L^\infty}+1\rt) \|f\|_{L^{2,p}_\nu} \cr
& \ls \epsilon_0^{\theta(m-1)} \epsilon_3 \|n \nabla^2 u\|_{L^2},
\end{aligned}$$
and for $k=2$,
\begin{align*}
\sfL_8 &\ls \|n\|_{L^\infty}^{\theta(m-1)-2} \|\nabla n\|_{L^\infty} \intrr |n \nabla^3 u|(|v| + |\xi|)f\,dxd\xi  + \|n\|_{L^\infty}^{\theta(m-1)-1} \|\nabla v\|_{L^\infty}\intrr |n \nabla^3 u| f\,dxd\xi\cr
&\quad + \|n\|_{L^\infty}^{\theta(m-1)-1}  \intrr |n \nabla^3 u|(|v| + |\xi|)|\nabla f|\,dxd\xi\cr
&\ls \|n\|_{L^\infty}^{\theta(m-1)-2} \|\nabla n\|_{L^\infty} (\|v\|_{L^\infty}+1)\|f\|_{L^{2,p}_\nu} \|n\nabla^3 u\|_{L^2} + \|n\|_{L^\infty}^{\theta(m-1)-1} \|\nabla v\|_{L^\infty} \|f\|_{L^{2,p}_\nu} \|n\nabla^3 u\|_{L^2}\\
&\quad + \|n\|_{L^\infty}^{\theta(m-1)-1} \lt( \|v\|_{L^\infty}+1\rt) \|\nabla f\|_{L^{2,p}_\nu} \|n\nabla^3 u\|_{L^2}\\
&\ls \epsilon_0^{\theta(m-1)}\epsilon_3 \|n\nabla^3 u\|_{L^2}.
\end{align*}
Thus, we combine all the estimates for $\sfL_i$'s and use Young's inequality to yield
\begin{align*}
\frac{d}{dt}&\|\nabla^k u\|_{L^2}^2 + \alpha \int_{\R^3} (n^2 + \eta^2) |\nabla^{k+1} u|^2\,dx + (\alpha+\beta) \int_{\R^3} (n^2 + \eta^2) |\nabla^k (\nabla \cdot u)|^2\,dx\\
&\le C\lt(\epsilon_0^{2\theta(\zeta-1)}\epsilon_3^2\|\nabla^k u\|_{L^2}^2 + \epsilon_0^{2\theta(\zeta-1)}\epsilon_3^2\rt),
\end{align*}
for $0 \le t \le T_2$. Here we remind the reader that our assumptions imply $2\theta(\zeta-1) \geq 8$. 

Finally, applying Gr\"onwall's lemma implies
\[
\|\nabla^k u(\cdot,t)\|_{L^2}^2 + \int_0^t \|\sqrt{n^2(\cdot,s)+\eta^2} \nabla^{k+1} u(\cdot,s)\|_{L^2}^2\,ds \le C\epsilon_0^2, \quad 0\le t \le T_3.
\]
From the above, we conclude the desired $H^2$-estimate of $u$. 
\end{proof}

We now have all the ingredients for the proof of Lemma \ref{lem_lsol}. 

\begin{proof}[Proof of Lemma \ref{lem_lsol}]
A simple combinations of Lemmas \ref{lem_n}, \ref{lem_f}, and \eqref{lem_u} yields
\begin{align*}
&1+n_\infty^2 + \|n(\cdot,t)-n_\infty\|_{H^3}^2 + \|u(\cdot,t)\|_{H^3}^2 + \|f(\cdot,\cdot,t)\|_{H^{2,p}_\nu}^2  \le C\epsilon_0^2, \\
& \int_0^t \lt(\|\sqrt{(n^2(\cdot,s)+\eta^2)} \nabla^3 u(\cdot,s)\|_{L^2}^2+  \|\sqrt{(n^2(\cdot,s)+\eta^2)} \nabla^4 u(\cdot,s)\|_{L^2}^2\rt)\,ds \le C\epsilon_0^2,\\
& \|\pa_t \nabla^k n(\cdot,t)\|_{L^2}^2 \leq C\epsilon_1^2\epsilon_{k+1}^2, \quad k=0,1,2,\\
& \|\pa_t \nabla^k u(t)\|_{L^2}^2 \le C\epsilon_0^{2\theta(\zeta-1)+2}\epsilon_{k+1}^2, \quad k = 0,1,  \quad \int_0^t \|\pa_t \nabla^2 u(s)\|_{L^2}^2ds \le C\epsilon_0^3, \quad \mbox{and}\\
&\|\pa_t \nabla^k_{(x,\xi)}f\|_{L^{2,p-2}_\nu}^2 \leq C\epsilon_0^{2m\theta+2} \epsilon_{k+1}^2, \quad k=0,1
\end{align*}
for $0 \le t \le T_3$. We now define the constants $\epsilon_i, i=1,\dots, 4$ as
\[
\epsilon_1 := C^{\frac12}\epsilon_0, \quad \epsilon_2 := C^{\frac12} \epsilon_0^{\theta\zeta+1}\epsilon_1, \quad \epsilon_3 := C^{\frac12} \epsilon_0\epsilon_2, \quad \epsilon_4 := C^{\frac12}\epsilon_0\epsilon_3, \quad T^* := \min\{T, (\epsilon_0^{\theta(\zeta-1)+1}\epsilon_3)^{-2}\}.
\]
Then this asserts \eqref{unif_bd}.
\end{proof}
%
%

\subsection{Local well-posedness of the linearized system \eqref{lin_sys} without the artificial viscosity}
Since the previous estimates are independent of the choice of $\eta$, by passing to the limit $\eta \to 0$, we can proceed to the proof for the well-posedness of \eqref{lin_sys} without the artificial viscosity, i.e. system \eqref{lin_sys} with $\eta=0$:
\bq\label{lin_sys2}
\begin{aligned}
&\pa_t f + \xi \cdot \nabla f + \nabla_\xi \cdot \lt( n^{m\theta} (v - \xi)f \rt) = 0, \quad (x,\xi,t) \in \R^3 \times \R^3 \times \R_+,\cr
&\pa_t n + v\cdot \nabla n + \theta^{-1}\psi \nabla \cdot v = 0, \quad (x,t) \in \R^3 \times \R_+,\cr
&\pa_t u + v \cdot \nabla u + \frac{\gamma}{\gamma-1}\nabla \lt(n^{\theta(\gamma-1)}\rt) +  n^2 Lu + \frac{\delta}{\delta-1}\nabla \lt(n^2\rt)\cdot \ml v = -n^{\theta(m-1)}\int_{\R^3} (v-\xi)f\,d\xi.
\end{aligned}
\eq

\begin{proof}[Proof of Theorem \ref{thm_lin}]
For the well-posedness, most of the proof follows the arguments in \cite[Lemma 3.5]{LPZ19}. Thus we only present the sketch of the proof.\\

\noindent $\diamond$ (Step A: Existence) First, we choose $T^* := T_3$ in Lemma  \ref{lem_u}. Then, we can obtain the following strong convergence up to a subsequence: for any $R>0$,
\[
(n^\eta, u^\eta) \to (n, u), \quad \mbox{in} \quad \mc([0,T^*], H^2(B_R)) \quad \mbox{as} \quad \eta \to 0.
\]
Here $B_R:= \{x \in \R^3 : |x| \leq R\}$. Moreover, the uniform-in-$\eta$ bounds imply the following weak (star) convergences up to a subsequence:
\begin{align*}
f^\eta &\rightharpoonup f\quad \mbox{weakly *} \quad \mbox{in } \ L^\infty(0,T^*;H^{2,p}_\nu(\R^3)),\\
\pa_t f^\eta &\rightharpoonup \pa_t f\quad \mbox{weakly *} \quad \mbox{in } \ L^\infty(0,T^*;H^{1,p-2}_\nu(\R^3)),\\
(n^\eta, u^\eta) &\rightharpoonup (n,u) \quad \mbox{weakly *} \quad \mbox{in } \ L^\infty(0,T^*;H^3(\R^3)),\\
\pa_t n^\eta &\rightharpoonup \pa_t n \quad \mbox{weakly *} \quad \mbox{in } \ L^\infty(0,T^*;H^1(\R^3)),\\
\pa_t u^\eta & \rightharpoonup \pa_t u  \quad \mbox{weakly} \quad \mbox{in } \ L^2(0,T^*;D^2(\R^3)), \quad \mbox{and}\\
n^\eta \nabla^4 u^\eta & \rightharpoonup n\nabla^4 u  \quad \mbox{weakly} \quad \mbox{in } \ L^2(\R^3 \times (0,T^*)).
\end{align*}
Thus, we can find a limit $(f,n,u)$ satisfying
\[
\begin{aligned}
& f \geq 0, \quad f \in L^\infty(0,T^*; H^{2,p}_\nu(\R^3 \times \R^3)) \quad \pa_t f \in L^\infty(0,T^*; H^{1,p-2}_\nu(\R^3 \times \R^3)), \cr
&  n - n_\infty \in L^\infty(0,T^*;H^3(\R^3)), \quad \pa_t n \in L^\infty(0,T^*;H^2(\R^3)), \cr
& u \in  L^\infty(0,T^*;H^3(\R^3)), \quad \mbox{and} \quad \pa_t u \in L^\infty(0,T^*;H^1(\R^3)) \cap L^2(0,T^*;D^2(\R^3)),
\end{aligned}
\]
Here, it is not difficult to show that $(f,n,u)$ satisfies \eqref{lin_sys2} in the sense of distribution. Furthermore, the arguments \cite[Lemma 3.5]{LPZ19} give us
\[
n - n_\infty \in \mc([0,T^*];H^3(\R^3)), \quad \pa_t n \in \mc([0,T^*];H^2(\R^3)), \quad u \in \mc([0,T^*];H^{s'}(\R^3)) \cap L^\infty(0,T^*;H^3(\R^3)),
\]
and
\[
\pa_t u \in \mc([0,T^*];H^1(\R^3)) \cap L^2(0,T^*;\dot{H}^2(\R^3)).
\]
For the time continuity of $f$ and $\pa_t f$, Sobolev embedding theorem with time and regularities of $(n,u)$ imply
\[
f \in \mc ([0,T^*]; H^{1,p}_\nu(\R^3 \times \R^3)), \quad \pa_t f \in \mc([0,T^*]; L^{2,p-2}_\nu(\R^3 \times \R^3)),
\]
\[
\nabla_{(x,\xi)}^2 f\in \mc_w([0,T^*]; L^{2,p}_\nu(\R^3\times\R^3)), \quad \mbox{and} \quad \pa_t \nabla_{(x,\xi)}f \in \mc_w([0,T^*]; L^{2,p-2}_\nu(\R^3\times\R^3)).
\]
To obtain the desired regularity, we recall the estimates in Lemma \ref{lem_f} and uniform bounds for $(f^\eta, n^\eta, u^\eta)$ to get
\[
\frac{d}{dt}\| f\|_{H^{2,p}_\nu}^2 \le C(1+\|n\|_{L^\infty})^{m\theta} (1+\|\nabla n\|_{H^2})^2(1+\|v\|_{H^3})^2\|f\|_{H^{2,p}_\nu}^2 \le C\| f\|_{H^{2,p}_\nu}^2.
\]
This implies
\[
\| f(\cdot,\cdot,t)\|_{H^{2,p}_\nu}^2 \le \| f_0\|_{H^{2,p}_\nu}^2 e^{Ct},
\]
and gives
\[
\limsup_{t \to 0} \|f(\cdot,\cdot,t)\|_{H^{2,p}_\nu}^2 \le \|f_0\|_{H^{2,p}_\nu}^2.
\]
Thus, $f$ is right continuous at $t=0$ in $H^{2,p}_\nu$ space and from the time reversibility of the kinetic equation in \eqref{lin_sys2}, we finally get
\[
f \in \mc([0,T^*];H^{2,p}_\nu(\R^3 \times \R^3)).
\]
This subsequently implies $\pa_t f \in \mc([0,T^*],H^{1,p-2}_\nu(\R^3 \times \R^3))$, and hence the existence theory is now established.\\

\noindent $\diamond$ (Step B: Uniqueness) Let $(f_1, n_1, u_1)$ and $(f_2, n_2, u_2)$ be two solutions to \eqref{lin_sys2} corresponding to the same initial data. For notational simplicity, we set $\bn := n_1 - n_2$, $\bu := u_1 - u_2$ and ${\bf F} := f_1 - f_2$. Then we have
\begin{align*}
&\pa_t {\bf F} + \xi \cdot \nabla {\bf F} + \nabla_\xi\cdot \lt( (n_1^{m\theta} - n_2^{m\theta})(v-\xi)f_1\rt) + \nabla_{\xi}\cdot (n_2^{m\theta}(v-\xi){\bf F})=0,\\
&\pa_t \bn + v \cdot \nabla \bn = 0,\\
&\pa_t \bu + v \cdot \nabla \bu  + \frac{\gamma}{\gamma-1}\nabla( n_1^{\theta(\gamma-1)} - n_2^{\theta(\gamma-1)}) + (n_1^2 - n_2^2) Lu_1 + n_2^2 L\bu + \nabla( n_1^2 - n_2^2) \ml v\\
&\hspace{3cm} = -(n_1^{\theta(m-1)}-n_2^{\theta(m-1)})\intr (v-\xi)f_1\,d\xi -n_2^{\theta(m-1)} \intr (v-\xi){\bf F}\,d\xi.
\end{align*}
Let us first start with the estimate of ${\bf F}$.
\begin{align*}
&\frac12\frac{d}{dt}\|{\bf F}\|_{L^{2,{p-2}}_\nu}^2 \cr
&\quad =  \frac{p-2}{2}\intrr \nu_{p-4}(x,\xi) x \cdot \xi {\bf F}^2\,dxd\xi\cr
&\qquad + 3 \intrr n_2^{m\theta}  \nu_{p-2}(x,\xi) {\bf F}^2 \,dxd\xi + \frac12 \intrr n_2^{m\theta} \nabla_\xi \cdot (\nu_{p-2}(x,\xi)(v - \xi)){\bf F}^2\,dxd\xi\cr
&\qquad + 3\intrr (n_1^{m\theta}-n_2^{m\theta}) \nu_{p-2}(x,\xi) f_1\,  {\bf F}\,dxd\xi    - \intrr (n_1^{m\theta}- n_2^{m\theta})\nu_{p-2}(x,\xi) (v-\xi)\cdot (\nabla_\xi f_1) \, {\bf F}\,dxd\xi \\
&\quad \ls (1+\|n\|_{L^\infty}^{m\theta})(1+\|v\|_{L^\infty})\|{\bf F}\|_{L^{2,{p-2}}_\nu}^2 + \|n_1^{m\theta}-n_2^{m\theta}\|_{L^6} \lt\|\lt(\intr \nu_{p-2}(x,\xi) f_1^2\,d\xi \rt)^{1/2}\rt\|_{L^3}\|{\bf F}\|_{L^{2,p-2}_\nu} \cr
&\qquad +  \|n_1^{m\theta}-n_2^{m\theta}\|_{L^6}\|v\|_{L^\infty} \lt\|\lt(\intr \nu_{p-2}(x,\xi) | \nabla_\xi f_1|^2\,d\xi \rt)^{1/2}\rt\|_{L^3}\|{\bf F}\|_{L^{2,p-2}_\nu}\\
&\qquad +  \|n_1^{m\theta}-n_2^{m\theta}\|_{L^6} \lt\|\lt(\intr \nu_p(x,\xi) | \nabla_\xi f_1|^2\,d\xi \rt)^{1/2}\rt\|_{L^3}\|{\bf F}\|_{L^{2,p-2}_\nu}\\
&\quad \ls \|{\bf F}\|_{L^{2,p-2}_\nu}^2 + \|\bn\|_{H^1}^2,
\end{align*}
where we used the estimates for \eqref{new_g}.\\

\noindent For $\bn$, direct calculation implies
\[
\frac{d}{dt}\|\bn\|_{H^1}^2 \ls \|v\|_{H^3}\|\bn\|_{H^1}^2 \ls \|\bn\|_{H^1}^2.
\]
For the estimate of $\bu$, we find
\begin{align*}
\frac12\frac{d}{dt}\|\bu\|_{L^2}^2 &= -\intr (v \cdot \nabla \bu)\cdot \bu \,dx -\frac{\gamma}{\gamma-1} \intr \nabla(n_1^{\theta(\gamma-1)} - n_2^{\theta(\gamma-1)}) \cdot \bu\,dx\\
&\quad - \intr (n_1^2 - n_2^2)Lu_1 \cdot \bu\,dx - \intr n_2^2 L\bu \cdot \bu\,dx\\
&\quad - \intr \nabla(n_1^2 - n_2^2) \ml v \cdot \bu\,dx - \intrr (n_1^{\theta(m-1)}-n_2^{\theta(m-1)})(v-\xi)f_1 \cdot \bu\,dxd\xi\\
&\quad -\intrr n_2^{\theta(m-1)}(v-\xi){\bf F} \cdot \bu\,dxd\xi\\
&=: \sum_{i=1}^7 \sfM_i.
\end{align*}
Here $\sfM_i, i\neq 4$ can be easily estimated as
\begin{align*}
\sfM_1 &\ls \|\nabla v\|_{L^\infty}\|\bu\|_{L^2}^2 \ls \|\bu\|_{L^2}^2,\\
\sfM_2 &\ls \|n_1^{\theta(\gamma-1)-1} - n_2^{\theta(\gamma-1)-1}\|_{L^2}\|\nabla n_1\|_{L^\infty}\|\bu\|_{L^2} + \|n_2\|_{L^\infty}^{\theta(\gamma-1)-1} \|\nabla \bn\|_{L^2} \|\bu\|_{L^2}\ls \|\bn\|_{H^1}\|\bu\|_{L^2},\\
\sfM_3 &\ls (\|n_1\|_{L^\infty} + \|n_2\|_{L^\infty}) \|\bn\|_{L^6} \|\nabla^2 u_1\|_{L^3} \|\bu\|_{L^2} \ls \|\bn\|_{H^1}\|\bu\|_{L^2},\\
\sfM_5 &\ls \|\bn\|_{L^2}(\|\nabla n_1\|_{L^\infty}+\|\nabla n_2\|_{L^\infty})\|\ml v\|_{L^\infty}\|\bu\|_{L^2} + \|n_2\|_{L^\infty} \|\nabla \bn\|_{L^2} \|\ml v\|_{L^\infty}\|\bu\|_{L^2}\ls \|\bn\|_{H^1}\|\bu\|_{L^2},\\
\sfM_6 &\ls \|\|n_1^{\theta(\gamma-1)-1} - n_2^{\theta(\gamma-1)-1}\|_{L^6} \lt(\|v\|_{L^\infty}\lt\|\intr f_1\,d\xi \rt\|_{L^3} + \lt\|\intr |\xi| f_1\,d\xi \rt\|_{L^3}\rt) \|\bu \|_{L^2} \ls \|\bn\|_{H^1}\|\bu\|_{L^2},\\
\sfM_7 &\ls \|n_2\|_{L^\infty}^{\theta(m-1)}(\|v\|_{L^\infty}+1) \lt(\|v\|_{L^\infty}\lt\|\intr {\bf F}\,d\xi \rt\|_{L^2} + \lt\|\intr |\xi| {\bf F}\,d\xi \rt\|_{L^2}\rt) \|\bu\|_{L^2} \ls \|{\bf F}\|_{L^{2,p-2}_\nu}\|\bu\|_{L^2}, 
\end{align*}
where we used
\[
|n_1^{\delta} - n_2^{\delta}| \ls (n_1^{\delta-1} + n_2^{\delta-1}) |n_1 - n_2|, \quad \delta>1.
\]
For the term $\sfM_4$, we get
\begin{align*}
\sfM_4&= -\alpha\intr n_2 (\nabla n_2 \cdot \nabla \bu) \cdot \bu\,dx - (\alpha+\beta) \intr n_2 (\nabla n_2 \cdot \bu) (\nabla \cdot \bu)\,dx\\
&\quad - \alpha\intr n_2^2 |\nabla \bu|^2\,dx - (\alpha+\beta) \intr n_2^2 (\nabla \cdot \bu)^2\,dx\\
&\le \alpha \|n_2 \nabla \bu\|_{L^2}\|\nabla n_2\|_{L^\infty}\|\bu\|_{L^2} + (\alpha+\beta) \|n_2 (\nabla \cdot \bu)\|_{L^2}\|\nabla n_2\|_{L^\infty}\|\bu\|_{L^2}\\
&\quad - \alpha\intr n_2^2 |\nabla \bu|^2\,dx - (\alpha+\beta) \intr n_2^2 (\nabla \cdot \bu)^2\,dx\\
&\le - \frac\alpha2\intr n_2^2 |\nabla \bu|^2\,dx - \frac{\alpha+\beta}{2} \intr n_2^2 (\nabla \cdot \bu)^2\,dx+ C\|\bu\|_{L^2}^2.
\end{align*}
Thus, by combining the estimates of $\sfM_i$, we deduce
\[
\frac{d}{dt}\|\bu\|_{L^2}^2 \ls \|u\|_{L^2}^2 + \|\bn \|_{H^1}^2 + \|{\bf F}\|_{L^{2,p-2}_\nu}^2.
\]
We finally gather the estimates for $({\bf F}, \bn, \bu)$ to yield
\[
\frac{d}{dt}\lt( \|{\bf F}\|_{L^{2,p-2}_\nu}^2 + \|\bn\|_{H^1}^2 + \|\bu\|_{L^2}^2\rt) \ls  \|{\bf F}\|_{L^{2,p-2}_\nu}^2 + \|\bn\|_{H^1}^2 + \|\bu\|_{L^2}^2,
\]
and we use Gr\"onwall's lemma to conclude the uniqueness of solutions.
\end{proof}

%
%
%
%

\section{Local well-posedness of regular solutions}\label{sec:nlin_ext}
In this section, we prove the local well-posedness of regular solutions to system \eqref{main_sys} based on the estimates for the linearized system from the previous section.

%
%

\subsection{Contruction of approximate solutions}
To show the well-posedness, we first let
\[
2+ n_\infty + \|n_0 - n_\infty\|_{H^3} + \|u_0\|_{H^3} + \|f_0\|_{H^{2,p}_\nu} \le \epsilon_0,
\]
and define $\epsilon_1$, $\epsilon_2$, $\epsilon_3$ and $\epsilon_4$ as in the previous section. Now, we consider a sequence $\{f^k, n^k, u^k\}$ of approximate solutions which satisfies 
\bq\label{cauchy}
\begin{aligned}
&\pa_t f^{k+1} + \xi \cdot \nabla f^{k+1} + \nabla_\xi \cdot \lt( (n^{k+1})^{m\theta} (u^k - \xi)f^{k+1} \rt) = 0, \cr
&\pa_t n^{k+1} + u^k\cdot \nabla n^{k+1} + \theta^{-1}n^k \nabla \cdot u^k = 0,\cr
&\pa_t u^{k+1} + u^k \cdot \nabla u^{k+1} + \frac{\gamma}{\gamma-1}\nabla \lt((n^{k+1})^{\theta(\gamma-1)}\rt) +  (n^{k+1})^2 Lu^{k+1} + \frac{\delta}{\delta-1}\nabla \lt((n^{k+1})^2\rt)\cdot \ml u^k \\
&\hspace{6cm}= -(n^{k+1})^{\theta(m-1)}\int_{\R^3} (u^k-\xi)f^{k+1}\,d\xi,
\end{aligned}
\eq
subject to initial data and far-field behavior:
\[
(f^{k+1}(x,\xi,0), n^{k+1}(x,0), u^{k+1}(x,0)) = (f_0(x,\xi), n_0(x), u_0(x)),
\]
\[
f^{k+1}(x,\xi,t) \to 0, \quad (n^{k+1}, u^{k+1}) \to (n_\infty,0),\quad \mbox{as} \quad |x|, |\xi|\to\infty.
\]
Here, the initial step $(f^0, n^0, u^0)$ satisfies the following system:
\begin{align*}
\begin{aligned}
&\pa_t f^0 + \xi \cdot \nabla f^0 + \nabla_\xi \cdot \lt( (n_0)^{m\theta} (u_0 - \xi)f^0 \rt) = 0, \cr
&\pa_t n^0 + u_0\cdot \nabla n^0= 0, \cr
&\pa_t u^0 -n^2 \Delta u^0= 0,
\end{aligned}
\end{align*}
subject to initial data and far-field behavior:
\[
(f^0(x,\xi,0), n^0(x,0), u^0(x,0)) = (f_0(x,\xi), n_0(x), u_0(x)), \quad (x,\xi) \in \R^3 \times \R^3
\]
and
\[
f^0(x,\xi,t) \to 0, \quad (n^0, u^0) \to (n_\infty,0),\quad \mbox{as} \quad |x|, \ |\xi|\to\infty,
\]
respectively. From the estimates in Section \ref{sec:lin_ext}, we can obtain the followings: for any $k\in\N$,  there exists ${\tilde T}\in (0,T^*]$ such that the sequence $\{f^k, n^k, u^k\}$ satisfies 
\[
\begin{aligned}
& f^k\geq 0, \quad f^k \in \mc([0,{\tilde T}]; H^{2,p}_\nu(\R^3 \times \R^3)) \quad \pa_t f^k \in \mc([0,{\tilde T}]; H^{1,p-2}_\nu(\R^3 \times \R^3)), \cr
&  n^k - n_\infty \in \mc([0,{\tilde T}];H^3(\R^3)), \quad \pa_t n^k \in \mc([0,{\tilde T}];H^2(\R^3)), \cr
& u^k \in \mc([0,{\tilde T}];H^{s'}(\R^3)) \cap L^\infty(0,{\tilde T};H^3(\R^3)), \quad \mbox{and} \quad \pa_t u^k \in \mc([0,{\tilde T}];H^1(\R^3)) \cap L^2(0,{\tilde T};\dot{H}^2(\R^3)),
\end{aligned}
\]
 such that for any $k\in\N$,
\begin{align*}
\begin{aligned}
&\sup_{0 \leq t \leq {\tilde T}}\lt(\|n^k(\cdot,t)-n_\infty\|_{H^3}^2 + \|u^k(\cdot,t)\|_{H^1}^2 + \|f^k(\cdot,\cdot,t)\|_{H^{2,p}_\nu}^2 \rt)\leq \epsilon_1^2,\cr
&\sup_{0 \leq t \leq {\tilde T}}\lt(\|\pa_t n^k(\cdot,t)\|_{L^2}^2 + \|\nabla^2 u^k(\cdot,t)\|_{L^2}^2  + \|\pa_t u^k(\cdot,t)\|_{L^2}^2 + \|\pa_t f^k(\cdot,\cdot,t)\|_{L^{2,p-2}_\nu}^2\rt)\\
&\hspace{4cm}+ \int_0^{{\tilde T}} \lt(\|n(\cdot,t)\nabla^3 u^k(\cdot,t)\|_{L^2}^2 + \|\pa_t\nabla u^k(\cdot,t)\|_{L^2}^2 \rt)dt \leq \epsilon_2^2,\cr
&\sup_{0 \leq t \leq {\tilde T}}\lt(\|\pa_t \nabla n^k(\cdot,t)\|_{L^2}^2 + \|\nabla^3 u^k(\cdot,t)\|_{L^2}^2 + \|\pa_t \nabla u^k(\cdot,t)\|_{L^2}^2 +\|\pa_t \nabla_{(x,\xi)} f^k(\cdot,\cdot,t)\|_{L^{2,p-2}_\nu}^2\rt)\\
&\hspace{4cm}+ \int_0^{{\tilde T}} \lt(\|n(\cdot,t)\nabla^4 u^k(\cdot,t)\|_{L^2}^2 + \|\pa_t \nabla^2 u^k(\cdot,t)\|_{L^2}^2 \rt)dt \leq \epsilon_3^2, \quad \mbox{and}\cr
&\sup_{0 \leq t \leq {\tilde T}} \|\pa_t \nabla^2 n^k(\cdot,t)\|_{L^2}^2 \leq \epsilon_4^2.\cr
\end{aligned}
\end{align*}

%
%
\subsection{Cauchy estimates}
For the approximate solutions we constructed, we can show that $\{f^k, n^k, u^k\}$ is a Cauchy sequence in $\mc([0,{\tilde T}];L^{2,p-2}_\nu(\R^3 \times \R^3)) \times \mc([0,{\tilde T}];H^1(\R^3)) \times\mc([0,{\tilde T}];H^1(\R^3))$. We would like to point out that in \cite{LPZ19}, it suffices to show that $\{n^k, u^k\}$ is Cauchy in $\mc([0,{\tilde T}];L^2(\R^3)) \times\mc([0,{\tilde T}];L^2(\R^3))$. However, due to the presence of the drag force in the fluid equations, that is not enough to close the Cauchy estimates. The lemma below clearly shows that we should further estimate the sequence $\{n^k, u^k\}$ in the higher-order Sobolev space to have the required Cauchy estimates. 

\begin{lemma}\label{cauchy_f} Let $\{f^k, n^k, u^k\}$ be the solutions to \eqref{cauchy} on the time interval $[0,{\tilde T}]$. Then, for $t \in [0,{\tilde T}]$, we have
\begin{align*}
\frac{d}{dt}\|f^{k+1}-f^k\|_{L^{2,p-2}_\nu}^2 &\le C(\|f^{k+1}-f^k\|_{L^{2,p-2}_\nu}^2 +\|n^{k+1}-n^k\|_{H^1}^2 + \|u^k - u^{k-1}\|_{H^1}^2),
\end{align*}
where $C>0$ is independent of $k$.
\end{lemma}
\begin{proof}
Straightforward computations give
\begin{align*}
\frac12\frac{d}{dt}\|f^{k+1} - f^k\|_{L^{2,p-2}_\nu}^2 &= - \intrr \nu_{p-2}(x,\xi) \nabla \cdot \lt(\xi (f^{k+1} - f^k)\rt) (f^{k+1} - f^k)\,dxd\xi\\
&\quad -\intrr \nu_{p-2}(x,\xi) \nabla_\xi \cdot \lt(((n^{k+1})^{m\theta} - (n^k)^{m\theta})(u^k - \xi)f^{k+1}\rt) (f^{k+1} - f^k)\,dxd\xi\\
&\quad - \intrr \nu_{p-2}(x,\xi) \nabla_\xi \cdot \lt((n^k)^{m\theta}(u^k - u^{k-1}) f^{k+1}\rt) (f^{k+1} - f^k)\,dxd\xi\\
&\quad - \intrr \nu_{p-2}(x,\xi) \nabla_\xi \cdot \lt((n^k)^{m\theta}(u^{k-1}-\xi ) (f^{k+1}-f^k)\rt) (f^{k+1} - f^k)\,dxd\xi\\
&=: \sum_{i=1}^4 \sfI_i,
\end{align*}
where $\sfI_1$ and $\sfI_2$ can be estimated as  
$$\begin{aligned}
\sfI_1 =  \frac{p-2}{2} \intrr \nu_{p-4}(x,\xi) (x \cdot \xi) (f^{k+1} - f^k)^2\,dxd\xi \ls \|f^{k+1} - f^k\|_{L^{2,p-2}_\nu}^2
\end{aligned}$$
and
\begin{align*}
\sfI_2 &= -\intrr \nu_{p-2}(x,\xi) ((n^{k+1})^{m\theta} - (n^k)^{m\theta})  (f^{k+1} - f^k)\lt(-3f^{k+1} + (u^k - \xi) \cdot \nabla_\xi f^{k+1} \rt)\,dxd\xi\\
&\ls \|(n^{k+1})^{m\theta} - (n^k)^{m\theta}\|_{L^6}\lt\|\lt( \intr \nu_{p-2}(x,\xi) (f^{k+1})^2\,d\xi\rt)^{1/2}\rt\|_{L^3} \|f^{k+1} -f^k\|_{L^{2,p-2}_\nu} \\
&\quad+ \|(n^{k+1})^{m\theta} - (n^k)^{m\theta}\|_{L^6}  \|u^k\|_{L^\infty}\lt\|\lt(\intr \nu_{p-2}(x,\xi) |\nabla_\xi f^{k+1}|^2\,d\xi\rt)^{1/2} \rt\|_{L^3}   \|f^{k+1} -f^k\|_{L^{2,p-2}_\nu}\\
&\quad+ \|(n^{k+1})^{m\theta} - (n^k)^{m\theta}\|_{L^6} \lt\|\lt(\intr \nu_p(x,\xi) |\nabla_\xi f^{k+1}|^2\,d\xi\rt)^{1/2} \rt\|_{L^3} \|f^{k+1} -f^k\|_{L^{2,p-2}_\nu}\\
&\ls \|\nabla(n^{k+1} -n^k)\|_{L^2}\|f^{k+1} - f^k\|_{L^{2,p-2}_\nu}.
\end{align*}
Here we used the estimates for \eqref{new_g} and
\bq\label{ineq_n}
|(n^{k+1})^{r} - (n^k)^{r}| \ls ((n^{k+1})^{r-1} + (n^k)^{r-1}) |n^{k+1} - n^k|, \quad r>1.
\eq
We also bound the terms $\sfI_3$ and $\sfI_4$ as
\begin{align*}
\sfI_3 &\ls \|n^k\|_{L^\infty}^{m\theta} \|u^k - u^{k-1}\|_{L^6}\lt\|\lt(\intr \nu_{p-2}(x,\xi) |\nabla_\xi f^{k+1}|^2\,d\xi\rt)^{1/2} \rt\|_{L^3}  \|f^{k+1} -f^k\|_{L^{2,p-2}_\nu}\\
&\ls \|\nabla(u^k - u^{k-1})\|_{L^2}\|f^{k+1} -f^k\|_{L^{2,p-2}_\nu}
\end{align*}
and
\begin{align*}
\sfI_4 &= - \intrr \nu_{p-2}(x,\xi) (n^k)^{m\theta}  (f^{k+1} - f^k)\lt(-3(f^{k+1} - f^k) + (u^k - \xi) \cdot \nabla_\xi (f^{k+1} - f^k) \rt)\,dxd\xi\\
&= 3 \intrr \nu_{p-2}(x,\xi) (n^k)^{m\theta} (f^{k+1} - f^k)^2\,dxd\xi \cr
&\quad + \frac12 \intrr (n^k)^{m\theta} \nabla_\xi \cdot (\nu_{p-2}(x,\xi) (u^k - \xi)) (f^{k+1} - f^k)^2\,dxd\xi\cr
&\ls \|n^k\|_{L^\infty}^{m\theta} (\| u^k\|_{L^\infty} + 1) \|f^{k+1} - f^k\|_{L^{2,p-2}_\nu}.
\end{align*}
Combining the above estimates for $\sfI_i, i=1,\dots,4$ concludes the desired result.
\end{proof}
In the following two lemmas, we provide the $H^1$-estimate of $n^{k+1}-n^k$ and $u^{k+1}-u^k$.
\begin{lemma}\label{cauchy_n}
Let $\{f^k, n^k, u^k\}$ be the solutions to \eqref{cauchy} on the time interval $[0,{\tilde T}]$. Then, for $t \in [0,{\tilde T}]$, we have
\begin{align*}
&\frac{d}{dt}\|n^{k+1}-n^k\|_{H^1}^2 \cr
&\quad \le C\lt( \|n^{k+1}-n^k\|_{H^1}^2 + \|n^k - n^{k-1}\|_{H^1}^2 + \|u^k-u^{k-1}\|_{H^1}^2 \rt) \\
&\qquad +C\lt(\|n^k \nabla(u^k - u^{k-1})\|_{L^2} \|n^{k+1} - n^k\|_{L^2} + \|n^k \nabla^2(u^k - u^{k-1})\|_{L^2} \|\nabla(n^{k+1} - n^k)\|_{L^2}\rt),
\end{align*}
where $C>0$ is independent of $k$.
\end{lemma}
\begin{proof}
Direct computation yields
\begin{align*}
\frac12&\frac{d}{dt}\|n^{k+1}-n^k\|_{L^2}^2 \\
&= -\intr u^k \cdot \nabla (n^{k+1}- n^k) (n^{k+1}-n^k)\,dx - \intr (u^k - u^{k-1}) \cdot \nabla n^k (n^{k+1}-n^k)\,dx\\
&\quad - \theta^{-1}\intr (n^k - n^{k-1}) (\nabla \cdot u^{k-1})(n^{k+1}-n^k)\,dx  - \theta^{-1} \intr n^k \nabla \cdot (u^k - u^{k-1}) (n^{k+1} - n^k)\,dx\\
&\ls  \|\nabla \cdot u^k\|_{L^\infty}\|n^{k+1}-n^k\|_{L^2}^2 + \|\nabla n^k\|_{L^\infty} \|u^k-u^{k-1}\|_{L^2} \|n^{k+1}-n^k\|_{L^2}\\
&\quad + \|\nabla \cdot u^{k-1}\|_{L^\infty} \|n^k - n^{k-1}\|_{L^2} \|n^{k+1} - n^k\|_{L^2} + \|n^k\|_{L^\infty} \|\nabla(u^k - u^{k-1})\|_{L^2} \|n^{k+1} - n^k\|_{L^2}\\
&\ls \|n^{k+1}-n^k\|_{L^2}^2 + \|n^k - n^{k-1}\|_{L^2}^2 + \|u^k-u^{k-1}\|_{H^1}^2,
\end{align*}
where we used Young's inequality.\\

\noindent For $H^1$-estimate, we get
\begin{align*}
\frac12&\frac{d}{dt}\|\pa(n^{k+1}-n^k)\|_{L^2}^2 \\
&= -\intr u^k \cdot \nabla \pa (n^{k+1}- n^k) \pa(n^{k+1}-n^k)\,dx -\intr \pa u^k \cdot \nabla  (n^{k+1}- n^k) (n^{k+1}-n^k)\,dx\\
&\quad - \intr (u^k - u^{k-1}) \cdot \nabla \pa n^k \pa (n^{k+1}-n^k)\,dx - \intr \pa (u^k - u^{k-1}) \cdot \nabla n^k \pa (n^{k+1}-n^k)\,dx\\
&\quad - \theta^{-1}\intr (n^k - n^{k-1}) (\nabla \cdot \pa u^{k-1})\pa (n^{k+1}-n^k)\,dx - \theta^{-1}\intr \pa (n^k - n^{k-1}) (\nabla \cdot u^{k-1})\pa (n^{k+1}-n^k)\,dx\\
&\quad - \theta^{-1} \intr n^k \nabla \cdot \pa (u^k - u^{k-1}) \pa(n^{k+1} - n^k)\,dx - \theta^{-1} \intr \pa n^k \nabla \cdot (u^k - u^{k-1}) \pa (n^{k+1} - n^k)\,dx\\
&\ls \|\nabla u^k\|_{L^\infty}\|\pa(n^{k+1} -n^k)\|_{L^2}^2 + \|u^k-u^{k-1}\|_{L^6} \|\nabla \pa n^k\|_{L^3}\|\pa(n^{k+1}-n^k)\|_{L^2}\\
&\quad +\|\pa (u^k-u^{k-1})\|_{L^2}\|\nabla n^k\|_{L^\infty}\|\pa(n^{k+1}-n^k)\|_{L^2} + \|n^k-n^{k-1}\|_{L^6}\|\nabla \cdot \pa u^{k-1}\|_{L^3}\|\pa(n^{k+1}-n^k)\|_{L^2}\\
&\quad + \|\pa(n^k-n^{k-1})\|_{L^2}\|\nabla \cdot u^{k-1}\|_{L^\infty}\|\pa(n^{k+1}-n^k)\|_{L^2} + \|n^k \nabla^2(u^k - u^{k-1})\|_{L^2}\|\pa(n^{k+1}-n^k)\|_{L^2}  \\
&\quad + \|\pa n^k \|_{L^\infty}\|\nabla\cdot(u^k - u^{k-1})\|_{L^2}\|\pa(n^{k+1}-n^k)\|_{L^2} \\
&\ls \|\nabla(n^{k+1} - n^k)\|_{L^2}^2 + \|\nabla(n^k - n^{k-1})\|_{L^2}^2 + \|\nabla(u^k - u^{k-1})\|_{L^2}^2 \cr
&\quad + \|n^k\nabla^2(u^k - u^{k-1})\|_{L^2}\|\pa(n^{k+1}-n^k)\|_{L^2}.
\end{align*}
Thus, we combine this with the $L^2$-estimate to complete the proof.
\end{proof}
\begin{lemma}\label{cauchy_u}
Let $\{f^k, n^k, u^k\}$ be the solutions to \eqref{cauchy} on the time interval $[0,{\tilde T}]$. Then, for $t \in [0,{\tilde T}]$, we have
\[
\begin{aligned}
\frac{d}{dt}&\|u^{k+1} - u^k\|_{H^1}^2 + \intr ((n^k)^2 + (n^{k+1})^2)|\nabla^2(u^{k+1} - u^k)|^2\,dx\\
&\le C(\|u^{k+1} - u^k\|_{H^1}^2 + \|u^k-u^{k-1}\|_{H^1}^2 + \|n^{k+1} - n^k\|_{H^1}^2 + \|f^{k+1} -f^k\|_{L^{2,2}_\nu}^2),
\end{aligned}
\]
where $C>0$ is independent of $k$.
\end{lemma}
\begin{proof}
We calculate $L^2$- and $H^1$-estimates step by step.\\

\noindent $\bullet$ (Step A: $L^2$-estimates): First, we get
\begin{align*}
\frac12\frac{d}{dt}\|u^{k+1} -u^k\|_{L^2}^2 &= -\intr (u^k \cdot \nabla(u^{k+1} - u^k) )\cdot (u^{k+1} - u^k)\,dx \\
&\quad-\intr ((u^k - u^{k-1}) \cdot \nabla u^k  )\cdot (u^{k+1} -u^k)\,dx\\
&\quad -\frac{\gamma}{\gamma-1}\intr\nabla((n^{k+1})^{\theta(\gamma-1)} - (n^k)^{\theta(\gamma-1)})\cdot (u^{k+1} -u^k)\,dx \\
&\quad- \intr ((n^{k+1})^2 Lu^{k+1} - (n^k)^2 Lu^k)\cdot (u^{k+1} -u^k)\,dx\\
&\quad -\frac{\delta}{\delta-1} \intr (\nabla ((n^{k+1})^2) \ml u^k - \nabla((n^k)^2)\ml u^{k-1})\cdot (u^{k+1} -u^k)\,dx\\
&\quad - \intrr ((n^{k+1})^{\theta(m-1)} - (n^k)^{\theta(m-1)})(u^k - \xi)f^{k+1} \cdot (u^{k+1} -u^k)\, dxd\xi\\
&\quad -\intrr (n^k)^{\theta(m-1)} (u^k -u^{k-1}) f^{k+1} \cdot (u^{k+1} -u^k)\,dxd\xi\\
&\quad - \intrr (n^k)^{\theta(m-1)}(u^{k-1}-\xi)(f^{k+1} -f^k)\cdot (u^{k+1} -u^k)\,dxd\xi\\
&=: \sum_{i=1}^8 \sfJ_i.
\end{align*}
\noindent $\diamond$ (Estimates for $\sfJ_1$ and $\sfJ_2$): Here, we have

\begin{align*}
\sfJ_1 &\ls \|\nabla u^k\|_{L^\infty}\|u^{k+1} -u^k\|_{L^2}^2 \ls \|u^{k+1} -u^k\|_{L^2}^2,\\
\sfJ_2 &\ls \|u^k - u^{k-1}\|_{L^2} \|\nabla u^k\|_{L^\infty}\|u^{k+1} -u^k\|_{L^2} \ls \|u^k - u^{k-1}\|_{L^2}^2 +  \|u^{k+1} -u^k\|_{L^2}^2.\\
\end{align*}
\vspace{0.1cm}

\noindent $\diamond$ (Estimates for $\sfJ_3$): We use \eqref{ineq_n} and $\theta(\gamma-1)\ge 4$ to get

\begin{align*}
\sfJ_3 &= -\theta\gamma \intr ((n^{k+1})^{\theta(\gamma-1)-1} - (n^k)^{\theta(\gamma-1)-1})\nabla n^{k+1} \cdot (u^{k+1} -u^k)\,dx\\
&\quad - \theta\gamma\intr (n^k)^{\theta(\gamma-1)-1} \nabla(n^{k+1}-n^k) \cdot (u^{k+1} -u^k)\,dx\\
&\ls \|n^{k+1} - n^k\|_{H^1}\|u^{k+1}-u^k\|_{L^2} \ls \|n^{k+1} - n^k\|_{H^1}^2+ \|u^{k+1}-u^k\|_{L^2}^2.
\end{align*}

\noindent $\diamond$ (Estimates for $\sfJ_4$): We again use \eqref{ineq_n} to have
\begin{align*}
\sfJ_4&= -\intr ((n^{k+1})^2 - (n^k)^2) Lu^{k+1} \cdot (u^{k+1} -u^k)\,dx\\
&\quad - \intr (n^k)^2 L(u^{k+1} - u^k) \cdot (u^{k+1} -u^k)\,dx\\
&\ls \|n^{k+1} -n^k\|_{L^6}\|Lu^{k+1}\|_{L^3}\|u^{k+1}-u^k\|_{L^2} + \|n^k\|_{L^\infty}\|n^k \nabla^2(u^{k+1} -u^k)\|_{L^2}\|u^{k+1} - u^k\|_{L^2}\\
&\ls\|n^{k+1} -n^k\|_{H^1}\|u^{k+1}-u^k\|_{L^2} +\|n^k \nabla^2(u^{k+1} -u^k)\|_{L^2}\|u^{k+1} - u^k\|_{L^2}.
\end{align*}

\noindent $\diamond$ (Estimates for $\sfJ_5$) We calculate as follows:
\begin{align*}
\sfJ_5 &=-\frac{\delta}{\delta-1}\intr (n^{k+1}-n^k)(\nabla n^{k+1} \cdot \ml u^k) \cdot (u^{k+1}-u^k)\,dx\\
&\quad -\frac{\delta}{\delta-1}\intr n^k( \nabla(n^{k+1} -n^k) \cdot \ml u^k) \cdot (u^{k+1}- u^k)\,dx\\
&\quad -\frac{\delta}{\delta-1}\intr n^k (\nabla n^k \cdot \ml (u^k-u^{k-1}))\cdot(u^{k+1} - u^k)\,dx\\
&\ls \|n^{k+1}-n^k\|_{L^2} \|\nabla n^{k+1}\|_{L^\infty}\|\ml u^k\|_{L^\infty}\|u^{k+1}-u^k\|_{L^2}\\
&\quad + \|n^k\|_{L^\infty}\|\nabla(n^{k+1} - n^k)\|_{L^2} \|\ml u^k\|_{L^\infty}\|u^{k+1} - u^k\|_{L^2}\\
&\quad + \|n^k\|_{L^\infty}\|\nabla n^k\|_{L^\infty}\|\ml(u^k - u^{k-1})\|_{L^2}\|u^{k+1}-u^k\|_{L^2}\\
&\ls \|n^{k+1}-n^k\|_{H^1}^2 + \|u^{k+1}-u^k\|_{L^2}^2 + \|u^k-u^{k-1}\|_{H^1}^2.
\end{align*}

\noindent $\diamond$ (Estimates for $\sfJ_6$-$\sfJ_8$): In this case,
\begin{align*}
\sfJ_6 &\ls \|n^{k+1} - n^k\|_{L^6}\lt(\|u^k\|_{L^\infty}\lt\|\intr f^{k+1}\,d\xi \rt\|_{L^3} + \lt\|\intr |\xi| f^{k+1}\,d\xi \rt\|_{L^3}\rt)\|u^{k+1}-u^k\|_{L^2}\cr
&\ls \|n^{k+1}-n^k\|_{H^1}^2 + \|u^{k+1} - u^k\|_{L^2}^2,\\
\sfJ_7 &\ls \|n^k\|_{L^\infty}^{m\theta} \|u^k - u^{k-1}\|_{L^6} \lt\| \intr f^{k+1}\,d\xi\rt\|_{L^3}\|u^{k+1}-u^k\|_{L^2}\cr
& \ls \|u^k -u^{k-1}\|_{H^1}^2 + \|u^{k+1}-u^k\|_{L^2}^2,\\
\sfJ_8&\ls \|n^k\|_{L^\infty}^{m\theta}\lt(\|u^{k-1}\|_{L^\infty}\lt\|\intr (f^{k+1}- f^k)\,d\xi \rt\|_{L^2} + \lt\|\intr |\xi| (f^{k+1}- f^k)\,d\xi \rt\|_{L^2}\rt)  \|u^{k+1} - u^k\|_{L^2}\cr
&\ls \|f^{k+1}- f^k\|_{L^{2,p-2}_\nu}^2 + \|u^{k+1} - u^k\|_{L^2}^2.
\end{align*}
Then we gather the estimates for $\sfJ_i$'s to yield
\bq\label{l2_u_cauchy}
\begin{aligned}
\frac{d}{dt}\|u^{k+1} - u^k\|_{L^2}^2 &\le C(\|u^{k+1}-u^k\|_{L^2}^2 + \|n^{k+1}-n^k\|_{H^1}^2 + \|u^k - u^{k-1}\|_{H^1}^2 + \|f^{k+1} - f^k\|_{L^{2,p-2}_\nu}^2)\\
&\quad +C\|n^k\nabla^2 (u^{k+1} - u^k)\|_{L^2}\|u^{k+1}-u^k\|_{L^2}.
\end{aligned}
\eq

\noindent $\bullet$ (Step B: $H^1$-estimates): In this case,
\begin{align*}
\frac12\frac{d}{dt}\|\pa(u^{k+1} -u^k)\|_{L^2}^2 &= -\intr \pa(u^k \cdot \nabla(u^{k+1} - u^k) )\cdot \pa(u^{k+1} - u^k)\,dx \\
&\quad-\intr \pa((u^k - u^{k-1}) \cdot \nabla u^k  )\cdot \pa(u^{k+1} -u^k)\,dx\\
&\quad -\frac{\gamma}{\gamma-1}\intr\nabla\pa((n^{k+1})^{\theta(\gamma-1)} - (n^k)^{\theta(\gamma-1)})\cdot \pa(u^{k+1} -u^k)\,dx \\
&\quad- \intr \pa((n^{k+1})^2 Lu^{k+1} - (n^k)^2 Lu^k)\cdot \pa(u^{k+1} -u^k)\,dx\\
&\quad -\frac{\delta}{\delta-1} \intr \pa(\nabla ((n^{k+1})^2) \ml u^k - \nabla((n^k)^2)\ml u^{k-1})\cdot \pa(u^{k+1} -u^k)\,dx\\
&\quad - \intrr \pa(((n^{k+1})^{\theta(m-1)} - (n^k)^{\theta(m-1)})(u^k - \xi)f^{k+1} )\cdot \pa(u^{k+1} -u^k) dxd\xi\\
&\quad -\intrr\pa( (n^k)^{\theta(m-1)} (u^k -u^{k-1}) f^{k+1}) \cdot \pa(u^{k+1} -u^k)\,dxd\xi\\
&\quad - \intrr \pa((n^k)^{\theta(m-1)}(u^{k-1}-\xi)(f^{k+1} -f^k))\cdot \pa(u^{k+1} -u^k)\,dxd\xi\\
&=: \sum_{i=1}^8 \sfK_i.
\end{align*}

\noindent $\diamond$ (Estimates for $\sfK_1$ \& $\sfK_2$): We easily get
\begin{align*}
\sfK_1 &\ls \|\nabla u^k\|_{L^\infty}\|\nabla(u^{k+1} -u^k)\|_{L^2}^2 \ls \|\nabla(u^{k+1} -u^k)\|_{L^2}^2,\\
\sfK_2 &\ls \|u^k - u^{k-1}\|_{L^6}\|\nabla \pa u^k\|_{L^3} \|\pa(u^{k+1} - u^k)\|_{L^2} + \|\pa(u^k-u^{k-1})\|_{L^2} \|\nabla u^k\|_{L^\infty}\|\pa(u^{k+1} - u^k)\|_{L^2} \\
&\ls \|\nabla(u^k-u^{k-1})\|_{L^2}^2 + \|\nabla(u^{k+1} -u^k)\|_{L^2}^2.
\end{align*}

\noindent $\diamond$ (Estimates for $\sfK_3$): Now, we have
\begin{align*}
\sfK_3 &=-\theta\gamma\intr \pa\lt((n^{k+1})^{\theta(\gamma-1)-1}\nabla n^{k+1} - (n^k)^{\theta(\gamma-1)-1} \nabla n^k\rt)\cdot \pa (u^{k+1}-u^k)\,dx\\
&=-\theta\gamma\intr \pa\lt(((n^{k+1})^{\theta(\gamma-1)-1} - (n^k)^{\theta(\gamma-1)-1})\nabla n^{k+1} \rt)\cdot \pa(u^{k+1} - u^k)\,dx\\
&\quad -\theta\gamma \intr\pa\lt((n^k)^{\theta(\gamma-1)-1} \nabla(n^{k+1} - n^k) \rt) \cdot \pa(u^{k+1}-u^k)\,dx\\
&= -\theta\gamma(\theta(\gamma-1)-1)\intr ((n^{k+1})^{\theta(\gamma-1)-2} - (n^k)^{\theta(\gamma-1)-2})\pa n^{k+1} \nabla n^{k+1}\cdot\pa(u^{k+1}-u^k)\,dx\\
&\quad - \theta\gamma(\theta(\gamma-1)-1)\intr (n^k)^{\theta(\gamma-1)-2} \pa(n^{k+1} - n^k) \nabla n^{k+1} \cdot \pa(u^{k+1} - u^k)\,dx\\
&\quad - \theta\gamma \intr \lt(((n^{k+1})^{\theta(\gamma-1)-1} - (n^k)^{\theta(\gamma-1)-1})\nabla \pa n^{k+1} \rt)\cdot \pa(u^{k+1} - u^k)\,dx\\
&\quad + \theta\gamma\intr (n^k)^{\theta(\gamma-1)-1}\nabla(n^{k+1}-n^k)\pa^2(u^{k+1}-u^k)\,dx\\
&\ls \|n^{k+1}-n^k\|_{L^2} \|\nabla n^{k+1}\|_{L^\infty}^2 \|\pa(u^{k+1} - u^k)\|_{L^2}\\
&\quad + \|n^k\|_{L^\infty}^{\theta(\gamma-1)-2} \|\pa(n^{k+1} - n^k)\|_{L^2}\|\nabla n^{k+1}\|_{L^\infty}\|\pa(u^{k+1} -u^k)\|_{L^2}\\
&\quad + \|n^{k+1} -n^k\|_{L^6}\|\nabla^2 n^{k+1}\|_{L^3} \|\pa(u^{k+1} - u^k)\|_{L^2}\\
&\quad +\|n^k\|_{L^\infty}^{\theta(\gamma-1)-2}\|\nabla(n^{k+1} -n^k)\|_{L^2} \|n^k \nabla^2 (u^{k+1} - u^k)\|_{L^2}\\
&\ls \|n^{k+1} - n^k\|_{H^1}^2 + \|\nabla(u^{k+1} - u^k)\|_{L^2}^2 + \|n^{k+1} - n^k\|_{H^1}\|n^k \nabla^2 (u^{k+1} - u^k)\|_{L^2}.
\end{align*}
\noindent $\diamond$ (Estimates for $\sfK_4$): Here, we estimate
\begin{align*}
\sfK_4 &= -2\intr \lt(n^{k+1} \pa n^{k+1} Lu^{k+1} - n^k \pa n^k Lu^k \rt)\cdot \pa(u^{k+1} - u^k)\,dx\\
&\quad - \intr \lt((n^{k+1})^2 L\pa u^{k+1} - (n^k)^2 L\pa u^k \rt)\cdot \pa(u^{k+1} - u^k)\,dx\\
&=: \sfK_4^1 + \sfK_4^2.
\end{align*}
First, $\sfK_4^1$ can be estimated as
\begin{align*}
\sfK_4^1 &= -2\intr (n^{k+1} - n^k) \pa n^{k+1} Lu^{k+1} \cdot \pa(u^{k+1}-u^k)\,dx\\
&\quad -2\intr n^k \pa(n^{k+1} - n^k) Lu^{k+1} \cdot \pa(u^{k+1} -u^k)\,dx\\
&\quad - 2\intr n^k \pa n^k L(u^{k+1} - u^k) \cdot \pa(u^{k+1}-u^k)\,dx\\
&\ls \|n^{k+1} - n^k\|_{L^6} \|\pa n^{k+1}\|_{L^\infty} \|Lu^{k+1}\|_{L^3}\|\pa(u^{k+1}-u^k)\|_{L^2}\\
&\quad + \|\pa(n^{k+1} - n^k)\|_{L^2}\|Lu^{k+1}\|_{L^3}\|n^k \pa(u^{k+1}-u^k)\|_{L^6}\\
&\quad + \|\pa n^k\|_{L^\infty}\|n^k L(u^{k+1}- u^k)\|_{L^2} \|\pa(u^{k+1} - u^k)\|_{L^2}\\
&\ls \|n^{k+1} - n^k\|_{H^1}^2 + \|\pa(u^{k+1}-u^k)\|_{L^2}^2 + \|n^{k+1} - n^k\|_{H^1} \|n^k \nabla^2(u^{k+1} -u^k)\|_{L^2} \\
&\quad + \|\pa(u^{k+1}-u^k)\|_{L^2}\|n^k \nabla^2(u^{k+1} -u^k)\|_{L^2}.
\end{align*}
On the other hand, for $\sfK_4^2$,
\begin{align*}
\sfK_4^2 &= -\frac12\intr ( (n^{k+1})^2 - (n^k)^2) L\pa u^{k+1} \cdot \pa(u^{k+1} - u^k)\,dx - \frac12\intr (n^k)^2 L\pa(u^{k+1} - u^k) \cdot \pa(u^{k+1} - u^k)\,dx\\
&\quad -\frac12\intr ( (n^{k+1})^2 - (n^k)^2) L\pa u^k \cdot \pa(u^{k+1} - u^k)\,dx - \frac12\intr (n^{k+1})^2 L\pa (u^{k+1} - u^k) \cdot \pa(u^{k+1} - u^k)\,dx\\
&=  -\frac12\intr ( n^{k+1}- n^k)(n^{k+1} + n^k) L\pa (u^{k+1}+u^k) \cdot \pa(u^{k+1} - u^k)\,dx\\
&\quad -\alpha \intr n^k (\nabla n^k \cdot \nabla\pa(u^{k+1} - u^k) )\cdot \pa(u^{k+1}-u^k)\,dx \\
&\quad -(\alpha+\beta) \intr n^k (\nabla n^k \cdot  \pa(u^{k+1}-u^k))\nabla\cdot \pa(u^{k+1} - u^k)\,dx\\
&\quad - \frac\alpha2 \intr (n^k)^2 |\nabla \pa(u^{k+1}-u^k)|^2\,dx - \frac{\alpha+\beta}{2}\intr (n^k)^2 (\nabla \cdot \pa(u^{k+1}-u^k))^2\,dx\\
&\quad -\alpha \intr n^{k+1} (\nabla n^{k+1} \cdot \nabla\pa(u^{k+1} - u^k) )\cdot \pa(u^{k+1}-u^k)\,dx \\
&\quad -(\alpha+\beta) \intr n^{k+1} (\nabla n^{k+1} \cdot  \pa(u^{k+1}-u^k))\nabla\cdot \pa(u^{k+1} - u^k)\,dx\\
&\quad - \frac\alpha2 \intr (n^{k+1})^2 |\nabla \pa(u^{k+1}-u^k)|^2\,dx - \frac{\alpha+\beta}{2}\intr (n^{k+1})^2 (\nabla \cdot \pa(u^{k+1}-u^k))^2\,dx\\
&\le \frac12\|n^{k+1}-n^k\|_{L^3}\|L\pa (u^{k+1}+u^k)\|_{L^2} \|(n^{k+1} + n^k) \pa(u^{k+1} - u^k)\|_{L^6}\\
&\quad +\alpha\|\nabla n^k\|_{L^\infty}\|n^k \nabla \pa(u^{k+1}-u^k)\|_{L^2}\|\pa(u^{k+1}-u^k)\|_{L^2} \\
&\quad + (\alpha+\beta)\|\nabla n^k\|_{L^\infty}\|n^k \nabla \cdot \pa(u^{k+1} - u^k)\|_{L^2}\|\pa(u^{k+1}-u^k)\|_{L^2} \\
&\quad - \frac\alpha2 \intr (n^k)^2 |\nabla \pa(u^{k+1}-u^k)|^2\,dx - \frac{\alpha+\beta}{2}\intr (n^k)^2 (\nabla \cdot \pa(u^{k+1}-u^k))^2\,dx\\
&\quad +\alpha\|\nabla n^{k+1}\|_{L^\infty}\|n^{k+1} \nabla \pa(u^{k+1}-u^k)\|_{L^2}\|\pa(u^{k+1}-u^k)\|_{L^2} \\
&\quad+ (\alpha+\beta)\|\nabla n^{k+1}\|_{L^\infty}\|n^k \nabla \cdot \pa(u^{k+1} - u^k)\|_{L^2}\|\pa(u^{k+1}-u^k)\|_{L^2} \\
&\quad - \frac\alpha2 \intr (n^{k+1})^2 |\nabla \pa(u^{k+1}-u^k)|^2\,dx - \frac{\alpha+\beta}{2}\intr (n^{k+1})^2 (\nabla \cdot \pa(u^{k+1}-u^k))^2\,dx\\
&\le C\|n^{k+1}-n^k\|_{H^1}\lt(\|\nabla(u^{k+1}-u^k)\|_{L^2} + \|n^k\nabla^2(u^{k+1}-u^k)\|_{L^2} + \|n^{k+1}\nabla^2 (u^{k+1}-u^k)\|_{L^2}\rt)\\
&\quad -\frac\alpha4 \intr ((n^k)^2 + (n^{k+1})^2) |\nabla \pa(u^{k+1} - u^k)|^2\,dx - \frac{\alpha+\beta}{4}\intr ((n^k)^2 + (n^{k+1})^2) (\nabla \cdot \pa(u^{k+1}-u^k))^2\,dx.
\end{align*}
Thus, we have
\begin{align*}
\sfK_4 &\le -\frac\alpha4 \intr ((n^k)^2 + (n^{k+1})^2) |\nabla \pa(u^{k+1} - u^k)|^2\,dx - \frac{\alpha+\beta}{4}\intr ((n^k)^2 + (n^{k+1})^2) (\nabla \cdot \pa(u^{k+1}-u^k))^2\,dx\\
&\quad + C\lt(\|n^{k+1} - n^k\|_{H^1}^2 + \|u^{k+1} - u^k\|_{H^1}^2 \rt)\\
&\quad + C\lt(\|n^{k+1} - n^k\|_{H^1} + \|u^{k+1} - u^k\|_{H^1} \rt)\|\sqrt{(n^k)^2 + (n^{k+1})^2} \nabla^2(u^{k+1}-u^k)\|_{L^2}.
\end{align*}

\noindent $\diamond$ (Estimates for $\sfK_5$): We obtain
\begin{align*}
\sfK_5 &= -\frac{2\delta}{\delta-1}\intr \pa\lt(n^{k+1} \nabla n^{k+1} \cdot \ml u^k - n^k \nabla n^k \cdot \ml u^{k-1} \rt) \cdot \pa(u^{k+1} - u^k)\,dx\\
&= -\frac{2\delta}{\delta-1}\intr \pa\lt((n^{k+1}-n^k) \nabla n^{k+1} \cdot \ml u^k \rt) \cdot \pa(u^{k+1} - u^k)\,dx\\
&\quad+ \frac{2\delta}{\delta-1}\intr \lt(n^k \nabla(n^{k+1} - n^k) \cdot \ml u^k \rt) \cdot \pa^2(u^{k+1} - u^k)\,dx\\
&\quad+\frac{2\delta}{\delta-1}\intr \lt(n^k \nabla n^k  \ml (u^k -u^{k-1})\rt) \cdot \pa^2(u^{k+1} - u^k)\,dx\\
&\ls \|\pa(n^{k+1}-n^k)\|_{L^2} \|\nabla n^{k+1}\|_{L^\infty}\|\ml u^k\|_{L^\infty}\|\pa(u^{k+1}-u^k)\|_{L^2}\\
&\quad + \|(n^{k+1} - n^k)\|_{L^6}\|\pa(\nabla n^{k+1} \cdot \ml u^k)\|_{L^3} \|\pa(u^{k+1} - u^k)\|_{L^2}\\
&\quad + \|\ml u^k\|_{L^\infty}\|\nabla(n^{k+1} - n^k)\|_{L^2}\|n^k \pa^2(u^{k+1} - u^k)\|_{L^2}^2\\
&\quad + \|\nabla n^k\|_{L^\infty}\|\ml(u^k - u^{k-1})\|_{L^2}\|n^k \pa^2(u^{k+1}-u^k)\|_{L^2}\\
&\ls \|n^{k+1}-n^k\|_{H^1}^2 + \|\pa(u^{k+1}- u^k)\|_{L^2}^2  \cr
&\quad + \lt(\|\nabla(n^{k+1} - n^k)\|_{L^2}+ \|\nabla(u^k - u^{k-1})\|_{L^2}\rt)\|n^k \pa^2(u^{k+1} - u^k)\|_{L^2}^2.
\end{align*}

\noindent $\diamond$ (Estimates for $\sfK_6$): In this case, we use \eqref{ineq_n} to obtain
\begin{align*}
\sfK_6 &= \intrr ((n^{k+1})^{\theta(m-1)}- (n^k)^{\theta(m-1)})(u^k-\xi)f^{k+1} \cdot\pa^2(u^{k+1} - u^k)\,dxd\xi\\
&\ls \intrr  ((n^{k+1})^{\theta(m-1)-1} + (n^k)^{\theta(m-1)-1}) |n^{k+1} - n^k| |u^k -\xi| f^{k+1} |\pa^2 (u^{k+1} - u^k)|\,dx\\
&\ls \|n^{k+1}-n^k\|_{L^6}  \lt(\|u^k\|_{L^\infty}\lt\|\intr f^{k+1}\,d\xi \rt\|_{L^3} + \lt\|\intr |\xi| f^{k+1}\,d\xi \rt\|_{L^3}\rt) \cr
&\hspace{3cm} \times \|((n^{k+1})^{\theta(m-1)-1} + (n^k)^{\theta(m-1)-1}) \nabla^2(u^{k+1}-u^k)\|_{L^2}\\
&\ls \|n^{k+1} - n^k\|_{H^1}\|\sqrt{(n^k)^2 + (n^{k+1})^2}\nabla^2 (u^{k+1}-u^k)\|_{L^2}.
\end{align*}

\noindent $\diamond$ (Estimates for $\sfK_7$ \& $\sfK_8$): Similarly as before, we estimate
\begin{align*}
\sfK_7&= \intrr (n^k)^{\theta(m-1)} (u^k - u^{k-1}) f^{k+1} \pa^2(u^{k+1} -u^k)\,dxd\xi\\
&\ls \|n^k\|_{L^\infty}^{\theta(m-1)-1} \|u^k - u^{k-1}\|_{L^6} \lt\|\intr f^{k+1}\,d\xi\rt\|_{L^3} \|n^k \nabla^2(u^{k+1}- u^k)\|_{L^2}\\
&\ls \|u^k - u^{k-1}\|_{H^1} \|n^k \nabla^2(u^{k+1}- u^k)\|_{L^2}
\end{align*}
and
\begin{align*}
\sfK_8&= \intrr (n^k)^{\theta(m-1)} (u^{k-1}-\xi) (f^{k+1} -f^k) \cdot \pa^2 (u^{k+1}-u^k)\,dxd\xi\\
&\ls \|n^k\|_{L^\infty}^{\theta(m-1)-1}(\|u^{k-1}\|_{L^\infty} + 1)\|f^{k+1}-f^k\|_{L^{2,p-2}_\nu} \|n^k \nabla^2 (u^{k+1}-u^k)\|_{L^2}\\
&\ls \|f^{k+1}-f^k\|_{L^{2,p-2}_\nu} \|n^k \nabla^2 (u^{k+1}-u^k)\|_{L^2}.
\end{align*}
Hence, we gather all the estimates for $J_{3i}$'s and use Young's inequality to yield
\bq\label{h1_u_cauchy}
\begin{aligned}
\frac{d}{dt}&\|\nabla(u^{k+1} - u^k)\|_{L^2}^2 + \intr ((n^k)^2 + (n^{k+1})^2)|\nabla^2(u^{k+1} - u^k)|^2\,dx\\
&\le C(\|u^{k+1} - u^k\|_{H^1}^2 + \|u^k-u^{k-1}\|_{H^1}^2 + \|n^{k+1} - n^k\|_{H^1}^2 + \|f^{k+1} -f^k\|_{L^{2,2}_\nu}^2).
\end{aligned}
\eq
Finally, we combine \eqref{l2_u_cauchy} with \eqref{h1_u_cauchy} to yield the desired estimate.
\end{proof}

%
%

\subsection{Proof of Theorem \ref{thm_ref}}
Now, we are in a position to present the details of the proof for Theorem \ref{thm_ref}. As mentioned before, this, combined with a simple observation, yields the proof of Theorem \ref{T1.1}. We refer to \cite[Section 3.6.1]{LPZ19}.\\

\noindent $\bullet$ (Existence): First, we collect the results from Lemmas \ref{cauchy_n}-\ref{cauchy_u} and use Young's inequality to have
\begin{align*}
\frac{d}{dt}&\lt(\|n^{k+1} - n^k\|_{H^1}^2 + \|u^{k+1}-u^k\|_{H^1}^2 + \|f^{k+1} -f^k\|_{L^{2,2}_\nu}^2 \rt) + \intr ((n^k)^2 + (n^{k+1})^2) |\nabla^2 (u^{k+1} - u^k)|^2\,dx\\
&\le C\lt(\|n^{k+1} - n^k\|_{H^1}^2 + \|u^{k+1}-u^k\|_{H^1}^2 + \|f^{k+1} -f^k\|_{L^{2,2}_\nu}^2 \rt)\cr
&\quad  + \frac12  \intr ((n^{k-1})^2 + (n^k)^2) |\nabla^2 (u^k - u^{k-1})|^2\,dx
\end{align*}
for $t \in [0,{\tilde T}]$. Here, we let
\begin{align*}
\mh^{k+1} (t) &:= \|(n^{k+1} - n^k)(\cdot,t)\|_{H^1}^2 + \|(u^{k+1}-u^k)(\cdot,t)\|_{H^1}^2 + \|(f^{k+1} -f^k)(\cdot,\cdot,t)\|_{L^{2,2}_\nu}^2,\\
\md^{k+1}(t) &:=  \intr ((n^k(x,t))^2 + (n^{k+1}(x,t))^2) |\nabla^2 (u^{k+1} - u^k)(x,t)|^2\,dx.
\end{align*}
Note that $\mh^{k+1}(0)=0$ for any $k\in\N$. Then, the above relation can be rewritten as
\bq\label{cauchy_all}
\frac{d}{dt}\mh^{k+1}(t)  + \md^{k+1}(t) \le C\lt(\mh^k (t) + \mh^{k+1}(t)\rt)+ \frac12  \md^k (t), \quad t \in [0,{\tilde T}].
\eq
We sum \eqref{cauchy_all} over $k$ to yield
\begin{align*}
\frac{d}{dt}\lt(\sum_{r=1}^k \mh^{r+1}(t) \rt)+ \sum_{r=1}^k \md^{r+1}(t) &\le C\sum_{r=1}^k \mh^{r+1}(t) + C\mh^1(t) + \frac12\sum_{r=2}^k  \md^{r}(t) + \frac12\md^1(t),
\end{align*}
and this gives, together with the uniform bound,
\begin{align*}
\frac{d}{dt}\lt(\sum_{r=1}^k \mh^{r+1}(t) \rt)+ \frac12\sum_{r=1}^k \md^{r+1}(t) &\le C\lt( 1+\sum_{r=1}^k \mh^{r+1}(t) \rt).
\end{align*}
Thus, Gr\"onwall's lemma implies
\begin{align*}
\sum_{r=1}^k \mh^{r+1}(t) + \frac12\int_0^t e^{C(t-s)}\sum_{r=1}^k \md^{r+1}(s)\,ds \le (e^{Ct}-1), \quad t \in [0,{\tilde T}].
\end{align*}
Since $C>0$ is independent of $k$, the above relation ultimately implies that the sequence $\{f^k, n^k, u^k\}$ is indeed a Cauchy sequence in $\mc([0,{\tilde T}];L^{2,p-2}_\nu(\R^3 \times \R^3)) \times \mc([0,{\tilde T}];H^1(\R^3)) \times\mc([0,{\tilde T}];H^1(\R^3))$ and converge to $(f, n,u) \in \mc([0,{\tilde T}];L^{2,p-2}_\nu(\R^3 \times \R^3)) \times \mc([0,{\tilde T}];H^1(\R^3)) \times\mc([0,{\tilde T}];H^1(\R^3))$ in s strong sense. Moreover, the uniform-in-$k$ bounds and uniqueness of weak limit imply, up to a subsequence,
\[
\begin{aligned}
f^k &\rightharpoonup f\quad \mbox{weakly *} \quad \mbox{in } \ L^\infty(0,T^*;H^{2,p}_\nu(\R^3)),\\
\pa_t f^k &\rightharpoonup \pa_t f\quad \mbox{weakly *} \quad \mbox{in } \ L^\infty(0,T^*;H^{1,p-2}_\nu(\R^3)),\\
(n^k, u^k) &\rightharpoonup (n,u) \quad \mbox{weakly *} \quad \mbox{in } \ L^\infty(0,T^*;H^3(\R^3)),\\
\pa_t n^k &\rightharpoonup \pa_t n \quad \mbox{weakly *} \quad \mbox{in } \ L^\infty(0,T^*;H^1(\R^3)),\\
\pa_t u^k& \rightharpoonup \pa_t u  \quad \mbox{weakly} \quad \mbox{in } \ L^2(0,T^*;D^2(\R^3)), \quad \mbox{and}\\
n^k \nabla^4 u^k & \rightharpoonup n\nabla^4 u  \quad \mbox{weakly} \quad \mbox{in } \ L^2( \R^3 \times (0,T^*)).
\end{aligned}
\]
Thus, we can find a limit $(f,n,u)$ satisfying
\[
\begin{aligned}
& f \geq 0, \quad f \in L^\infty(0,T^*; H^{2,p}_\nu(\R^3 \times \R^3)) \quad \pa_t f \in L^\infty(0,T^*; H^{1,p-2}_\nu(\R^3 \times \R^3)), \cr
&  n - n_\infty \in L^\infty(0,T^*;H^3(\R^3)), \quad \pa_t n \in L^\infty(0,T^*;H^2(\R^3)), \cr
& u \in  L^\infty(0,T^*;H^3(\R^3)), \quad \mbox{and}\quad  \pa_t u \in L^\infty(0,T^*;H^1(\R^3)) \cap L^2(0,T^*;D^2(\R^3)),
\end{aligned}
\]
For the time continuity, we refer to Section \ref{sec:lin_ext} and \cite[Lemma 3.5]{LPZ19}.\\

\noindent $\bullet$ (Uniqueness): Suppose that we have two regular solutions $(f_1, n_1, u_1)$ and $(f_2, n_2, u_2)$ corresponding to the same initial data $(f_0,n_0,u_0)$. Then, we can use the arguments in Lemmas \ref{cauchy_n} and \ref{cauchy_u} to yield
\begin{align*}
\frac{d}{dt}&\lt(\|n_1 -n_2\|_{H^1}^2 + \|u_1 - u_2\|_{H^1}^2 + \|f_1 - f_2\|_{L^{2,2}_\nu}^2 \rt) + \intr (n_1^2 + n_2^2)|\nabla^2(u_1 - u_2)|^2\,dx \\
&\le C\lt(\|n_1 -n_2\|_{H^1}^2 + \|u_1 - u_2\|_{H^1}^2 + \|f_1 - f_2\|_{L^{2,2}_\nu}^2 \rt),
\end{align*}
and we use Gr\"onwall's lemma to get the desired result.

%
%
%
%

\section{Finite-time singularity formation for classical solutions}\label{sec:blow}

In this section, we study the finite-time singularity formation for classical solutions to the system \eqref{main_sys} and prove Theorem \ref{main_thm2}. 
%
We first start with some energy estimates of solutions to the system \eqref{main_sys}.
\begin{lemma}\label{lem_energy}Let $(f,\rho,u)$ be a classical solution to the system \eqref{main_sys}  in the interval $[0,T]$. 
Then we have
$$\begin{aligned}
&(i)  \,\,\,\,\frac{d}{dt} m_\rho(t) = \frac{d}{dt} m_f(t) = 0, \quad  \frac{d}{dt} M(t)= 0,\cr
&(ii) \,\, \frac{d}{dt}E(t) + \intr \lt(2\mu(\rho) \mathbb{T}(u) : \mathbb{T}(u) + \lambda(\rho)|\nabla \cdot u|^2\rt)dx + \intrr \rho^m|u-\xi|^2 f\,dxd\xi = 0
\end{aligned}$$
for all $t \in [0,T]$. Here $\mathbb{A}:\mathbb{B} = \sum_{i,j=1}^3 a_{ij} b_{ij}$ for $\mathbb{A} = (a_{ij}), \mathbb{B} = (b_{ij}) \in \R^{3 \times 3}$. 
\end{lemma}

In the lemma below, we next estimate the total momentum of inertia and some bound estimates on the momentum and internal energy. We also provide the growth rate of the total momentum of inertia in time. 
\begin{lemma}\label{lem_useful}Let $(f,\rho,u)$ be a classical solution to the system \eqref{main_sys} in the interval $[0,T]$. Then for $t \in [0,T]$, the followings hold:
\begin{itemize}
\item[(i)] The momentum of inertia satisfy
\[
\frac{d}{dt}I_\rho(t) = W_\rho(t), \quad \frac{d}{dt}I_f(t) = W_f(t), \quad \mbox{i.e.,} \quad \frac{d}{dt} I(t) = W(t).
\]
\item[(ii)] The momentum can be bounded from above as
\[
M^2 \leq 4\max\{m_\rho, m_f\}\lt(E_k(t) + E_f(t)\rt).
\]
\item[(iii)] The internal energy can be bounded from below as 
\[
E_i(t) \geq \frac{C_0}{I_\rho(t)^{\frac{3(\gamma-1)}{2}}}.
\]
where $C_0 > 0$ is given by
\[
C_0 := \lt(\frac{\pi^{3/2}}{\Gamma\lt(5/2 \rt)}\rt)^{1 - \gamma}\frac{(m_\rho(0))^{\frac{5\gamma - 3}{2}}}{2^{\frac{5\gamma - 3}{2}}(\gamma-1)}.
\]

\item[(iv)] The total momentum of inertia is bounded by a quadratic function of $t$:
\[
I(t) \leq I(0) + C_1t + C_2t^2,
\]
where $C_1$ and $C_2$ are positive constants given by
\[
C_1:= W(0)  + \frac{(2\alpha + 3\beta)}{\delta-1} (m_\rho(0))^{\frac{\gamma-\delta}{\gamma-1}}(\gamma-1)^{\frac{\delta-1}{\gamma-1}}E(0)^{\frac{\delta-1}{\gamma-1}}
\]
and
\[
C_2 :=\frac12\lt( \max\{2,3(\gamma - 1)\}E(0)\rt),
\]
respectively.
\end{itemize}
\end{lemma}
\begin{proof} The proofs of the assertions (i)--(iii) can be found in \cite{C17}. Since the proof for (iv) is slightly different from \cite[Lemma 3.2]{C17}, for the completeness of our work, we briefly sketch it here.\\

\noindent We first easily find
\begin{align}\label{new1}
\begin{aligned}
\frac{d}{dt} W(t) &= \intr \rho |u|^2\,dx + 3\intr p(\rho)\,dx + \intrr |\xi|^2 f\,dxd\xi - \intr (2\mu(\rho) + 3\lambda(\rho))(\nabla \cdot u)\,dx\cr
&=2(E_f + E_k) + 3(\gamma - 1) E_i- (2\alpha + 3\beta)\intr \rho^\delta (\nabla \cdot u)\,dx.
\end{aligned}
\end{align}
On the other hand, we notice that
$$\begin{aligned}
\frac{d}{dt}\intr \mu(\rho)\,dx &= \intr \mu^\prime(\rho) \rho_t\,dx \cr
&= -\intr \mu^\prime(\rho)\nabla \cdot(\rho u)\,dx \cr
&= -\intr (\rho \mu^\prime(\rho) - \mu(\rho))(\nabla \cdot u)\,dx \cr
&=  -\alpha(\delta-1)\intr \rho^\delta (\nabla \cdot u)\,dx.
\end{aligned}$$
This together with Lemma \ref{lem_energy} implies
$$
\begin{aligned}
\frac{d}{dt} W(t) &= 2(E_k + E_f) + 3(\gamma - 1) E_i + \frac{2\alpha + 3\beta}{\alpha(\delta-1)}\frac{d}{dt}\intr \mu(\rho)\,dx\cr
&\leq \max\{2,3(\gamma - 1)\}E(0) + \frac{2\alpha + 3\beta}{\alpha(\delta-1)}\frac{d}{dt}\intr \mu(\rho)\,dx.
\end{aligned}
$$
Thus we obtain
\[
W(t)   \leq W(0)  + \frac{c_\mu(2\alpha + 3\beta)}{\delta-1}  + \lt(  \max\{2,3(\gamma - 1)\}E(0) \rt)t,
\]
where we used
\bq\label{new2}
\intr \rho^\delta\,dx \leq (m_\rho(0))^{\frac{\gamma-\delta}{\gamma-1}}(\gamma-1)^{\frac{\delta-1}{\gamma-1}}E_i(t)^{\frac{\delta-1}{\gamma-1}} \leq (m_\rho(0))^{\frac{\gamma-\delta}{\gamma-1}}(\gamma-1)^{\frac{\delta-1}{\gamma-1}}E(0)^{\frac{\delta-1}{\gamma-1}} =: c_\mu.
\eq
Subsequently, we have
\[
I(t) \leq I(0) + \lt( W(0)  + \frac{c_\mu(2\alpha + 3\beta)}{\delta-1}\rt)t + \frac12\lt( \max\{2,3(\gamma - 1)\}E(0)\rt)t^2,
\]
due to $E(t) \geq 0$ and $I^\prime(t) = W(t)$.
\end{proof}

We next recall from \cite[Lemma A.1]{C17} the following Gr\"onwall-type inequality.  
\begin{lemma}\label{lem_usef2} Let us consider the following differential inequality:
\[
f^\prime(t) \leq \frac{a}{t+1} f(t) + \frac{b}{(t+1)^{2c}}f^c,
\]
where $a,b$, and $c$ are positive constants, and $f \geq 0$. Suppose $c \leq 1$. 
\begin{itemize}
\item If $c = 1$, $f$ satisfies
\[
f(t) \leq f(0)e^b (t+1)^a.
\]
\item If $c < 1$ and $2c + a(1-c) = 1$, $f$ satisfies
\[
f(t)^{1-c} \leq f(0)^{1-c}(t+1)^{1-2c} + (1-2c)(t+1)^{1-2c}\ln(t+1).
\]
\item If $c < 1$ and $2c + a(1-c) \neq 1$, $f$ satisfies
\[
f(t)^{1-c} \leq f(0)^{1-c}(t+1)^{a(1-c)} + \frac{b(1-c)}{1 - (2c + a(1-c))}\lt( (t+1)^{1-2c} - (t+1)^{a(1-c)} \rt).
\]
\end{itemize}
\end{lemma}

We now consider 
\[
J(t) := I(t) - (t+1)W(t) + (t+1)^2E(t),
\]
from which we estimate the upper bound on $E_i$.

\begin{lemma}\label{lem_gd} For $1 < \gamma < \frac53$, if $\gamma - \frac13< \delta < \gamma$, we have the following upper bounds of $J$:
\[
J(t) \leq \lt(J(0)^{\frac{\gamma-\delta}{\gamma-1}} +\frac{(2\alpha + 9\beta) \,(m_\rho(0))^{\frac{\gamma-\delta}{\gamma-1}}(\gamma-1)^{\frac{\delta-\gamma}{\gamma-1}}(\gamma-\delta)}{4(1 - 3(\gamma - \delta))}\rt)^{\frac{\gamma-1}{\gamma-\delta}}(t+1)^{2 - 3(\gamma-1)}.
\]
\end{lemma}
\begin{proof}We first decompose the function $J$ into two terms: $J = J_\rho + J_f$ where
\[
J_\rho(t) := I_\rho(t) - (t+1)W_\rho(t) + (t+1)^2E_\rho(t) \quad \mbox{and} \quad J_f(t) := I_f(t) - (t+1)W_f(t) + (t+1)^2E_f(t).
\]
Here $E_\rho := E_k + E_i$. Then since
\[
W_\rho(t)^2 \leq 4I_\rho(t) E_k(t) \quad \mbox{and} \quad W_f(t)^2\leq 4I_f(t) E_f(t),
\]
we get
\[
I_\rho(t) - (t+1)W_\rho(t) + (t+1)^2E_k(t) \geq 0 \quad \mbox{and} \quad I_f(t) - (t+1)W_f(t) + (t+1)^2E_f(t) \geq 0.
\]
This implies
\bq\label{est_j}
J_\rho(t) \geq (t+1)^2 E_i(t), \quad J_f(t) \geq 0, \quad \mbox{and} \quad J(t) \geq (t+1)^2 E_i(t).
\eq
On the other hand, it follows from Lemma \ref{lem_useful} (i) that we find
\[
\frac{d}{dt} J_\rho(t) = -(t+1)\frac{d}{dt} W_\rho(t) + 2(t+1)E_\rho(t) + (t+1)^2 \frac{d}{dt} E_\rho(t),
\]
and
\[
\frac{d}{dt} J_f(t) = -(t+1)\frac{d}{dt} W_f(t) + 2(t+1)E_f(t) + (t+1)^2 \frac{d}{dt}E_f(t).
\]
We then combine that with the identity \eqref{new1} to get
\begin{align}\label{est_jj}
\begin{aligned}
\frac{d}{dt}J(t) &= -(t+1)\frac{d}{dt}W(t) + 2(t+1)E(t) + (t+1)^2\frac{d}{dt}E(t)\cr
&= (t+ 1)(2 - 3(\gamma - 1))E_i + (t+1)\intr (2\mu(\rho) + 3\lambda(\rho))(\nabla \cdot u)\,dx + (t+1)^2\frac{d}{dt} E(t).
\end{aligned}
\end{align}
On the other hand, we use the H\"older's inequality to find
\begin{align}\label{est_jjj}
\begin{aligned}
2(t+1)\intr \mu(\rho)(\nabla \cdot u)\,dx &\leq \frac12\intr \mu(\rho)\,dx + (t+1)^2\intr 2\mu(\rho)|\nabla \cdot u|^2\,dx\cr
&\leq \frac12\intr \mu(\rho)\,dx  + (t+1)^2\intr 2\mu(\rho)\mathbb{T}(u):\mathbb{T}(u)\,dx\cr
3(t+1)\intr \lambda(\rho)(\nabla \cdot u)\,dx & \leq \frac 94\intr \lambda(\rho)\,dx + (t+1)^2\intr \lambda(\rho)|\nabla \cdot u|^2\,dx\cr
&= \frac 94\intr \lambda(\rho)\,dx+ (t+1)^2\intr \lambda(\rho)|\nabla \cdot u|^2\,dx.
\end{aligned}
\end{align}
This and together with the energy estimate in Lemma \ref{lem_energy} implies
$$\begin{aligned}
&(t+1)\intr (2\mu(\rho) + 3\lambda(\rho))(\nabla \cdot u)\,dx \cr
&\quad \leq \frac12\intr \mu(\rho)\,dx + \frac 94\intr \lambda(\rho)\,dx + (t+1)^2\lt(\intr 2\mu(\rho)\mathbb{T}(u):\mathbb{T}(u)\,dx+ \intr \lambda(\rho)|\nabla \cdot u|^2\,dx\rt)\cr
&\quad \leq \lt(\frac\alpha2 + \frac{9\beta}4\rt)\intr \rho^\delta\,dx -(t+1)^2\frac{d}{dt} E(t),
\end{aligned}$$
due to $\mu(\rho) = \alpha \rho^\delta = (\alpha/\beta) \lambda(\rho)$.\\

\noindent Then we estimate
$$\begin{aligned}
\frac{d}{dt} J(t) &=(t+ 1)(2 - 3(\gamma - 1))E_i + (t+1)\intr (2\mu(\rho) + 3\lambda(\rho))(\nabla \cdot u)\,dx + (t+1)^2\frac{d}{dt} E(t)\cr
& \leq (t+1)(2 - 3(\gamma - 1))E_i(t) + \frac{2\alpha + 9\beta}{4}\intr \rho^\delta\,dx\cr
&\leq (t+1)(2 - 3(\gamma - 1))E_i(t) + \frac{(m_\rho(0))^{\frac{\gamma-\delta}{\gamma-1}}(\gamma-1)^{\frac{\delta-1}{\gamma-1}}\lt(2\alpha + 9\beta \rt)}{4}E_i(t)^{\frac{\delta-1}{\gamma-1}},
\end{aligned}$$
where we used \eqref{new2}, \eqref{est_jj}, and \eqref{est_jjj}. If $2 - 3(\gamma - 1) > 0$, then it follows from \eqref{est_j} that
\[
\frac{d}{dt}J(t) \leq \frac{2 - 3(\gamma - 1)}{(t+1)}J(t) + \frac{(m_\rho(0))^{\frac{\gamma-\delta}{\gamma-1}}(\gamma-1)^{\frac{\delta-1}{\gamma-1}}\lt(2\alpha + 9\beta \rt)}{4(t+1)^{\frac{2(\delta - 1)}{\gamma-1}}}J(t)^{\frac{\delta-1}{\gamma-1}}.
\]
We now apply Lemma \ref{lem_usef2} to the above differential inequality with 
\[
a = 2 - 3(\gamma - 1), \quad b = \frac{(m_\rho(0))^{\frac{\gamma-\delta}{\gamma-1}}(\gamma-1)^{\frac{\delta-1}{\gamma-1}}\lt(2\alpha + 3\beta  \rt)}{4},\quad \mbox{and}\quad c= \frac{\delta-1}{\gamma-1}.
\]
Note that if $\delta > \gamma - \frac13$, then it is clear to get $2c + a(1-c) = 2-3(\gamma-\delta) > 1$. Thus we have
\[
J(t) \leq \lt(J(0)^{\frac{\gamma-\delta}{\gamma-1}} +\frac{\lt(2\alpha + 9\beta  \rt) \,(m_\rho(0))^{\frac{\gamma-\delta}{\gamma-1}}(\gamma-1)^{\frac{\delta-\gamma}{\gamma-1}}(\gamma-\delta)}{4(1 - 3(\gamma - \delta))}\rt)^{\frac{\gamma-1}{\gamma-\delta}}(t+1)^{2 - 3(\gamma-1)}.
\]
This completes the proof.
\end{proof}

\begin{proof}[Proof of Theorem \ref{main_thm2}]  It follows from Lemma \ref{lem_useful} (iii), Lemma \ref{lem_gd}, and \eqref{est_j} that
\[
\frac{C_0}{I_\rho(t)^{\frac{3(\gamma-1)}{2}}} \leq E_i(t) \leq \frac{C_3}{(t+1)^{3(\gamma - 1)}},
\]
where $C_3$ is a positive constant given by
\[
C_3 :=  \lt(J(0)^{\frac{\gamma-\delta}{\gamma-1}} +\frac{\lt(2\alpha + 9\beta \rt) \,(m_\rho(0))^{\frac{\gamma-\delta}{\gamma-1}}(\gamma-1)^{\frac{\delta-\gamma}{\gamma-1}}(\gamma-\delta)}{4(1 - 3(\gamma - \delta))}\rt)^{\frac{\gamma-1}{\gamma-\delta}}.
\]
On the other hand, we also find from Lemma \ref{lem_useful} (iv) that  
\[
I_\rho(t) \leq I(t) \leq I(0) + C_1 \, t + C_2\, t^2,
\]
and this yields
\[
\frac{C_0}{(I(0) + C_1 \, t + C_2 \, t^2)^{\frac{3(\gamma-1)}{2}}} \leq E_i(t) \leq \frac{C_3}{(t+1)^{3(\gamma - 1)}}.
\]
Therefore if the initial data satisfy
\[
C_0 > C_3 C_2,
\]
the life-span $T$ of classical solutions should be finite.
\end{proof}

%
%
%
%

\section*{Acknowledgments}
Y.-P. Choi was supported by National Research Foundation of Korea(NRF) grant funded by the Korea government(MSIP) (No. 2017R1C1B2012918). The work of J. Jung was supported by NRF grant (No. 2019R1A6A1A10073437).

\appendix

\section{Growth estimate of the support of $f$ in velocity}\label{app_spt}

In our study, we considered the particle distribution function $f$ which has a finite velocity moment, and this led to introducing the weighted Sobolev space, $H^{k,p}_\nu(\R^3 \times \R^3)$. On the other hand, if we assume that the initial data $f_0$ has a compact support in velocity, then we can show that our regular solution $f(\cdot,\cdot,t)$ is compactly supported in velocity. Thus, in this case, we do not need to introduce the weight Sobolev space, and the dealing with the standard Sobolev space $H^{k,p}(\R^3 \times \R^3)$ is enough to have our regular solution. 

In order to illustrate this strategy, in this appendix, we provide the estimate of the support of $f$ in velocity. We set $\om(t)$ and $R^\xi(t)$ the $\xi$-projection of support of $f(\cdot,\cdot,t)$ and the maximum value of $v$ in $\om(t)$. Then, under assumption \eqref{as_psi_v}, we define the following forward characteristics $(X(s), \Xi(s)) := (X(s;x,\xi,0), \Xi(s;x,\xi,0))$ for a given $(x,\xi) \in \R^3 \times \R^3$ associated to the kinetic equation in \eqref{lin_sys} by
\begin{align}\label{eqn_tra2}
\begin{aligned}
\frac{dX(s)}{ds} &= \Xi(s),\cr
\frac{d\Xi(s)}{ds} &=  n^{m\theta}(X(s),s)(v(X(s),s) - \Xi(s)),
\end{aligned}
\end{align}
with the initial condition $(X(0),\Xi(0)) = (x,\xi)$. Note that the regularity assumptions on $v$ in \eqref{as_psi_v}, the regularity estimate on $n$ in Proposition \ref{prop_lsol}, and Cauchy--Lipschitz theorem yield the well-posedness of the characteristics $(X(s), \Xi(s))$ on the time interval $[0,T]$. 

In the lemma below, we present the growth estimate for support of $f$ in velocity.
\begin{lemma} Let $(X(s), \Xi(s))$ be the solution to the characteristic equations \eqref{eqn_tra2} emanating from an initial point in the support of $f_0$. Then we have 
\[
\sup_{0 \leq t \leq T}|\Xi(t)| \leq |\xi| + C
\]
for some $C>0$. In particular, the support of $f(\cdot,\cdot,t)$ in velocity is bounded from above by some positive constant over the time interval $[0,T]$, that is,  $|\om(t)| \leq C$ for $t \in [0,T]$. 

\end{lemma}
\begin{proof}It follows from the velocity equation in \eqref{eqn_tra2} that
\[
\frac{d|\Xi(s)|}{ds} \leq  n^{m\theta}(X(s),s)(|v(X(s),s)| - |\Xi(s)|).
\]
Applying Gr\"onwall's lemma to the above gives
$$\begin{aligned}
|\Xi(t)| &\leq |\xi| \exp\lt(- \int_0^t n^{m\theta}(X(s),s)\,ds \rt) \cr
&\quad + \int_0^t n^{m\theta}(X(s),s) |v(X(s),s)| \exp\lt( -\int_s^t n^{m\theta}(X(\tau),\tau)\,d\tau \rt) ds\cr
&\leq |\xi| + \int_0^t n^{m\theta}(X(s),s) |v(X(s),s)|\,ds
\end{aligned}$$
for $\xi \in \om(0)$ and $t \in [0,T]$, where we used the positivity of $n$. Moreover, by the assumptions on $v$ in \eqref{as_psi_v} and Proposition \ref{prop_lsol}, we estimate
$$\begin{aligned}
\int_0^t n^{m\theta}(X(s),s) |v(X(s),s)|\,ds &\leq \int_0^t \|n(\cdot,s)\|_{L^\infty}^{m\theta} \|v(\cdot,s)\|_{L^\infty}\,ds\cr
&\leq C\int_0^t \lt(\|n(\cdot,s) - n_\infty\|_{L^\infty}^{m\theta} + |n_\infty|^{m\theta}\rt)\|v(\cdot,s)\|_{L^\infty}\,ds\cr
&\leq Ct,
\end{aligned}$$
for some $C>0$. This concludes
\[
|\Xi(t)| \leq |\xi| + C
\]
for $t \in [0,T]$.
\end{proof}

The above lemma shows the finite propagation of support of $f$ in velocity, and this simplifies many estimates related to the drag force term in the momentum equations. 
%
%
%
%
%
\section{$H^{2,p}_\nu$-estimate of $f$ \& $H^{1,p-2}_\nu$-estimate of $\pa_t f$}\label{app_f}
In this appendix, we first estimate $\dot{H}^{2,p}_\nu$ norm of $f$ to complete the proof for the first assertion in Lemma \ref{lem_f}.

Note for $i,j=1,2,3$ that $\pa_{ij} f$ satisfies 
$$\begin{aligned}
&\pa_t \pa_{ij} f + \xi \cdot \nabla \pa_{ij} f +  \pa_i (n^{m\theta} (v-\xi) \cdot \nabla_\xi \pa_j f) +  \pa_i (n^{m\theta} \pa_ j v \cdot \nabla_\xi f)\cr
&\quad = 3\pa_i (\pa_j(n^{m\theta})  f) -  \pa_i (\pa_j(n^{m\theta})(v-\xi) \cdot \nabla_\xi f) + 3 \pa_i(n^{m\theta} \pa_j f).
\end{aligned}$$
Thus we get
$$\begin{aligned}
&\frac12\frac{d}{dt}\intrr \nu_p(x,\xi) |\pa_{ij}f|^2\,dxd\xi \cr
&\quad = -\intrr \nu_p(x,\xi) \pa_{ij} f \cdot \lt(  \xi \cdot \nabla \pa_{ij} f  + \pa_i \lt( n^{m\theta} (v - \xi) \cdot \nabla_\xi \pa_j f\rt) + \pa_i \lt( n^{m\theta} \pa_j v \cdot\nabla_\xi f\rt)\rt)dxd\xi\cr
&\qquad +\intrr \nu_p(x,\xi) \pa_{ij} f \cdot \lt( 3\pa_i \lt(\pa_j(n^{m\theta}) f \rt) - \pa_i \lt(\pa_j(n^{m\theta}) (v- \xi)\cdot\nabla_\xi f \rt)   \rt)dxd\xi\cr
&\qquad -\intrr \nu_p(x,\xi) \pa_{ij} f \cdot \lt(3\pa_i \lt( n^{m\theta}\pa_j f\rt)\rt)dxd\xi\cr
&\quad =: \sum_{i=1}^6 \sfI_i,
\end{aligned}$$
where $\sfI_1$ can be estimated as
\[
\sfI_1 =  p\intrr \nu_{p-2}(x,\xi) x \cdot \xi |\pa_{ij} f|^2\,dxd\xi \leq C\|\pa_{ij} f\|_{L^{2,p}_\nu}^2.
\]
We separately estimate $\sfI_2$ as follows:
\begin{align*}
\sfI_2 &= -\intrr \nu_p(x,\xi) \pa_{ij} f \cdot (\pa_i (n^{m\theta})(v-\xi) \cdot \nabla_\xi \pa_j f + n^{m\theta} \pa_i v \cdot \nabla_\xi \pa_j f)\,dxd\xi\\
&\quad -\intrr \nu_p(x,\xi)\pa_{ij}f \cdot (n^{m\theta} (v-\xi)\cdot \nabla_\xi \pa_{ij} f )\,dxd\xi\cr
&=: \sfI_2^1 + \sfI_2^2 + \sfI_2^3.
\end{align*}
For $\sfI_2^1$,

\[\begin{aligned}
\sfI_2^1 &= -m\theta \intrr \nu_p(x,\xi)(\pa_{ij} f) n^{m\theta-1} (\pa_i n) v\cdot \nabla_\xi \pa_j f\,dxd\xi\\
&\quad +m\theta \intrr \nu_p(x,\xi) (\pa_{ij} f) n^{m\theta-1} (\pa_i n) \xi\cdot \nabla_\xi \pa_j f\,dxd\xi\\
&\le C\|n\|_{L^\infty}^{m\theta-1}\|\nabla n\|_{L^\infty}\|v\|_{L^\infty}\|f\|_{H_\nu^{2,p}}^2+ (m\theta)^2\intrr \nu_p(x,\xi) |\pa_i n|^2 n^{m\theta-2}|\nabla_\xi \pa_j f|^2\,dxd\xi
\\
&\quad +\frac14\intrr  \nu_p(x,\xi)|\xi|^2 n^{m\theta} |\pa_{ij} f|^2\,dxd\xi\\
&\le C\|n\|_{L^\infty}^{m\theta-2}(1+\|n\|_{L^\infty}+\|\nabla n\|_{H^2})^2(1+\|v\|_{L^\infty})\|f\|_{H_\nu^{2,p}}^2 +\frac14\intrr  \nu_p(x,\xi)|\xi|^2 n^{m\theta} |\pa_{ij} f|^2\,dxd\xi.
 \end{aligned}\]
We easily estimate $\sfI_2^2$ as

\[
\sfI_2^2 \le \|n\|_{L^\infty}^{m\theta}\|\nabla v\|_{L^\infty}\|f\|_{H_\nu^{2,p}}^2.
\]
For $\sfI_2^3$,

\[\begin{aligned}
\sfI_2^3&= \frac12\intrr \nabla_\xi \cdot (\nu_p(x,\xi)(v-\xi)) n^{m\theta}|\pa_{ij}f|^2\,dxd\xi\\
&\leq \frac12\intrr (p\nu_{p-2}(x,\xi) + 2\nu_p(x,\xi))(v-\xi)\cdot \xi n^{m\theta}|\pa_{ij}f|^2\,dxd\xi\\
&\le C\|n\|_{L^\infty}^{m\theta}(1+\|v\|_{L^\infty})^2\|\pa_{ij}f\|_{L_\nu^{2,p}}^2 - \frac12\intrr\nu_p(x,\xi)|\xi|^2n^{m\theta} |\pa_{ij} f|^2\,dxd\xi.
\end{aligned}\]
Thus, we can get

\[
\sfI_2 \le C\|n\|_{L^\infty}^{m\theta-2}(1+\|n\|_{L^\infty}+\|\nabla n\|_{H^2})^2(1+\|v\|_{H^3})^2\|f\|_{H_\nu^{2,p}}^2 -\frac14\intrr \nu_p(x,\xi)|\xi|^2 n^{m\theta}|\pa_{ij}f|^2\,dxd\xi.
\]

For the estimate of $\sfI_3$, we get
\begin{align*}
\sfI_3 &\leq C\|\pa_{ij} f\|_{L^{2,p}_\nu}\|n\|_{L^\infty}^{m\theta-1}  \lt(\|\pa_i n\|_{L^\infty}\|\pa_j v\|_{L^\infty}\|\nabla_\xi f\|_{L^{2,p}_\nu}+ \|n\|_{L^\infty}\|\pa_j v\|_{L^\infty}\|\nabla_\xi \pa_i f\|_{L^{2,p}_\nu} \rt)\cr
&\quad + C\|n\|_{L^\infty}^{m\theta}\intrr \nu_p(x,\xi) | \pa_{ij} f| |\pa_{ij} v| |\nabla_\xi f|\,dxd\xi.
\end{align*}
If we set 
\bq\label{new_g}
g(x) := \lt(\intr \nu_p(x,\xi) |\nabla_\xi f|^2\,d\xi \rt)^{1/2},
\eq
then by H\"older's inequality we find
\[
|\nabla g(x)| \leq \frac p2 g(x)+ \lt(\intr \nu_p(x,\xi) |\nabla \nabla_\xi f|^2\,d\xi \rt)^{1/2},
\]
and thus by Sobolev embedding $H^1(\R^3) \subseteq L^3(\R^3)$
\[
\|g\|_{L^3} \leq C\|g\|_{H^1} \leq C\|f\|_{H^{2,p}_\nu}
\]
for some $C>0$. This implies that  the last term can be estimated as
$$\begin{aligned}
&C\|n\|_{L^\infty}^{m\theta}\intrr \nu_p(x,\xi) | \pa_{ij} f| |\pa_{ij} v| |\nabla_\xi f|\,dxd\xi\cr
&\quad \leq C\|n\|_{L^\infty}^{m\theta}\intr|\pa_{ij} v|  \lt(\intr \nu_p(x,\xi) | \pa_{ij} f|^2\,d\xi\rt)^{1/2}  \lt(\intr \nu_p(x,\xi) |\nabla_\xi f|^2\,d\xi \rt)^{1/2} dx\cr
&\quad \leq C\|n\|_{L^\infty}^{m\theta}\|\pa_{ij}v\|_{L^6}\|\pa_{ij}f\|_{L^{2,p}_\nu}\|g\|_{L^3}\cr
&\quad \leq C\|n\|_{L^\infty}^{m\theta}\|\nabla^3 v\|_{L^2}\|f\|_{H^{2,p}_\nu}^2.
\end{aligned}$$
Hence we have
\[
\sfI_3 \ls \|n\|_{L^\infty}^{m\theta-1}(1+\|n\|_{L^\infty}+\|\nabla n\|_{H^2} )(1+\|v\|_{H^3} )\|f\|_{H^{2,p}_\nu}^2.
\]
In a similar way, we also deduce
$$\begin{aligned}
\sfI_4 &\ls \|n\|_{L^\infty}^{m\theta-2}\lt(\|\nabla n\|_{L^\infty}^2 + \|n\|_{L^\infty}\|\nabla^2 n\|_{L^6} + \|n\|_{L^\infty}\|\nabla n\|_{L^\infty}\rt)\|f\|_{H^{2,p}_\nu}^2\cr
&\ls \|n\|_{L^\infty}^{m\theta-2} \|\nabla n\|_{H^2} \lt(\|\nabla n\|_{H^2} + \|n\|_{L^\infty}\rt)\|f\|_{H^{2,p}_\nu}^2
\end{aligned}$$
and
$$\begin{aligned}
\sfI_6 &\ls \|\pa_{ij} f\|_{L^{2,p}_\nu}\|n\|_{L^\infty}^{m\theta-1}(\|\nabla n\|_{L^\infty}\|\pa_j f\|_{L^{2,p}_\nu} +\|n\|_{L^\infty} \|\pa_{ij}f\|_{L^{2,p}_\nu})\cr
&\ls \|n\|_{L^\infty}^{m\theta-1}(\|\nabla n\|_{H^2} + \|n\|_{L^\infty})\|f\|_{H^{2,p}_\nu}^2.
\end{aligned}$$
Finally, we also separately estimate $\sfI_5$ as
$$\begin{aligned}
\sfI_5 &=-\intrr \nu_p(x,\xi) \pa_{ij}f \cdot (\pa_{ij}(n^{m\theta}) (v-\xi)\cdot\nabla_\xi f  + \pa_j(n^{m\theta})\pa_iv \cdot \nabla_\xi f)\,dxd\xi\\
&\quad -\intrr \nu_p(x,\xi)\pa_{ij}f \cdot (\pa_j(n^{m\theta})(v-\xi)\cdot\nabla_\xi \pa_j f)\,dxd\xi\\
&=: \sfI_5^1 + \sfI_5^2 + \sfI_5^3.
\end{aligned}$$
For $\sfI_5^1$, the estimates for \eqref{new_g} and Young's inequality give
\begin{align*}
\sfI_5^1 &= -m\theta\intrr \nu_p(x,\xi)\pa_{ij}f ((\pa_{ij}n) n^{m\theta-1} + (m\theta-1) (\pa_i n)(\pa_j n) n^{m\theta-2})(v-\xi)\cdot\nabla_\xi f\,dxd\xi\\
&\le C\|n\|_{L^\infty}^{m\theta-1}\|v\|_{L^\infty}\intrr\nu_p(x,\xi)|\pa_{ij} f| |\pa_{ij}n||\nabla_\xi f|\,dxd\xi + C\|n\|_{L^\infty}^{m\theta-2}\|\nabla n\|_{L^\infty}^2\|v\|_{L^\infty}\|f\|_{H_\nu^{2,p}}^2\\
&\quad +m\theta \|n\|_{L^\infty}^{\frac{m\theta}{2}-1}\intrr \nu_p(x,\xi) n^{\frac{m\theta}{2}} |\xi| |\pa_{ij}f|  |\pa_{ij}n| |\nabla_\xi f| \,dxd\xi \\
&\quad+ m\theta(m\theta-1)\|\nabla n\|_{L^\infty}^2 \intrr \nu_p(x,\xi)n^{m\theta-2} |\xi| |\pa_{ij} f| |\nabla_\xi f|\,dxd\xi\\
&\le C\|n\|_{L^\infty}^{m\theta-1}(1+\|\nabla n\|_{H^2})^2 \|v\|_{L^\infty}\|f\|_{H_\nu^{2,p}}^2\\
&\quad + m\theta\|n\|_{L^\infty}^{\frac{m\theta}{2}-1}\|\pa_{ij} n\|_{L^6}\lt(\intrr \nu_p(x,\xi)|\xi|^2 n^{m\theta}|\pa_{ij} f|^2\,dxd\xi \rt)^{1/2} \lt\|\intr \nu_p(x,\xi)|\nabla_\xi f|^2\,d\xi\rt\|_{L^3}\\
&\quad +\frac{1}{16}\intrr \nu_p(x,\xi)|\xi|^2 n^{m\theta} |\pa_{ij}f|^2\,dxd\xi + C\|\nabla n\|_{L^\infty}^4\|n\|_{L^\infty}^{m\theta-4}\|\nabla_\xi f\|_{L_\nu^{2,p}}^2\\
&\le C\|n\|_{L^\infty}^{m\theta-4}(1+\|n\|_{L^\infty} + \|\nabla n\|_{H^2})^4 (1+\|v\|_{L^\infty})\|f\|_{H_\nu^{2,p}}^2  \\
&\quad +\frac{1}{8}\intrr \nu_p(x,\xi)|\xi|^2 n^{m\theta} |\pa_{ij}f|^2\,dxd\xi,
\end{align*}
where we used $m\theta \ge 4$.

We directly estimate $\sfI_5^2$ as
\[
\sfI_5^2 \le m\theta \|n\|_{L^\infty}^{m\theta-1} \|\nabla n\|_{L^\infty}\|\nabla v\|_{L^\infty} \|f\|_{H_\nu^{2,p}}^2.
\]
Similarly to $\sfI_5^1$, we estimate $\sfI_5^3$ as
\[\begin{aligned}
\sfI_5^3 &\le C\|n\|_{L^\infty}^{m\theta-1}\|\nabla n\|_{L^\infty}\|v\|_{L^\infty}\|f\|_{H_\nu^{2,p}}^2\\
&\quad + m\theta\|\nabla n\|_{L^\infty}\intrr \nu_p(x,\xi)n^{m\theta-1} |\xi||\pa_{ij} f| |\nabla_\xi \pa_j f|\,dxd\xi\\
&\le C\|n\|_{L^\infty}^{m\theta-1}\|\nabla n\|_{L^\infty}\|v\|_{L^\infty}\|f\|_{H_\nu^{2,p}}^2\\
&\quad + \frac18 \intrr \nu_p(x,\xi)|\xi|^2 n^{m\theta} |\pa_{ij} f|^2\,dxd\xi + C\|n\|_{L^\infty}^{m\theta-2}\|\nabla n\|_{L^\infty}\|\nabla_\xi \pa_j f\|_{L_\nu^{2,p}}^2\\
&\le C\|n\|_{L^\infty}^{m\theta-2}(1+\|n\|_{W^{1,\infty}})^2(1+\|v\|_{L^\infty})\|f\|_{H_\nu^{2,p}}^2 + \frac18 \intrr \nu_p(x,\xi)|\xi|^2 n^{m\theta} |\pa_{ij} f|^2\,dxd\xi.
\end{aligned}\]
Hence we obtain
\[
\sfI_5 \le C\|n\|_{L^\infty}^{m\theta-4}(1+\|n\|_{L^\infty}+\|\nabla n\|_{H^2})^4(1+\|v\|_{H^3})\|f\|_{H_\nu^{2,p}}^2 + \frac14 \intrr \nu_p(x,\xi)|\xi|^2 n^{m\theta} |\pa_{ij} f|^2\,dxd\xi.
\]
Combining the estimates $\sfI_i, i=1,\dots,6$ yields
\bq\label{est_hf3}
\frac{d}{dt} \|\pa_{ij} f\|_{L^{2,p}_\nu}^2 \leq  C(1+\|n\|_{L^\infty})^{m\theta-4} (1+\|n\|_{L^\infty} +\|\nabla n\|_{H^2})^4 (1+\|v\|_{H^3} )^2 \|f\|_{H^{2,p}_\nu}^2.
\eq
We next find that 
\[
\pa_t \pa_{\xi_j}\pa_i f + \xi \cdot \nabla \pa_{\xi_j}\pa_i f + \pa_i \lt(n^{m\theta}(v-\xi)\cdot \nabla_\xi \pa_{\xi_j}f \rt) +\pa_{ij}f - 4\pa_i(n^{m\theta} \pa_{\xi_j}f) = 0.
\]
Then we get
\begin{align*}
\begin{aligned}
&\frac12\frac{d}{dt}\intrr \nu_p(x,\xi) |\pa_{\xi_j}\pa_i f|^2\,dxd\xi \cr
&\quad = - \intrr \nu_p(x,\xi) \pa_{\xi_j}\pa_i f \cdot \lt( \xi \cdot \nabla \pa_{\xi_j}\pa_i f  + \pa_i (n^{m\theta} (v-\xi)\cdot \nabla_\xi \pa_{\xi_j}f)   \rt)dxd\xi\cr
&\qquad - \intrr \nu_p(x,\xi) \pa_{\xi_j}\pa_i f \cdot \lt(   \pa_{ij} f - 4\pa_i(n^{m\theta} \pa_{\xi_j}f) \rt)dxd\xi\cr
&\quad =:\sfJ_1 + \sfJ_2 + \sfJ_3 + \sfJ_4.
\end{aligned}
\end{align*}
We can easily estimate $\sfJ_1$, $\sfJ_3$ and $\sfJ_4$ as
\[\begin{aligned}
\sfJ_1&\le C\|\pa_{\xi_j}\pa_i f\|_{L_\nu^{2,p}}^2,\\
\sfJ_3&\le \|f\|_{H_\nu^{2,p}}^2,\\
\sfJ_4&\le  C\|n\|_{L^\infty}^{m\theta-1}\|\nabla n\|_{L^\infty}\|\pa_{\xi_j}\pa_if\|_{L_\nu^{2,p}}\|\pa_{\xi_j}f\|_{L_\nu^{2,p}} + C\|n\|_{L^\infty}^{m\theta}\|\pa_{\xi_j}\pa_i f\|_{L_\nu^{2,p}}^2 \\
&\le C\|n\|_{L^\infty}^{m\theta-1}(\|n\|_{L^\infty}+\|\nabla n\|_{H^2})\|f\|_{H_\nu^{2,p}}^2.
\end{aligned}\]
For $\sfJ_2$, we use the previous arguments for $\sfI_5$ to yield 
\[\begin{aligned}
\sfJ_2&= -\intrr \nu_p(x,\xi)\pa_{\xi_j}\pa_i f \cdot (\pa_i(n^{m\theta}) (v-\xi)\cdot\nabla_\xi \pa_{\xi_j} f)\,dxd\xi\\
&\quad -\intrr\nu_p(x,\xi)\pa_{\xi_j}\pa_i f \cdot(n^{m\theta}\pa_i v \cdot \nabla_\xi \pa_{\xi_j} f)\,dxd\xi\\
&\quad +\frac12\intrr \nabla_\xi\cdot(\nu_p(x,\xi)(v-\xi)) n^{m\theta} |\pa_{\xi_j}\pa_i f|^2\,dxd\xi\\
&\le C\|n\|_{L^\infty}^{m\theta-1}(\|n\|_{L^\infty} + \|\nabla n\|_{H^2})(1+\|v\|_{H^3})^2\|f\|_{H_\nu^{2,p}}^2.
\end{aligned}\]
Thus, we obtain
\bq\label{est_hf4}
\frac{d}{dt} \|\pa_{\xi_j}\pa_i f\|_{L^{2,p}_\nu}^2 \leq  C(1+\|n\|_{L^\infty})^{m\theta-1} (1+\|n\|_{L^\infty} +\|\nabla n\|_{H^2}) (1+\|v\|_{H^3} )^2 \|f\|_{H^{2,p}_\nu}^2.
\eq

We finally take the differential operator $\pa_{\xi_i \xi_j}$ to the kinetic part of \eqref{lin_sys} to estimate
\begin{align}\label{est_hf5}
\begin{aligned}
\frac12&\frac{d}{dt}\intrr \nu_p(x,\xi) |\pa_{\xi_i \xi_j}f|^2 dxd\xi \\
&= -\intrr \nu_p(x,\xi) \pa_{\xi_i \xi_j} f \cdot \lt( \xi \cdot \nabla \pa_{\xi_i \xi_j} f + \pa_i \pa_{\xi_j} f + \pa_{\xi_i}\pa_j f  \rt)dxd\xi\cr
&\quad + \intrr \nu_p(x,\xi) \pa_{\xi_i \xi_j} f \cdot \lt(  5 n^{m\theta} \pa_{\xi_i \xi_j} f  -  n^{m\theta}(v-\xi)\cdot \nabla_\xi \pa_{\xi_i \xi_j} f \rt)dxd\xi\cr
&\le C(1+\|n\|_{L^\infty})^{m\theta}\|f\|_{H^{2,p}_\nu}^2 + \frac12\intrr \nabla_\xi \cdot (\nu_p(x,\xi)(v-\xi)) n^{m\theta}|\pa_{\xi_i \xi_j}f|^2\,dxd\xi\\
&\le  C(1+\|n\|_{L^\infty})^{m\theta}(1+\|v\|_{L^\infty})^2\|f\|_{H^{2,p}_\nu}^2.
\end{aligned}
\end{align}
We now combine this with \eqref{est_hf2}, \eqref{est_hf3}, \eqref{est_hf4}, and \eqref{est_hf5} to have 
$$\begin{aligned}
\frac{d}{dt} \|f\|_{H^{2,p}_\nu}^2 &\leq C(1+\|n\|_{L^\infty})^{m\theta-4} (1+\|n\|_{L^\infty}+\|\nabla n\|_{H^2} )^4 (1+\|v\|_{H^3})^2 \|f\|_{H^{2,p}_\nu}^2\cr
&\leq C\epsilon_0^{m\theta}\epsilon_3^2 \|f\|_{H^{2,p}_\nu}^2,
\end{aligned}$$
thus applying Gr\"onwall's lemma yields
\[
 \|f(\cdot,\cdot,t)\|_{H^{2,p}_\nu}^2 \leq \epsilon_0^2 \exp\lt(C\epsilon_0^{m\theta}\epsilon_3^2 t\rt)
\]
for $0 \leq t \leq T_1$, due to $1 < \epsilon_i$ for $i=0,1,\dots,4$. We then choose $T_2 = \min\lt\{ T_1, (\epsilon_0^{m\theta}\epsilon_3^2)^{-1}\rt\}$ to obtain
\[
\|f(\cdot,\cdot,t)\|_{H^{2,p}_\nu} \leq C\epsilon_0 \quad \mbox{for} \quad 0 \leq t \leq T_2.
\]
For the $H^{1,p-2}_\nu$-estimate of $\pa_t f$, we find
$$\begin{aligned}
\|\pa_t f\|_{L^{2,p-2}_\nu} &\ls \|\xi \cdot \nabla f + \nabla_\xi \cdot ( n^{m\theta} (v-\xi)f)\|_{L^{2,p-2}_\nu} \cr
&\ls \|\nabla f\|_{L^{2,p}_\nu} + \|n^{m\theta} f + n^{m\theta}(v-\xi)\cdot\nabla_\xi f\|_{L^{2,p-2}_\nu}\cr
&\ls \|\nabla f\|_{L^{2,p}_\nu} + \|n\|_{L^\infty}^{m\theta}\|f\|_{L^{2,p-2}_\nu} + \|n\|_{L^\infty}^{m\theta} \|v\|_{L^6}\|g\|_{L^3} + \|n\|_{L^\infty}^{m\theta} \|\nabla_\xi f\|_{L^{2,p}_\nu}\cr
&\ls  \|\nabla f\|_{L^{2,p}_\nu} + \|n\|_{L^\infty}^{m\theta}\|f\|_{H^{2,p}_\nu}(1 + \|\nabla v\|_{L^2}) \cr
&\ls \epsilon_0^{m\theta+1} \epsilon_1
\end{aligned}$$
and
\begin{align*}
&\|\pa_t \nabla_{(x,\xi)}f\|_{L^{2,p-2}_\nu} \cr
&\quad \ls \lt\|\nabla_{(x,\xi)}\lt(\xi \cdot \nabla f + \nabla_\xi \cdot ( n^{m\theta} (v-\xi)f)\rt)\rt\|_{L^{2,p-2}_\nu}\\
&\quad \ls \|\nabla f\|_{L^{2,p-2}_\nu} + \|\nabla \nabla_{(x,\xi)} f\|_{L^{2,p}_\nu} +  \|n^{m\theta-1} (\nabla n) f\|_{L^{2,p-2}_\nu} + \|n^{m\theta} \nabla_{(x,\xi)}f\|_{L^{2,p-2}_\nu}\\
&\qquad + \|n^{m\theta-1}(\nabla n) (v-\xi)\cdot \nabla_\xi f\|_{L^{2,p-2}_\nu} + \|n^{m\theta}\nabla v \cdot \nabla_\xi f\|_{L^{2,p-2}_\nu} + \|n^{m\theta}(v-\xi)\nabla_\xi \nabla_{(x,\xi)}f \|_{L^{2,p-2}_\nu}\\
&\quad\ls \|f\|_{H^{2,p}_\nu} + \|n\|_{L^\infty}^{m\theta-1}\|\nabla n\|_{L^\infty} \|f\|_{L^{2,p-2}_\nu} + \|n\|_{L^\infty}^{m\theta} \|\nabla_{(x,\xi)}f\|_{L^{2,p-2}_\nu} \cr
&\qquad + \|n\|_{L^\infty}^{m\theta-1} \|\nabla n\|_{L^\infty}(1+ \|v\|_{L^\infty}) \|\nabla_\xi f\|_{L^{2,p}_\nu}+ \|n\|_{L^\infty}^{m\theta} \|\nabla v\|_{L^6} \|g\|_{L^3}  \\
&\qquad + \|n\|_{L^\infty}^{m\theta} (1+\|v\|_{L^\infty})\|\nabla_\xi \nabla_{(x,\xi)}f\|_{L^{2,p}_\nu}\\
&\quad\ls \epsilon_0^{m\theta+1}\epsilon_2
\end{align*}
for $0 \leq t \leq T_2$, where $g$ is given as in \eqref{new_g}. This completes the proof.

%
%
%
%
\section{$H^3$-estimate of $u$ \& $H^2$-estimate of $\pa_t u$}\label{app_c}

In this part, we provide the rest of the proof of Lemma \ref{lem_u}. \\

\noindent Taking $k=3$ in \eqref{u_H2}, we find 
\begin{align*}
&\frac12\frac{d}{dt}\|\pa^3 u\|_{L^2}^2 \cr
&\quad = -\int_{\R^3} (v \cdot \nabla \pa^3 u)\cdot \pa^3 u\,dx - \int_{\R^3} \lt( \pa^3 (v \cdot \nabla u)  - v \cdot \nabla \pa^3 u\rt)\cdot \pa^3 u\,dx\\
&\qquad -\frac{\gamma}{\gamma-1} \int_{\R^3} \pa^3 \nabla(n^{\theta(\gamma-1)}) \cdot \pa^3 u\,dx - \int_{\R^3} (n^2 +\eta^2)\pa^3 Lu \cdot \pa^3 u\,dx \\
&\qquad - \int_{\R^3} \lt(\pa^3 ((n^2 +\eta^2)Lu - (n^2 + \eta^2)\pa^3 Lu\rt)\cdot \pa^3 u\,dx -\frac{\delta}{\delta-1} \int_{\R^3} \nabla (n^2) \pa^3 \ml v \cdot \pa^3 u\,dx \\
&\qquad - \int_{\R^3} \lt(\pa^3 (\nabla(n^2) \ml v) - \nabla (n^2) \pa^3 \ml v \rt)\cdot \pa^3 u\,dx +\intrr \pa^2 \lt(n^{\theta(m-1)} (v-\xi)f \rt) \cdot \pa^4 u\,dxd\xi\\
&\quad =: \sum_{i=1}^8 \sfK_i,
\end{align*}
where we first easily get
\[
\sfK_1 \ls \epsilon_3 \|\nabla^3 u\|_{L^2}^2, \quad \sfK_2 \ls \epsilon_3 \|u\|_{H^3}^2.
\]
The term $\sfJ_3$ can be treated as
\begin{align*}
\sfK_3 &= \theta\gamma \int_{\R^3} \pa^2(n^{\theta(\gamma-1)-1} \pa n) \pa^3 (\nabla \cdot u)\,dx\\
&\ls \int_{\R^3} \lt| \lt(n^{\theta(\gamma-1)-3} (\pa n)^3  + n^{\theta(\gamma-1)-2} (\pa n) (\pa^2 n) + n^{\theta(\gamma-1)-1}\pa^3 n\rt) \pa^3 (\nabla \cdot u)  \rt|\,dx\\
&\ls\lt( \|n\|_{L^\infty}^{\theta(\gamma-1)-4} \|\nabla n\|_{L^\infty}^2 \|\nabla n\|_{L^2} + \|n\|_{L^\infty}^{\theta(\gamma-1)-3} \|\nabla n\|_{L^\infty} \|\nabla^2 n\|_{L^2} + \|n\|_{L^\infty}^{\theta(\gamma-1)-2} \|\nabla^3 n\|_{L^2}\rt) \|n \pa^3 (\nabla u)\|_{L^2}\\
&\ls \epsilon_0^{\theta(\gamma-1)-1} \|n \pa^3 (\nabla u)\|_{L^2}.
\end{align*}
Similarly to $H^2$-estimates, $\sfK_4$ can be estimated as
\begin{align*}
\sfK_4 &\le -\frac{3\alpha}{4}\int_{\R^3} (n^2 + \eta^2) |\pa^3 \nabla u|^2\,dx -\frac{3(\alpha+\beta)}{4}\int_{\R^3} (n^2 + \eta^2) (\pa^3 (\nabla \cdot u))^2\,dx\\
&\quad + (8\alpha+4\beta) \|\nabla n\|_{L^\infty}^2 \|\pa^3 u\|_{L^2}^2.
\end{align*}
For $\sfK_5$, we have
\begin{align*}
\sfK_5 &\ls \int_{\R^3} n |\nabla^3 n| |Lu|  |\pa^3 u|\,dx + \int_{\R^3} |\nabla n| |\nabla^2 n| |Lu| |\pa^3 u|\,dx\\
&\quad + \int_{\R^3} | \nabla n|^2 |L \pa u| |\pa^3 u|\,dx + \int_{\R^3} n |\nabla^2 n| |L\pa u| |\pa^3 u|\,dx\\
&\quad + \int_{\R^3} n |\nabla n| |L \pa^2 u| |\pa^3 u|\,dx\\
&\ls \|\nabla^3 n\|_{L^2} \|Lu\|_{L^6} \|n\nabla^3 u\|_{L^3} + \|\nabla n\|_{L^\infty}\|\nabla^2 n\|_{L^3} \|Lu\|_{L^6} \|\pa^3 u\|_{L^2}\\
&\quad + \|\nabla n\|_{L^\infty}^2 \|L\pa u\|_{L^2} \|\pa^3 u\|_{L^2} + \|\nabla^2 n\|_{L^3} \|n \pa^3 u\|_{L^6} \|L \pa u\|_{L^2}\\
&\quad + \|\nabla n\|_{L^\infty} \|n L \pa^2 u\|_{L^2} \|\pa^3 u\|_{L^2}\\
&\ls  \epsilon_0 \|\nabla^3 u\|_{L^2} \lt( \|\nabla n\|_{L^\infty}\|\nabla^3 u\|_{L^2} + \|n \nabla^4 u\|_{L^2}\rt) + \epsilon_0^2 \|\nabla^3 u\|_{L^2}^2\\
&\ls \epsilon_0^2 \|\nabla^3 u\|_{L^2}^2 + \epsilon_0 \|\nabla^3 u\|_{L^2} \|n \nabla^4 u\|_{L^2}.
\end{align*}
The term $\sfK_6$ can be treated as
\begin{align*}
\sfK_6 &\ls \int_{\R^3} \lt|\pa (n\nabla n) \cdot \pa^2\ml v \cdot \pa^3 u\rt|\,dx + \int_{\R^3} n\lt| \nabla n \cdot \pa^2 \ml v \cdot \pa^4 u\rt|\,dx\\
&\ls \|\nabla n\|_{L^\infty}^2 \|\nabla^3 v\|_{L^2} \|\nabla^3 u\|_{L^2} + \|\nabla^2 n\|_{L^3} \|\nabla^{2}\ml v\|_{L^2} \|n\nabla^3 u\|_{L^6}\\
&\quad + \|\nabla n\|_{L^\infty} \|\nabla^2 \ml v\|_{L^2} \|n\nabla^4 u\|_{L^2}\\
&\ls \epsilon_0^2\epsilon_3 \|\nabla^3 u\|_{L^2} + \epsilon_0\epsilon_3 \|n\nabla^4 u\|_{L^2}.
\end{align*}
For $\sfK_7$, we deduce
\begin{align*}
\sfK_7 &\ls \lt|\int_{\R^3} \pa^3 (n \nabla n)\cdot \ml v \cdot\pa^3 u\,dx\rt|+ \sum_{m=1}^{2} \int_{\R^3} \lt| \pa^m (n\nabla n) \cdot \pa^{3-m} \ml v  \cdot \pa^3 u \rt|\,dx\\
&\ls  \lt|\int_{\R^3} (n \nabla \pa^3 n)\cdot \ml v \cdot\pa^3 u\,dx\rt| + \int_{\R^3} |\nabla n| |\nabla^3 n| |\ml v| |\pa^3 u|\,dx\\
&\quad+ \int_{\R^3} |\nabla^2 n|^2 |\ml v| |\pa^3 u|\,dx + \int_{\R^3} |\nabla n|^2 |\pa^2 \ml v| |\pa^3 u|\,dx \\
&\quad + \int_{\R^3} n |\nabla^2 n| |\pa^2 \ml v| |\pa^3 u|\,dx  +\int_{\R^3} |\nabla n| |\nabla^2 n| |\pa \ml v| |\pa^3 u|\,dx \\
&\quad + \int_{\R^3} n |\nabla^3 n| |\pa \ml v| |\pa^3 u|\,dx\\
&\ls  \int_{\R^3} n |\nabla^3 n| |\ml v| |\pa^4 u|\,dx + \int_{\R^3} |\nabla n| |\nabla^3 n| |\ml v| |\pa^3 u|\,dx \\
&\quad + \int_{\R^3} |\nabla^2 n|^2 |\ml v| |\pa^3 u|\,dx  + \int_{\R^3} |\nabla n|^2 |\pa^2 \ml v| |\pa^3 u|\,dx \\
&\quad+ \int_{\R^3} n |\nabla^2 n| |\pa^2 \ml v| |\pa^3 u|\,dx +\int_{\R^3} |\nabla n| |\nabla^2 n| |\pa \ml v| |\pa^3 u|\,dx \\
&\quad + \int_{\R^3} n |\nabla^3 n| |\pa \ml v| |\pa^3 u|\,dx\\
&\ls \|\nabla^3 n\|_{L^2}\|\ml v\|_{L^\infty}\|n\nabla^4 u\|_{L^2}+ \|\nabla n\|_{L^\infty}\|\nabla^3 n\|_{L^2}\|\ml v\|_{L^\infty}\|\pa^3 u\|_{L^2}\\
&\quad + \|\nabla^2 n\|_{L^3} \|\nabla^2 n\|_{L^6} \|\ml v\|_{L^\infty}\|\pa^3 u\|_{L^2} + \|\nabla n\|_{L^\infty}^2 \|\nabla^3 v\|_{L^2}\|\pa^3 u\|_{L^2}\\
&\quad + \|\nabla^2 n\|_{L^3}\|\nabla^3 v\|_{L^2} \|n \pa^3 u\|_{L^6} + \|\nabla n\|_{L^\infty}\|\nabla^2 n\|_{L^3} \|\pa \ml v\|_{L^6} \|\pa^3 u\|_{L^2}\\
&\quad + \|\nabla^3 n\|_{L^2}\|\pa\ml v\|_{L^3}\|n\pa^3 u\|_{L^6}\\
&\ls \epsilon_0\epsilon_3 \|n\nabla^4 u\|_{L^2} + \epsilon_0^2 \epsilon_3 \|\nabla^3 u\|_{L^2}.
\end{align*}
We estimate $\sfK_8$ as
$$\begin{aligned}
\sfK_8&\ls \intrr n^{\theta(m-1)-2} |\nabla n|^2 |v-\xi| f |\pa^4 u|\,dxd\xi + \intrr n^{\theta(m-1)-1}|\nabla^2 n| |v-\xi| f |\pa^4 u|\,dxd\xi\\
&\quad + \intrr n^{\theta(m-1)-1} |\nabla n| |\nabla v| f |\pa^4 u|\,dxd\xi + \intrr n^{\theta(m-1)-1} |\nabla n| |v-\xi| |\nabla f| |\pa^4 u|\,dxd\xi\\
&\quad + \intrr n^{\theta(m-1)}\lt( |\nabla^2 v|  f + |\nabla v| |\nabla f| + |v-\xi| |\nabla^2 f| \rt) |\pa^4 u|\,dxd\xi\\
&\ls \|n\|_{L^\infty}^{\theta(m-1)-3} \|\nabla n\|_{L^\infty}^2 \lt(\|v\|_{L^\infty} \lt\|\intr f\,d\xi \rt\|_{L^2} + \lt\|\intr |\xi|f\,d\xi \rt\|_{L^2} \rt)  \|n\nabla^4 u\|_{L^2} \cr
&\quad + \|n\|_{L^\infty}^{\theta(m-1)-2}\|\nabla^2 n\|_{L^6} \lt(\|v\|_{L^\infty}\lt\|\intr f\,d\xi \rt\|_{L^3} + \lt\|\intr |\xi| f\,d\xi \rt\|_{L^3}\rt)  \|n \nabla^4 u\|_{L^2}\\
&\quad + \|n\|_{L^\infty}^{\theta(m-1)-2} \|\nabla n\|_{L^\infty}\|\nabla v\|_{L^\infty} \lt\|\intr f\,d\xi \rt\|_{L^2} \|n\nabla^4 u\|_{L^2} \cr
&\quad + \|n\|_{L^\infty}^{\theta(m-1)-2} \|\nabla n\|_{L^\infty} \lt(\|v\|_{L^\infty} \lt\|\intr |\nabla f|\,d\xi \rt\|_{L^2} + \lt\|\intr |\xi||\nabla f|\,d\xi \rt\|_{L^2} \rt) \|n\nabla^4 u\|_{L^2} \\
&\quad + \|n\|_{L^\infty}^{\theta(m-1)-1} \lt(\|\nabla^2 v\|_{L^6}\lt\|\intr f\,d\xi \rt\|_{L^3} + \|\nabla v\|_{L^\infty}  \lt\|\intr |\nabla f|\,d\xi \rt\|_{L^2}    \rt) \|n\nabla^4 u\|_{L^2}\\
&\quad + \|n\|_{L^\infty}^{\theta(m-1)-1}   \lt(\|v\|_{L^\infty} \lt\|\intr |\nabla^2 f|\,d\xi \rt\|_{L^2} + \lt\|\intr |\xi||\nabla^2 f|\,d\xi \rt\|_{L^2}   \rt) \|n\nabla^4 u\|_{L^2}\\
&\ls \epsilon_0^{\theta(m-1)}\epsilon_3 \|n\nabla^4 u\|_{L^2}.
\end{aligned}$$

We combine all the estimates for $\sfK_i$'s and use Young's inequality to yield
\begin{align*}
\frac{d}{dt}&\|\nabla^3 u\|_{L^2}^2 + \alpha \int_{\R^3} (n^2 + \eta^2) |\nabla^{4} u|^2\,dx + (\alpha+\beta) \int_{\R^3} (n^2 + \eta^2) |\nabla^3 (\nabla \cdot u)|^2\,dx\\
&\le C\lt(\epsilon_0^{2\theta(\zeta-1)}\epsilon_3^2\|u\|_{H^3}^2 + \epsilon_0^{2\theta(\zeta-1)}\epsilon_3^2\rt)\\
&\le C\lt(\epsilon_0^{2\theta(\zeta-1)}\epsilon_3^2\|\nabla^3 u\|_{L^2}^2 + \epsilon_0^{2\theta(\zeta-1)+2}\epsilon_3^2\rt),
\end{align*}
for $0 \le t \le T_2$. Thus we also obtain
\[
\|\nabla^3 u(t)\|_{L^2}^2 + \int_0^t \|\sqrt{n^2(s)+\eta^2} \nabla^4 u(s)\|_{L^2}^2\,ds \le C\epsilon_0^2, \quad 0 \le t \le T_3.
\]

For the time derivative estimate, we  have from \eqref{lin_sys} that
\begin{align*}
\|\pa_t u\|_{L^2}^2 &\ls \lt\| v \cdot \nabla u + \theta \gamma n^{\theta(\gamma-1)-1} \nabla n + (n^2 + \eta^2) Lu + \frac{2\delta}{\delta-1} n \nabla n \cdot \ml v + n^{\theta(m-1)}\int_{\R^3} (v-\xi)f\,d\xi \rt\|_{L^2}\\
&\ls \|v\|_{L^6}\|\nabla u\|_{L^3} + \|n\|_{L^\infty}^{\theta(\gamma-1)-1}\|\nabla n\|_{L^2} + (1+\|n\|_{L^\infty}^2)\|\nabla^2 u\|_{L^2} + \|n\|_{L^\infty}\|\nabla n\|_{L^\infty}\|\ml v\|_{L^2} \\
&\quad + \|n\|_{L^\infty}^{\theta(m-1)}\lt(\|v\|_{L^6}\lt\|\intr f\,d\xi \rt\|_{L^3} + \lt\|\intr |\xi| f\,d\xi \rt\|_{L^3}\rt) \\
&\ls \epsilon_0 \epsilon_1 + \epsilon_0^{\theta(\gamma-1)} + \epsilon_0^3 + \epsilon_0^2\epsilon_1 + \epsilon_0^{\theta(m-1)+1} \epsilon_1\\
&\ls \epsilon_0^{\theta(\zeta-1)+1}\epsilon_1,
\end{align*}
\begin{align*}
&\| \pa_t \nabla u\|_{L^2}\cr
&\quad \ls \| \nabla(v \cdot \nabla u)\|_{L^2} + \|\nabla (n^{\theta(\gamma-1)-1}\nabla n)\|_{L^2} + \|n \nabla n \cdot Lu\|_{L^2} + \|(n^2+\eta^2) L\nabla u\|_{L^2}\\
&\qquad + \| \nabla (n\nabla n \ml v)\|_{L^2} + \lt\| \nabla\lt(n^{\theta(m-1)}\int_{\R^3}(v-\xi)f\,d\xi\rt)\rt\|_{L^2}\\
&\quad \ls \|v\|_{H^2}\|u\|_{H^3} + \|n\|_{L^\infty}^{\theta(\gamma-1)-2}\|\nabla n\|_{L^4}^2 + \|n\|_{L^\infty}^{\theta(\gamma-1)-1}\|\nabla^2 n\|_{L^2} + \|n\nabla n\|_{L^\infty} \|\nabla^2 u\|_{L^2}\\
&\qquad + (1+\|n\|_{L^\infty}^2)\|\nabla^3 u\|_{L^2} + \|\nabla n\|_{L^\infty}^2\|\ml v\|_{L^2} + \|n\|_{L^\infty}\|\nabla^2 n\|_{L^6}\|\ml v\|_{L^3} + \|n\nabla n\|_{L^\infty}\|\nabla^2 v\|_{L^2}\\
&\qquad + \|n\|_{L^\infty}^{\theta(m-1)-1}\lt(\|\nabla n\|_{L^\infty}\lt(\|v\|_{L^6}\lt\|\intr f\,d\xi \rt\|_{L^3} + \lt\|\intr |\xi| f\,d\xi \rt\|_{L^3}\rt) + \|n\|_{L^\infty}\|\nabla v\|_{L^6}\lt\|\intr f\,d\xi \rt\|_{L^3}\rt) \cr
&\qquad  + \|n\|_{L^\infty}^{\theta(m-1)}\lt(\|v\|_{L^6}\lt\|\intr |\nabla f|\,d\xi \rt\|_{L^3} + \lt\|\intr |\xi| |\nabla f|\,d\xi \rt\|_{L^3}\rt)\\
&\quad \ls  \epsilon_0 \epsilon_2 + \epsilon_0^{\theta(\gamma-1)} +\epsilon_0^3 + \epsilon_0^2\epsilon_2 + \epsilon_0^{\theta(m-1)+1}\epsilon_2\\
&\quad \ls \epsilon_0^{\theta(\zeta-1)+1}\epsilon_2,
\end{align*}
and
\begin{align*}
&\|\pa_t \nabla^2 u\|_{L^2} \cr
&\quad \ls\| \nabla^2(v \cdot \nabla u)\|_{L^2} + \|\nabla^2 (n^{\theta(\gamma-1)-1}\nabla n)\|_{L^2} + \|\nabla (n\nabla n)\cdot Lu\|_{L^2} + \|n \nabla n \cdot L\nabla u\|_{L^2} \\
&\qquad + \|(n^2+\eta^2) L\nabla^2 u\|_{L^2} + \| \nabla^2 (n\nabla n \ml v)\|_{L^2} + \lt\| \nabla^2\lt(n^{\theta(m-1)}\int_{\R^3}(v-\xi)f\,d\xi\rt)\rt\|_{L^2}\\
&\quad \ls \|v\|_{H^3}\|u\|_{H^3} + \|n\|_{L^\infty}^{\theta(\gamma-1)-3}\|\nabla n\|_{L^6}^3 + \|n\|_{L^\infty}^{\theta(\gamma-1)-2}\|\nabla n\|_{L^\infty}\|\nabla^2 n\|_{L^2} + \|n\|_{L^\infty}^{\theta(\gamma-1)-1}\|\nabla^3 n\|_{L^2}\\
&\qquad + \|\nabla n\|_{L^\infty}^2\|\nabla^2 u\|_{L^2} + \|n\|_{L^\infty}\|\nabla^2 n\|_{L^3} \|\nabla^2 u\|_{L^6} + \|n\nabla n\|_{L^\infty}\|\nabla^3 u\|_{L^2}\\
&\qquad + (1+\|n\|_{L^\infty})\|\sqrt{(n^2+\eta^2)} \nabla^4 u\|_{L^2} +  \|\nabla^2(n\nabla n)\|_{L^2}\|\ml v\|_{L^\infty} + \|n\nabla n\|_{L^\infty}\|\nabla^2 \ml v\|_{L^2}\\
&\qquad + \|n\|_{L^\infty}^{\theta(m-1)-2}\|\nabla n\|_{L^\infty}^2 \lt(\|v\|_{L^\infty}\lt\|\intr f\,d\xi \rt\|_{L^2} + \lt\|\intr |\xi| f\,d\xi \rt\|_{L^2} \rt)\\
&\qquad + \|n\|_{L^\infty}^{\theta(m-1)-1} \|\nabla^2 n\|_{L^6} \lt(\|v\|_{L^\infty}\lt\|\intr f\,d\xi \rt\|_{L^3} + \lt\|\intr |\xi| f\,d\xi \rt\|_{L^3}\rt)  \\
&\qquad + \|n\|_{L^\infty}^{\theta(m-1)-1}\|\nabla n\|_{L^\infty}\lt(\|\nabla v\|_{L^\infty}\lt\|\intr f\,d\xi\rt\|_{L^2} + \lt(\|v\|_{L^\infty}\lt\|\intr |\nabla f|\,d\xi \rt\|_{L^3} + \lt\|\intr |\xi| |\nabla f|\,d\xi \rt\|_{L^3}\rt)  \rt)\\
&\qquad + \|n\|_{L^\infty}^{\theta(m-1)}\lt(\|\nabla^2 v\|_{L^6}\lt\|\intr f\,d\xi \rt\|_{L^3} + \|\nabla v\|_{L^\infty}\lt\|\intr |\nabla f|\,d\xi\rt\|_{L^2}\rt)\cr
&\qquad  +\|n\|_{L^\infty}^{\theta(m-1)} \lt(\|v\|_{L^\infty}\lt\|\intr |\nabla^2 f|\,d\xi \rt\|_{L^2} + \lt\|\intr |\xi| |\nabla^2 f|\,d\xi \rt\|_{L^2}\rt)\\
&\quad \ls \epsilon_0 \epsilon_3 + \epsilon_0^{\theta(\gamma-1)} + \epsilon_0^3 + \epsilon_0 \|\sqrt{(n^2+\eta^2)} \nabla^4 u\|_{L^2} + \epsilon_0^2\epsilon_3 + \epsilon_0^{\theta(m-1)+1}\epsilon_3\\
&\quad \ls \epsilon_0^{\theta(\zeta-1)+1}\epsilon_3 + \epsilon_0 \|\sqrt{(n^2+\eta^2)} \nabla^4 u\|_{L^2}
\end{align*}
for $0 \le t \le T_3$. Hence we conclude 
\[
\int_0^t \|\pa_t \nabla^2 u(s)\|_{L^2}^2\,ds \ls T_3\epsilon_0^{2\theta(\zeta-1)+2}\epsilon_3^2 + \epsilon_0 \int_0^{T_3}  \|\sqrt{(n^2(s)+\eta^2)} \nabla^4 u(s)\|_{L^2}^2\,ds\ls \epsilon_0^3,
\]
for $0 \le t \le T_3$.
%
%
%
%

\end{document}